\documentclass{article}
\usepackage{amsmath}
\usepackage{amsfonts}
\usepackage{amsthm}
\usepackage{amssymb}
\usepackage{mathabx}
\usepackage{enumerate}
\usepackage{ytableau}

\newtheorem{thm}{Theorem}[section]
\newtheorem{prop}[thm]{Proposition}
\newtheorem{lem}[thm]{Lemma}
\newtheorem{cor}[thm]{Corollary}
\newtheorem{conj}[thm]{Conjecture}

\theoremstyle{definition}
\newtheorem{defn}[thm]{Definition}
\newtheorem{rmk}[thm]{Remark}
\newtheorem{ex}[thm]{Example}

\begin{document}

\title{Stable Higher Specht Polynomials and Representations of Infinite Symmetric Groups}

\author{Shaul Zemel}

\maketitle


\section*{Introduction}

Consider the space $\mathbb{Q}[\mathbf{x}_{n}]_{1}$ of linear polynomials in $n$ variables. It is acted upon by the symmetric group $S_{n}$ permuting the variables, and decomposes as the direct sum of a trivial representation on the space spanned by the sum of all the variables, and the irreducible standard representation on the subspace $\mathbb{Q}[\mathbf{x}_{n}]_{1}^{0}$ defined by the vanishing of the sum of the coefficients.

When considering the analogue $\mathbb{Q}[\mathbf{x}_{\infty}]_{1}$ of linear polynomials in infinitely many variables, it comes endowed with an action of the symmetric group $S_{\mathbb{N}}$. The similarly defined subspace $\mathbb{Q}[\mathbf{x}_{\infty}]_{1}^{0}$ is a sub-representation, which can be shown to still be irreducible, and it has co-dimension 1 in the full representation, with a trivial quotient. However, there is no symmetric direct summand completing it, as the sum of all the variables is no longer a polynomial, but rather an element of the ring $\Lambda$ of symmetric functions. Hence the natural limit $\mathbb{Q}[\mathbf{x}_{\infty}]_{1}$ of $\mathbb{Q}[\mathbf{x}_{n}]_{1}$ as $n\to\infty$ is no longer completely reducible, with one irreducible component remaining, and another one leaves the space in the limit.

The goal of this paper, which is based on the previous works \cite{[Z2]} and \cite{[Z3]}, is to show how the theory of stable higher Specht polynomials, and their generalizations, explain this phenomenon and related ones, in any degree.

\medskip

We thus introduce some notation and terminology. The seminal work \cite{[B]} considered the ring $R_{n}$ of co-invariants of the action of the symmetric group $S_{n}$ on the ring $\mathbb{Q}[\mathbf{x}_{n}]$ of polynomials in $n$ variables from various perspectives, including determining its representation type as the regular one. The classical construction of Specht polynomials yields a single copy of each Specht module in $R_{n}$, and \cite{[ATY]} defined higher Specht polynomials in order to produce the remaining copies.

The paper \cite{[HRS]} defined more general quotients $R_{n,k}$ and $R_{n,k,s}$ that are related to appropriate actions of $S_{n}$, where some associated algebraic varieties we considered in \cite{[PR]}, and there are some relations with the Delta conjecture---see \cite{[HRW]}. In \cite{[Z2]} we showed how to decompose $R_{n,k}$ according to the orbits of the action of $S_{n}$, and how appropriate normalizations of the higher Specht polynomials relates those with $n$ variables to those with $n+1$ variables in a natural way. The sequel \cite{[Z3]} did the same for $R_{n,k,s}$, constructed generalized higher Specht polynomials, and used both types of these polynomials for obtaining decompositions of the subspace $\mathbb{Q}[\mathbf{x}_{\infty}]_{d}$ of $\mathbb{Q}[\mathbf{x}_{\infty}]$ that is determined by the homogeneity degree $d$.

In more detail, assume that $T$ is a standard Young tableau of some shape of size $n$, and $M$ is a semi-standard Young tableau of the same shape. Then \cite{[Z3]} defines (by extending \cite{[ATY]}, \cite{[Z2]}, and others), the generalized higher Specht polynomial $F_{M,T}\in\mathbb{Q}[\mathbf{x}_{n}]$, normalized in such a way that the ``leading monomial'' comes with the coefficient 1 (see Definition \ref{Spechtdef} and Theorem \ref{repsSpecht} below for the precise construction). Then there is a standard Young tableau $\iota T$ of a shape of size $n+1$, and a semi-standard Young tableau $\hat{\iota}M$ of that shape, such that the generalized higher Specht polynomial $F_{\hat{\iota}M,\iota T}\in\mathbb{Q}[\mathbf{x}_{n+1}]$ produces $F_{M,T}$ under the substitution $x_{n+1}=0$, like in the theory of symmetric functions (see Definition \ref{iotadef} and Proposition \ref{forstab} below).

\medskip

As in the classical construction from \cite{[M]} and \cite{[Pe]}, these generalized higher Specht polynomials span irreducible representations of $S_{n}$ inside $\mathbb{Q}[\mathbf{x}_{n}]$, realizing the associated Specht module. The representations obtained from the usual higher Specht polynomials by multiplication by appropriate symmetric functions produce the higher Specht bases for $R_{n,k}$ and $R_{n,k,s}$ from \cite{[GR]}, which were decomposed into direct sums $R_{n,I}$ and $R_{n,I}^{\mathrm{hom}}$ for (multi-)sets $I$. The map taking $F_{M,T}$ to $F_{\hat{\iota}M,\iota T}$ (and a symmetric function in $n$ variables to the same symmetric function but now in $n+1$ variables) relate these representations with the ones in which the index $n$ is increased by 1.

We recall from \cite{[CF]}, \cite{[CEF]}, \cite{[Fa]}, and \cite{[SS]} the notion of a stable family of representations (or FI-modules), or the centrally stable representations from \cite{[Pu]}. Remark 4.24 of \cite{[Z2]} and Remark 3.28 of \cite{[Z3]} explain how the representations constructed in these references have this useful property. In this paper we go one step further, and construct actual limits of these (centrally) stable representations, which are representations of the infinite symmetric group $S_{\mathbb{N}}$, permuting all the positive integers.

\medskip

The natural domain where these limit representations live is the ring $\tilde{\Lambda}$ of \emph{eventually symmetric functions}. If $S_{\mathbb{N}}^{(n)} \subseteq S_{\mathbb{N}}$ is the subgroup fixing the numbers between 1 and $n$, then a power series from $\mathbb{Q}\ldbrack\mathbf{x}_{\infty}\rdbrack$ is eventually symmetric if its is fixed by $S_{\mathbb{N}}^{(n)}$ for some $n$ and its degree is bounded (see Definition \ref{evsym} below). On $\tilde{\Lambda}$ the representations of $S_{\mathbb{N}}$ and of the union $S_{\infty}$ of the finite symmetric groups there are the same, and it is generated by $\mathbb{Q}[\mathbf{x}_{\infty}]$ and $\Lambda$ in an algebraically independent manner (see Proposition \ref{samereps} below). 

For constructing limits, we define infinite Ferrers diagrams, infinite standard Young tableau, and infinite semi-standard Young tableau (see Definition \ref{tabinf} below), and show in Lemma \ref{limtab} how to obtain these by limits of finite objects. If $\hat{T}$ and $\hat{M}$ are obtained in this way from $T$ and $M$ from above, then the stable generalized higher Specht polynomial $F_{\hat{M},\hat{T}}\in\tilde{\Lambda}$ from Definition \ref{infSpecht} below is the unique one (in the appropriate sense) yielding $F_{M,T}$ after one substitutes $x_{m}=0$ for every $m>n$ (see Proposition \ref{limSpecht} below). We prove in Theorem \ref{repsinf} below that these objects span irreducible representations of $S_{\mathbb{N}}$ and of $S_{\infty}$, with an isomorphism type of such representations associated to every infinite Ferrers diagram, and consider various direct sums of such representations, generalizing the constructions from \cite{[Z2]} and \cite{[Z3]} in the finite case.

Note that our representations of $S_{\infty}$ are, in some sense, orthogonal to the theory of Thoma parameters from \cite{[T]}, \cite{[BO]}, and other (they are, by definition, \emph{tame} in the terminology of, e.g., \cite{[O]} and \cite{[MO]})---see Remark \ref{Thoma} below.

\medskip

As mentioned above, the representation $\mathbb{Q}[\mathbf{x}_{\infty}]_{d}$ obtained in the infinite case, and thus the homogeneous component $\tilde{\Lambda}_{d}$ of $\tilde{\Lambda}$ containing it, are no longer completely reducible. We show in Theorem \ref{compred} below how our direct sums produce various decompositions of the maximal completely reducible sub-representations $\mathbb{Q}[\mathbf{x}_{\infty}]_{d}^{0}$ and $\tilde{\Lambda}_{d}^{0}$ of these representations. The analysis is based on a parameter $f$ attached to each irreducible component (see Definition \ref{repswithmon} below), where the case $f=0$ is the one for which $F_{\hat{\iota}M,\iota T}=F_{M,T}$ and hence $F_{\hat{M},\hat{T}}$ is a polynomial (rather than a more general power series from $\tilde{\Lambda}$)---see Lemma \ref{sameiota} below. This parameter produces a natural filtration on the spaces $\mathbb{Q}[\mathbf{x}_{n}]_{d}$ for finite $n$.

The final goal of this paper is to define a filtrations on $\mathbb{Q}[\mathbf{x}_{\infty}]_{d}$ and $\tilde{\Lambda}_{d}$, the isomorphism types of whose quotients are described explicitly, such that the graded part at each step is the maximal completely reducible sub-representation of the corresponding quotient. In order to do this, we investigate how polynomials in $\mathbb{Q}[\mathbf{x}_{n}]_{d}$ associated with a parameter $f$ behave when viewed as elements of $\mathbb{Q}[\mathbf{x}_{n+1}]_{d}$. Based on some examples, we conjecture that this parameter $f$ is ``independent of $n$'' (see Lemma \ref{conjeq} below, that is based on Conjecture \ref{polsincn} below), though the structure in terms of the representations themselves is not very straightforward---see Remark \ref{notiota} below.

Assuming the veracity of our conjecture, we can define the filtrations in Definitions \ref{Qxinfdf} and \ref{Lambdafilt} below, and establish their properties in Proposition \ref{filtprop} and Theorems \ref{filtQxinf} and \ref{filtrations} below. The resulting forms of their semi-simplifications are given in Corollaries \ref{semisimp} and \ref{submulti} below. 

\medskip

This paper is divided into 4 sections. Section \ref{TPR} reviews the basic notions and constructions from \cite{[Z2]} and \cite{[Z3]}, for introducing the notation. Section \ref{EvSymFunc} defines eventually symmetric functions, the infinite diagrams and tableaux, and the stable versions of the (generalized) higher Specht polynomials and the representations they span. Section \ref{MaxCompRed} introduces the parameter $f$, presents the various decompositions of the maximal completely reducible sub-representation $\tilde{\Lambda}_{d}$, and proves that it is indeed such. Finally, in Section \ref{FullRep} we pose our conjecture, and show how it yields the desired filtrations.

\medskip

I am grateful to B. Sagan, B. Rhoades, M. Gillespie, and S. van Willigenburg, for their interest, for several encouraging conversations on this subject, and for helpful comments. Special thanks are due to G. I. Olshanski for several useful comments about the relation with the general theory of $S_{\infty}$ and for referring me to \cite{[MO]}. I am indebted, in particular, D. Grinberg for a detailed reading of previous versions, for introducing me to several references, and for numerous suggestions which drastically improved the presentation of this paper.

\section{Tableaux, Polynomials, and Representations \label{TPR}}

We begin by recalling the notation from \cite{[Z2]} and \cite{[Z3]}. We write $\lambda \vdash n$ to indicate that $\lambda$ is a partition of $n$, which we identify with its Ferrers diagram, and we write $\lambda=\operatorname{sh}(T)$ for expressing that the shape of the tableau $T$ is (the diagram of) $\lambda$. We will use $\mathbb{N}_{n}$ for the set of integers between 1 and $n$, and write, as usual, $\operatorname{SYT}(\lambda)$ for the set of standard Young tableau of shape $\lambda$ (the entries of which are, as always, the elements of $\mathbb{N}_{n}$).

We recall that a \emph{weak composition} of $n$, of length $k$, is a sequence $\{\alpha_{h}\}_{h=0}^{k-1}$ of non-negative integers summing to $n$, a situation that we denote by $\alpha\vDash_{w}n$. If all the entries of $\alpha$ are positive then we call $\alpha$ simply a \emph{composition} of $n$ and write $\alpha \vDash n$, and in any case we denote by $\ell(\alpha)$ the length $k$ of $\alpha$ (as we do for partitions, which are a special case of compositions).

We will also adopt the following terminology and notation from \cite{[Z3]}, for which we recall that the entries of semi-standard Young tableaux are non-negative integers, including 0, in our normalization.
\begin{defn}
Take a partition $\lambda \vdash n$, and an integer $d\geq0$.
\begin{enumerate}[$(i)$]
\item The set of semi-standard Young tableaux of shape $\lambda$ is $\operatorname{SSYT}(\lambda)$.
\item If $M\in\operatorname{SSYT}(\lambda)$ then $\Sigma(M)$ stands for the sum of the entries of $M$.
\item We let $\operatorname{SSYT}_{d}(\lambda)$ denote the set of $M\in\operatorname{SSYT}(\lambda)$ for which $\Sigma(M)=d$.
\item We define the \emph{content} of $M$ to be entries appearing in $M$, written as a non-decreasing sequence of length $n$ (we can also view the content as a multi-set, as in Definition \ref{multisets} below).
\item If $\mu$ is some content, then the set of $M\in\operatorname{SSYT}(\lambda)$ whose content is $\mu$ will be written as $\operatorname{SSYT}_{\mu}(\lambda)$.
\end{enumerate} \label{defSSYT}
\end{defn}
The set $\operatorname{SSYT}_{d}(\lambda)$ from Definition \ref{defSSYT} is thus the disjoint union of the sets $\operatorname{SSYT}_{\mu}(\lambda)$, where $\mu$ runs over the contents having sum $d$.

We will use the notation from \cite{[CZ]}, \cite{[Z2]}, and \cite{[Z3]}, where if $T\in\operatorname{SYT}(\lambda)$ and $1 \leq i \leq n$ then $v_{T}(i)$ is the box containing $i$ in $T$, whose column is denoted by $C_{T}(i)$, and the row is $R_{T}(i)$ (this makes sense more generally for any tableau $T$ of content $\mathbb{N}_{n}$). Then the set $\operatorname{Dsi}(T)$ of \emph{descending indices}, or \emph{descents}, of $T$ consists of those $i\in\mathbb{N}_{n-1}$ for which $R_{T}(i+1)>R_{T}(i)$ (and then $C_{T}(i+1) \leq C_{T}(i)$). We also recall our convention for working with multi-sets and destandardizations.
\begin{defn}
Fix some integer $n\geq1$.
\begin{enumerate}[$(i)$]
\item A \emph{multi-set} $J$ inside $\mathbb{N}_{n}\cup\{0\}$ is determined by giving a non-negative, finite, integral multiplicity to each element of $\mathbb{N}_{n}\cup\{0\}$.
\item The size of $J$ is the sum of these multiplicities.
\item If $J$ is a multi-set of size $k-1$, ordered as an increasing sequence and completed by $j_{0}=0$ and $j_{k}=n$, then the weak composition $\operatorname{comp}_{n}J\vDash_{w}n$ associated with $J$ is $\{\alpha_{h}\}_{h=0}^{k-1}$ where $\alpha_{h}:=j_{h+1}-j_{h}$.
\item The \emph{content represented by $\alpha$}, for $\alpha$ as above, is the one where each number $0 \leq h<k$ shows up $\alpha_{h}$ times. We may write it as a non-decreasing sequence of non-negative integers that are smaller than $k$, of length $n$.
\item The set $\hat{J}$ associated with $J$ is the subset of $\mathbb{N}_{n-1}$ whose elements are those that appear in $J$ with some positive multiplicity. We write $\hat{k}-1$ for its size.
\item Given a subset $D\subseteq\mathbb{N}_{n-1}$, we write $D \subseteq J$ and say that $D$ is contained in $J$ in case $D\subseteq\hat{J}$. The complement $J \setminus D$ is then the multi-set in which the multiplicity of an element $i\in\mathbb{N}_{n-1}$ is one less than that in $J$ when $i \in D$, and the same as in $J$ if $i \not\in D$. This applies, in particular, for $J\setminus\hat{J}$.
\item For $T\in\operatorname{SYT}(\lambda)$ for which $\operatorname{Dsi}(T) \subseteq J$, we denote by $\operatorname{ct}_{J}(T)$ the tableau of shape $\lambda$ in which we replace every $i\in\mathbb{N}_{n}$ by the number showing up at the $i$th location in the sequence describing the content represented by $\operatorname{comp}_{n}J$.
\item If $J=\operatorname{Dsi}(T)$ then $\operatorname{ct}_{J}(T)$ is denoted simply by $\operatorname{ct}(T)$.
\item A tableau that is of the form $\operatorname{ct}(T)$ for some $T\in\operatorname{SYT}(\lambda)$ is called a \emph{cocharge tableau} of shape $\lambda$. The set of those tableaux is denoted by $\operatorname{CCT}(\lambda)$.
\end{enumerate} \label{multisets}
\end{defn}
For more on compositions, see, e.g., pages 17--18 of \cite{[St1]}.

\begin{ex}
With $\lambda=221\vdash5$ and the multi-set $J:=\{0,2,2,3,5\}$ inside $\mathbb{N}_{5}\cup\{0\}$, we take
\[T:=\begin{ytableau} 1 & 2 \\ 3 & 5 \\ 4 \end{ytableau},\mathrm{\ \ so\ that\ \ }M:=\operatorname{ct}_{J}(T)=\begin{ytableau} 1 & 1 \\ 3 & 4 \\ 4 \end{ytableau},\mathrm{\ \ and\ \ }C:=\operatorname{ct}(T)=\begin{ytableau} 0 & 0 \\ 1 & 2 \\ 2 \end{ytableau},\] where $\operatorname{Dsi}(T)=\{2,3\}=\hat{J} \subseteq J$ and $\operatorname{comp}_{5}J=020120\vDash_{w}5$. The content represented by the latter composition is 11344, which is indeed the content of $M$, and the content 00122 of $C$ is represented by $\operatorname{comp}_{5}\hat{J}=212\vDash5$. The first 0 in $\operatorname{comp}_{5}J$ is because $0 \in J$ and causes $\operatorname{ct}_{J}(T)$ to contain only positive entries, and the last 0 is due to the fact that $n=5 \in J$, and does not affect $\operatorname{ct}_{J}(T)$. \label{ctJex}
\end{ex}

The following results are easy consequences of Definition \ref{multisets}, and are proved in \cite{[Z2]} and \cite{[Z3]} (though they were known before).
\begin{lem}
The following assertions hold for any $n$.
\begin{enumerate}[$(i)$]
\item The map $\operatorname{comp}_{n}$ is a bijection between multi-set of size $k-1$ inside $\mathbb{N}_{n}\cup\{0\}$ and weak composition of length $k$ of $n$.
\item The multi-set $J$ is an ordinary subset of $\mathbb{N}_{n-1}$ if and only if $\hat{J}=J$ and thus $\hat{k}=k$, which is equivalent to $\operatorname{comp}_{n}J \vDash n$, without the subscript $w$.
\item The map $\operatorname{ct}_{J}$ is a bijection from the set $\{T\in\operatorname{SYT}(\lambda)\;|\;\operatorname{Dsi}(T) \subseteq J\}$ onto the set $\operatorname{SSYT}_{\mu}(\lambda)$, where $\mu$ is the content represented by $\operatorname{comp}_{n}J$.
\item A tableau $C$ lies in $\operatorname{CCT}(\lambda)$ if and only if it is semi-standard, and wherever $h>0$ shows up in $C$, its leftmost instance is in a row below at least one instance of $h-1$.
\item The map $\operatorname{ct}:\operatorname{SYT}(\lambda)\to\operatorname{CCT}(\lambda)$ is a bijection.
\end{enumerate} \label{compctJ}
\end{lem}
Based on the bijectivity from Lemma \ref{compctJ}, we have inverse maps $\operatorname{comp}_{n}^{-1}$ from weak compositions of $n$ to multi-sets (which restricts to one from compositions to subsets of $\mathbb{N}_{n-1}$), as well as $\operatorname{ct}_{J}^{-1}$ from semi-standard tableaux of the appropriate content $\mu$ into standard ones (with the image determined by a containment condition) and $\operatorname{ct}^{-1}:\operatorname{CCT}(\lambda)\to\operatorname{SYT}(\lambda)$. The map $\operatorname{ct}_{J}^{-1}$ is a standardization map, which was defined, with the destandardization, in Definition A1.5.5 of \cite{[St2]} and the end of Section 2.1 of \cite{[vL]}, among others. The set $\operatorname{CCT}(\lambda)$ is, up to a change of normalization, the set of quasi-Yamanouchi tableaux from \cite{[AS]}, \cite{[BCDS]}, and others, with the bijection with $\operatorname{SYT}(\lambda)$ showing up there as well.

\medskip

Given $T\in\operatorname{SYT}(\lambda)$, the set $\operatorname{Asi}(T)$ of \emph{ascending indices}, or \emph{ascents}, of $T$, is the complement of $\operatorname{Dsi}(T)$ in $\mathbb{N}_{n-1}$. We will also use the following notation.
\begin{defn}
Take $\lambda \vdash n$ and tableau $T\in\operatorname{SYT}(\lambda)$ and $M\in\operatorname{SSYT}(\lambda)$.
\begin{enumerate}[$(i)$]
\item The set $\{n-i\;|\;i\in\operatorname{Dsi}(T)\}$ will be denoted by $\operatorname{Dsi}^{c}(T)$. Similarly we write $\operatorname{Asi}^{c}(T)$ for $\{n-i\;|\;i\in\operatorname{Asi}(T)\}$. They are complements inside $\mathbb{N}_{n-1}$.
\item If $J$ is the multi-set such that the content of $M$ is represented by $\operatorname{comp}_{n}J$, then we write $\operatorname{Dsp}^{c}(M)$ for the multi-set $\{n-i\;|\;n>i \in J\}$.
\item If $I$ is a multi-set containing $\operatorname{Dsi}^{c}(T)$ from part $(i)$ as in Definition \ref{multisets}, then the complement $I\setminus\operatorname{Dsi}^{c}(S)$ from that definition will be denoted by $\operatorname{Asi}^{c}_{I}(S)$. Similarly, if $M$ is a cocharge tableau $C\in\operatorname{CCT}(\lambda)$ and $I$ contains $\operatorname{Dsp}^{c}(C)$ as above then we write $\operatorname{Asp}^{c}_{I}(C)$ for the corresponding complement.
\end{enumerate} \label{sets}
\end{defn}
Definitions \ref{multisets} and \ref{sets} imply that if $M=\operatorname{ct}_{J}(S)$ for $S\in\operatorname{SYT}(\lambda)$ then $\operatorname{Dsi}^{c}(S)\subseteq\operatorname{Dsp}^{c}(M)$, and equality in the latter containment is equivalent to $M$ being $\operatorname{ct}(S)\in\operatorname{CCT}(\lambda)$. The removal of the multiplicity of $n$ from $J$ in the definition of $\operatorname{Dsp}^{c}(M)$ amounts to preventing this multi-set from containing 0, which will be more convenient for our purposes. For the tableaux from Example \ref{ctJex} we have $\operatorname{Dsp}^{c}(M)=\{2,3,3,5\}$ and $\operatorname{Dsp}^{c}(C)=\{2,3\}$.

Given any $\lambda \vdash n$, we denote by $\operatorname{ev}:\operatorname{SYT}(\lambda)\to\operatorname{SYT}(\lambda)$ the \emph{Sch\"{u}tzenberger involution}, also known as \emph{evacuation}, which is given explicitly in Section 3.9 of \cite{[Sa]}, among other references. One of its useful properties is the following one.
\begin{lem}
We have $\operatorname{Dsi}(T)=\operatorname{Dsi}^{c}(\operatorname{ev}T)$ and $\operatorname{Asi}(T)=\operatorname{Asi}^{c}(\operatorname{ev}T)$ for every $T\in\operatorname{SYT}(\lambda)$. \label{evDsiAsi}
\end{lem}
Lemma \ref{evDsiAsi} follows directly from Theorem A1.4.3 of \cite{[St2]} and Definition \ref{sets}.

We also recall the following definition, used in \cite{[Z2]} and \cite{[Z3]}.
\begin{defn}
Consider an integer $n\geq1$.
\begin{enumerate}[$(i)$]
\item If $\lambda \vdash n$, the we define $\lambda_{+} \vdash n+1$ by putting one more box in the first row of $\lambda$.
\item Let $T$ be a tableau of shape $\lambda$ and content $\mathbb{N}_{n}$. Then $\iota T$ is the tableau obtained from $T$ by increasing the shape $\lambda$ to $\lambda_{+}$ and filling the extra box with the entry $n+1$.
\item In case $S\in\operatorname{SYT}(\lambda)$, we define $\tilde{\iota}S:=\operatorname{ev}\iota(\operatorname{ev}S)$.
\item Assume that $M\in\operatorname{SSYT}(\lambda)$, and let $J$ be the multi-set for which $\operatorname{comp}_{n}J$ represents the content of $M$. Then we set $J_{+}:=\{j+1\;|\;j \in J\}$ (as a multi-set) and define $\hat{\iota}M:=\operatorname{ct}_{J_{+}}\big(\tilde{\iota}\operatorname{ct}_{J}^{-1}(M)\big)$.
\end{enumerate} \label{iotadef}
\end{defn}
As it is clear from the definition that if $T\in\operatorname{SYT}(\lambda)$ then $\iota T\in\operatorname{SYT}(\lambda_{+})$ and $\operatorname{Dsi}(\iota T)=\operatorname{Dsi}(T)$ (but $\operatorname{Asi}(\iota T)=\operatorname{Asi}(T)\cup\{n\}$), it follows from Definition \ref{iotadef} and Lemma \ref{evDsiAsi} that $\tilde{\iota}S\in\operatorname{SYT}(\lambda_{+})$ for $S\in\operatorname{SYT}(\lambda)$ and $\operatorname{Dsi}(\tilde{\iota}S)=\operatorname{Dsi}(S)_{+}$ (in the same interpretation of $J_{+}$), and applied for $S:=\operatorname{ct}_{J}^{-1}(M)$ we indeed obtain $\operatorname{Dsi}(\tilde{\iota}S) \subseteq J_{+}$, making $\hat{\iota}M$ well-defined via Lemma \ref{compctJ}.

\begin{ex}
The tableau $T$ from Example \ref{ctJex} is its own $\operatorname{ev}$-image. With $M$ and $C$ there (the latter of which satisfies the condition from part $(iv)$ of Lemma \ref{compctJ}), we get $\lambda_{+}=321\vdash6$ as well as \[\iota T=\begin{ytableau} 1 & 2 & 6 \\ 3 & 5 \\ 4 \end{ytableau},\ \ \tilde{\iota}T=\begin{ytableau} 1 & 2 & 3 \\ 4 & 6 \\ 5 \end{ytableau},\ \ \hat{\iota}M=\begin{ytableau} 0 & 1 & 1 \\ 3 & 4 \\ 4 \end{ytableau},\mathrm{\ and\ }\hat{\iota}C=\begin{ytableau} 0 & 0 & 0 \\ 1 & 2 \\ 2 \end{ytableau}.\] \label{iotaex}
\end{ex}

\medskip

The notions from Definitions \ref{defSSYT}, \ref{multisets}, and \ref{iotadef} combine in the following ways, as established in \cite{[Z2]} and \cite{[Z3]}.
\begin{lem}
Fix $\lambda \vdash n$, $S\in\operatorname{SYT}(\lambda)$, and $M\in\operatorname{SSYT}(\lambda)$.
\begin{enumerate}[$(i)$]
\item If $\{i_{g}\}_{g=1}^{k-1}$ is the increasing sequence consisting of the elements of the multi-set $\operatorname{Dsp}^{c}(M)$, of size $k-1$, then the number of entries of $M$ that equal $k-g$ or more is $i_{g}$ for every $1 \leq g<k$. This is also valid for $g=0$ with $i_{0}=0$ and for $g=k$ where $i_{k}=n$.
\item We have the equality $\sum_{i\in\operatorname{Dsp}^{c}(M)}i=\Sigma(M)$, the sum with multiplicities.
\item we can determine $M$ by the semi-standard filling of all its rows except the first one by positive integers, plus indicating the multiplicities of all the positive integers in the first row, plus the value of $n$.
\item If $S_{+}$ is the tableau of shape $\lambda$ in which $v_{S}(i)$ contains $i+1$ for every $i\in\mathbb{N}_{n}$, then by enlarging $\lambda$ to $\lambda_{+}$, shoving the entries of the first row of $S_{+}$ one spot to the right, and filling the upper left box with 1 produces $\tilde{\iota}S\in\operatorname{SYT}(\lambda_{+})$.
\item For $S\in\operatorname{SYT}(\lambda)$ we have $\operatorname{Dsi}^{c}(\tilde{\iota}S)=\operatorname{Dsi}^{c}(S)$.
\item Given $M\in\operatorname{SSYT}(\lambda)$, adding the box to get the shape $\lambda_{+}$, pushing all the first row one box to the right, and putting 0 in the remaining box yields $\hat{\iota}M\in\operatorname{SSYT}(\lambda_{+})$.
\item We have $\operatorname{Dsp}^{c}(\hat{\iota}M)=\operatorname{Dsp}^{c}(M)$ and $\Sigma(\hat{\iota}M)=\Sigma(M)$. Moreover, the data from part $(iii)$ is the same for $\hat{\iota}M$ and for $M$, with $n$ replaced by $n+1$.
\item If $C\in\operatorname{CCT}(\lambda)$ then $\hat{\iota}C\in\operatorname{CCT}(\lambda_{+})$ and it equals $\operatorname{ct}\big(\tilde{\iota}\operatorname{ct}^{-1}(C)\big)$.
\end{enumerate} \label{rels}
\end{lem}
Definition \ref{iotadef} is a special case of increasing $\lambda$ and tableaux of shape $\lambda$ by adding any box which is an external corner of $\lambda$, with Lemma \ref{rels} extending to this more general construction, as \cite{[Z2]} and \cite{[Z3]} present. However, for the purposes of the current paper the external corner in the first row suffices. One checks that for the tableaux from Example \ref{iotaex} we indeed have $\operatorname{Dsi}(\iota T)=\{2,3\}$, $\operatorname{Dsi}^{c}(\tilde{\iota}T)=\{2,3\}$, $\operatorname{Dsp}^{c}(\hat{\iota}M)=\{2,3,3,5\}$, and $\operatorname{Dsp}^{c}(\hat{\iota}C)=\{2,3\}$, matching with the sets arising from $T$, $M$, and $C$ from Example \ref{ctJex}, and that the entry sums 13 and 5 of $M$ and $C$ (or $\hat{\iota}M$ and $\hat{\iota}C$) respectively are the sum of the entry of their $\operatorname{Dsp}^{c}$-multi-sets.

\medskip

We now turn to polynomials, in finitely and infinitely many variables.
\begin{defn}
Let $n\geq1$ and $d\geq0$ be integers.
\begin{enumerate}[$(i)$]
\item We write $\mathbb{Q}[\mathbf{x}_{n}]$ for the $\mathbb{Q}[x_{1},\ldots,x_{n}]$, as well as $\mathbb{Z}[\mathbf{x}_{n}]$ for $\mathbb{Z}[x_{1},\ldots,x_{n}]$. The polynomial ring over $\mathbb{Q}$ in infinitely many variables $\{x_{i}\}_{i=1}^{\infty}$ will be denoted by $\mathbb{Q}[\mathbf{x}_{\infty}]$, with the counterpart $\mathbb{Z}[\mathbf{x}_{\infty}]$ over $\mathbb{Z}$.
\item The parts of $\mathbb{Q}[\mathbf{x}_{n}]$, $\mathbb{Z}[\mathbf{x}_{n}]$, $\mathbb{Q}[\mathbf{x}_{\infty}]$, and $\mathbb{Z}[\mathbf{x}_{\infty}]$ that consist of polynomials that are homogeneous of degree $d$ will be denoted by $\mathbb{Q}[\mathbf{x}_{n}]_{d}$, $\mathbb{Z}[\mathbf{x}_{n}]_{d}$, $\mathbb{Q}[\mathbf{x}_{\infty}]_{d}$, and $\mathbb{Z}[\mathbf{x}_{\infty}]_{d}$ respectively.
\item Given a monomial in $\mathbb{Q}[\mathbf{x}_{\infty}]$, we define the multi-set of positive exponents appearing in it to be its \emph{content}, viewed, e.g., as a partition of its degree.
\item For $\mu \vdash d$, write $\mathbb{Q}[\mathbf{x}_{n}]_{\mu}$, $\mathbb{Z}[\mathbf{x}_{n}]_{\mu}$, $\mathbb{Q}[\mathbf{x}_{\infty}]_{\mu}$, and $\mathbb{Z}[\mathbf{x}_{\infty}]_{\mu}$ for the subspace or subgroup of $\mathbb{Q}[\mathbf{x}_{n}]$, $\mathbb{Z}[\mathbf{x}_{n}]$, $\mathbb{Q}[\mathbf{x}_{\infty}]$, and $\mathbb{Z}[\mathbf{x}_{\infty}]$ respectively that are spanned, over $\mathbb{Q}$, $\mathbb{Z}$, $\mathbb{Q}$, and $\mathbb{Z}$ respectively by the monomials of content $\mu$.
\end{enumerate} \label{Qxninfd}
\end{defn}
It is clear that some groups in Definition \ref{Qxninfd} are obtained as intersections of others, like $\mathbb{Q}[\mathbf{x}_{n}]_{d}\cap\mathbb{Z}[\mathbf{x}_{n}]=\mathbb{Z}[\mathbf{x}_{n}]_{d}$, $\mathbb{Q}[\mathbf{x}_{\infty}]_{\mu}\cap\mathbb{Z}[\mathbf{x}_{\infty}]=\mathbb{Z}[\mathbf{x}_{\infty}]_{\mu}$, and others. We view $\mathbb{Q}[\mathbf{x}_{n}]$ and its subspaces as endowed with the natural action of the symmetric group $S_{n}$.

\medskip

Given any $\lambda \vdash n$ and tableau $T\in\operatorname{SYT}(\lambda)$ (or more generally, a tableau of shape $\lambda$ and content $\mathbb{N}_{n}$), we have the subgroup $R(T)$ of $S_{n}$ preserving the rows of $T$, with $C(T)$ being the subgroup acting only on the columns of $T$, while $\tilde{C}$ is the larger subgroup which also allows for interchanging of complete columns of the same length, as in, e.g., Lemma 2.1 of \cite{[Z2]}. That lemma also describes a character $\widetilde{\operatorname{sgn}}:\tilde{C}(T)\to\{\pm1\}$, which restricts to the usual sign on $C(T)$ but is trivial on $R(T)\cap\tilde{C}(T)$ (and $\tilde{C}(T)$ is the semi-direct product in which the latter subgroup acts on $C(T)$ by conjugation).

We recall from \cite{[Z2]}, \cite{[Z3]}, and others the following definition.
\begin{defn}
Take $\lambda$ and $T$ as above, and fix some $M\in\operatorname{SSYT}(\lambda)$.
\begin{enumerate}[$(i)$]
\item For every $i\in\mathbb{N}_{n}$ let $h_{i}$ be the entry of $M$ lying in the box $v_{T}(i)$, and set $p_{M,T}:=\prod_{i=1}^{n}x_{i}^{h_{i}}\in\mathbb{Q}[\mathbf{x}_{n}]$. In fact, one can define similarly the monomial $p_{H,T}$ for every tableau $H$ of shape $\lambda$ whose entries are non-negative integers.
\item The action of $R(T)$ on tableaux of shape $\lambda$ coincides with the action on the polynomials from part $(i)$. We denote by $\operatorname{st}_{T}M$ the stabilizer of $M$, or equivalently of $p_{M,T}$, in this action of $R(T)$.
\item We define the \emph{generalized higher Specht polynomial} $F_{M,T}$ associated with $M$ and $T$ to be the image of the orbit sum $\sum_{\tau \in R(T)/\operatorname{st}_{T}M}p_{M,T}$ under the operator $\sum_{\sigma \in C(T)}\operatorname{sgn}(\sigma)\sigma\in\mathbb{Q}[S_{n}]$.
\item Assume that $I$ is a (finite) multi-set inside $\mathbb{N}_{n}\cup\{0\}$, with the corresponding set $\hat{I}$ from Definition \ref{multisets}. For $i\in\hat{I}$ we set $r_{i}:=|\{j\in\mathbb{N}_{n-1}\setminus\hat{I}\;|\;j<i\}|$, while if $i$ is in the multi-set complement $I\setminus\hat{I}$ from that definition, we define $r_{i}:=n-\hat{k}+|\{j\in\hat{I}\cup\{n\}\;|\;j>i\}|$.
\item If $I$ is such a multi-set and $C\in\operatorname{CCT}(\lambda)$ is such that $\operatorname{Dsp}^{c}(C)$ from Definition \ref{sets} is contained in $I$, then set $\vec{h}(C,I)$ to be the characteristic vector of the complement $\operatorname{Asp}^{c}_{I}(C)$ from Definition \ref{multisets} with the multiplicity of 0 omitted. Similarly, we write $h_{r}:=\big|\{i\in\operatorname{Asp}^{c}_{I}(C)\;|\;r_{i}=r\}\big|$ for every $1 \leq r \leq n$, and we denote the vector $\{h_{r}\}_{r=1}^{n}$ by $\vec{h}_{C}^{I}$.
\item Take $I$ as above, and assume that $M$ is such a cocharge tableau $C$. Then we write $F_{C,T}^{I}$ for the product of $F_{C,T}$ with the symmetric polynomial $\prod_{i\in\operatorname{Asp}^{c}_{I}(C)}e_{r_{i}}=\prod_{i=1}^{n}e_{r_{i}}^{h_{i}}$, the latter expression defined using $\vec{h}_{C}^{I}$.
\item For such $I$ and $M=C$ we denote the product of $F_{C,T}$ with $\prod_{i\in\operatorname{Asp}^{c}_{I}(C)}e_{i}$, in which $e_{i}$ appears to the power showing up as the $i$th entry of $\vec{h}(C,I)$, by $F_{C,T}^{I,\mathrm{hom}}$.
\item Let $V_{M}$ be the subspace of $\mathbb{Q}[\mathbf{x}_{n}]$ that is generated by $F_{M,T}$ for all $T$ of shape $\lambda$ and content $\mathbb{N}_{n}$. In case $M=C\in\operatorname{CCT}(\lambda)$ and we are given any vector $\vec{h}=\{h_{r}\}_{r=1}^{n}$, we define $V_{C}^{\vec{h}}$ to be the image of $V_{C}$ under multiplication by the symmetric function $\prod_{i=1}^{n}e_{r_{i}}^{h_{i}}$ inside $\mathbb{Q}[\mathbf{x}_{n}]$.
\end{enumerate} \label{Spechtdef}
\end{defn}
In fact, in Definition 2.3 of \cite{[Z2]} and Definition 2.1 of \cite{[Z3]} we used the \emph{Young symmetrizer} $\varepsilon_{T}:=\sum_{\sigma \in C(T)}\sum_{\tau \in R(T)}\operatorname{sgn}(\sigma)\sigma\tau\in\mathbb{Z}[S_{n}]$, and defined $F_{M,T}$ to be $\varepsilon_{T}p_{M,T}$ divided by the size of $\operatorname{st}_{T}M$. This is clearly equivalent to our Definition \ref{Spechtdef} (see also Lemma 2.5 of \cite{[Z2]} and Lemma 2.4 of \cite{[Z3]}), but will generalize more naturally to the infinite setting below.

\begin{ex}
Take $T$, $M$, and $C$ as in Example \ref{ctJex}. Then $p_{C,T}=x_{3}x_{4}^{2}x_{5}^{2}$ and $p_{M,T}=x_{1}x_{2}x_{3}^{3}x_{4}^{4}x_{5}^{4}$, and after applying the normalized $\varepsilon_{T}$-operator we get \[F_{C,T}=(x_{3}x_{4}^{2}-x_{4}x_{3}^{2}-x_{1}x_{4}^{2}+x_{4}x_{1}^{2}+x_{1}x_{3}^{2}-x_{3}x_{1}^{2})(x_{5}^{2}-x_{2}^{2})\] and \[F_{M,T}=(x_{1}x_{3}^{3}x_{4}^{4}-x_{1}x_{4}^{3}x_{3}^{4}-x_{1}^{3}x_{3}x_{4}^{4}+x_{4}^{3}x_{3}x_{1}^{4}+x_{1}^{3}x_{3}^{2}x_{4}-x_{3}^{3}x_{1}^{4}x_{4})(x_{5}^{4}x_{2}-x_{2}^{4}x_{5}).\] The fact that the latter is divisible by the symmetric product $x_{1}x_{2}x_{3}x_{4}x_{5}$ corresponds to $M$ not containing 0.
\label{Spechtex}
\end{ex}

\medskip

We recall that the irreducible representation of $S_{n}$ are parameterized by the partitions $\lambda \vdash n$, by associating with every such $\lambda$ the \emph{Specht module} $\mathcal{S}^{\lambda}$. The basic properties of the objects from Definition \ref{Spechtdef} are as follows.
\begin{thm}
Fix a partition $\lambda \vdash n$, as well as an element $M\in\operatorname{SSYT}(\lambda)$, of content $\mu$, and set $d:=\Sigma(M)$.
\begin{enumerate}[$(i)$]
\item For each $T$ of shape $\lambda$ and content $\mathbb{N}_{n}$, the polynomial $F_{M,T}$ from Definition \ref{Spechtdef} lies in $\mathbb{Z}[\mathbf{x}_{n}]_{\mu}\subseteq\mathbb{Z}[\mathbf{x}_{n}]_{d}$, contains the monomial $p_{M,T}$ with the coefficient 1, and is acted upon by the group $\tilde{C}(T)$ via the character $\widetilde{\operatorname{sgn}}$.
\item The representation $V_{M}$ admits $\{F_{M,T}\;|\;T\in\operatorname{SYT}(\lambda)\}$ as a basis over $\mathbb{Q}$, and it is an irreducible representation of $S_{n}$ that is isomorphic to $\mathcal{S}^{\lambda}$ on a subspace of $\mathbb{Q}[\mathbf{x}_{n}]_{\mu}\subseteq\mathbb{Q}[\mathbf{x}_{n}]_{d}$.
\item When $M$ is a cocharge tableau $C\in\operatorname{CCT}(\lambda)$, the space $V_{C}^{\vec{h}}$ is also a copy of the Specht module $\mathcal{S}^{\lambda}$ on a homogeneous subspace of $\mathbb{Q}[\mathbf{x}_{n}]$, with a basis indexed by $\operatorname{SYT}(\lambda)$.
\end{enumerate} \label{repsSpecht}
\end{thm}
Part $(i)$ of Theorem \ref{repsSpecht} was proved in Proposition 2.6 of \cite{[Z2]} and Proposition 2.5 of \cite{[Z3]}. The other parts are older, already showing up, in principle, in \cite{[Pe]} and \cite{[M]} (see also Theorem 3.2 of \cite{[Z2]} and Theorem 2.10 of \cite{[Z3]}).

We give here the form of the bases of some representations $V_{M}$ with small $\Sigma(M)$, where $n$ is assumed to be large.
\begin{ex}
The representations $V_{00\cdots00}$, $V_{00\cdots01}$, $V_{00\cdots11}$, and $V_{00\cdots12}$ are trivial, spanned by the symmetric elements 1, $\sum_{i=1}^{n}x_{i}$, $\sum_{i=1}^{n}\sum_{j>i}x_{i}x_{j}$, and $\sum_{i=1}^{n}\sum_{j \neq i}x_{i}^{2}x_{j}$ of $\mathbb{Q}[\mathbf{x}_{n}]$ respectively. The Specht polynomial basis for $V_{\substack{00\cdots00 \\ 1\hphantom{1\cdots11}}}$ is $\{x_{j}-x_{1}\}_{j=2}^{n}$, while that of $V_{\substack{00\cdots00 \\ 11\hphantom{\cdots11}}}$ consists of $\{(x_{2}-x_{1})(x_{j}-x_{3})\}_{j=4}^{n}$ and $\{(x_{j}-x_{1})(x_{k}-x_{2})\}_{3 \leq j<k \leq n}$. For $V_{\substack{00\cdots01 \\ 1\hphantom{1\cdots12}}}$ we get the higher Specht polynomial basis $\big\{(x_{j}-x_{1})\sum_{k\neq1,j}x_{k}\big\}_{j=2}^{n}$, while the generalized higher Specht polynomial bases of $V_{\substack{00\cdots01 \\ 2\hphantom{2\cdots22}}}$ is given by $\big\{(x_{j}^{2}-x_{1}^{2})\sum_{k\neq1,j}x_{k}+(x_{j}^{2}x_{1}-x_{j}x_{1}^{2})\big\}_{j=2}^{n}$, and that of $V_{\substack{00\cdots02 \\ 1\hphantom{1\cdots13}}}$ is $\big\{(x_{j}-x_{1})\sum_{k\neq1,j}x_{k}^{2}-(x_{j}^{2}x_{1}-x_{j}x_{1}^{2})\big\}_{j=2}^{n}$. Finally, $V_{\substack{00\cdots00 \\ 1\hphantom{1\cdots11} \\ 2\hphantom{2\cdots22}}}$ admits the Specht basis $\{x_{l}^{2}(x_{i}-x_{1})-x_{l}(x_{i}^{2}-x_{1}^{2})+(x_{i}^{2}x_{1}-x_{i}x_{1}^{2})\}_{2 \leq i<l \leq n}$, and for $V_{\substack{00\cdots00 \\ 12\hphantom{\cdots22}}}$ we have the union of $\{(x_{2}-x_{1})(x_{j}^{2}-x_{3}^{2})+(x_{2}^{2}-x_{1}^{2})(x_{j}-x_{3})\}_{j=4}^{n}$ and $\{(x_{j}-x_{1})(x_{k}^{2}-x_{2}^{2})+(x_{j}^{2}-x_{1}^{2})(x_{k}-x_{2})\}_{3 \leq j<k \leq n}$. \label{smallex}
\end{ex}

\begin{rmk}
When doing calculations, taking combinations rather than the generalized higher Specht bases, or mixing the bases of two different representations, can sometimes yield simpler expressions. In the penultimate representation from Example \ref{smallex}, taking the difference between the basis element associated, for $j\geq4$, with $i=2$ and $l=j$ and the one with $i=2$ and $j=3$ yields $(x_{2}-x_{1})(x_{j}^{2}-x_{3}^{2})-(x_{2}^{2}-x_{1}^{2})(x_{j}-x_{3})$, while if $3 \leq j<k \leq n$ then summing the one corresponding to $i=j$ and $l=k$ with the element arising from $i=2$ and $l=j$ produces $(x_{j}-x_{1})(x_{k}^{2}-x_{2}^{2})-(x_{j}^{2}-x_{1}^{2})(x_{k}-x_{2})$. Combining these with the respective polynomials of the last representation in that example shows  that the direct sum of these representations is spanned by $\{(x_{2}-x_{1})(x_{j}^{2}-x_{3}^{2})\}_{j=4}^{n}$, $\{(x_{2}^{2}-x_{1}^{2})(x_{j}-x_{3})\}_{j=4}^{n}$, $\{(x_{j}-x_{1})(x_{k}^{2}-x_{2}^{2})\}_{3 \leq j<k \leq n}$, $\{(x_{j}^{2}-x_{1}^{2})(x_{k}-x_{2})\}_{3 \leq j<k \leq n}$, and the remaining element $x_{3}^{2}(x_{2}-x_{1})-x_{3}(x_{2}^{2}-x_{1}^{2})+(x_{2}^{2}x_{1}-x_{2}x_{1}^{2})$. Similarly, summing the $j$th element of the two representations of degree 3 and $f=1$ there yields $(x_{j}^{2}-x_{1}^{2})\sum_{k\neq1,j}x_{k}+(x_{j}-x_{1})\sum_{k\neq1,j}x_{k}^{2}$. \label{simpbasis}
\end{rmk}

\medskip

The group $S_{n}$ is embedded into $S_{n+1}$ as the stabilizer of $n+1$ in the latter group. For every tableau $T$ of some shape $\lambda \vdash n$ and content $\mathbb{N}_{n}$, with $\iota T$ as in Definition \ref{iotadef}, since $n+1$ is the single element of its column in $\iota T$, it is clear that $C(\iota T)$ is the image of $C(T)$ under that embedding. Using this observation and the fact that if $M\in\operatorname{SSYT}(\lambda)$ then $p_{M,T}$ from Definition \ref{Spechtdef} is in the $R(\iota T)$-orbit of $p_{\hat{\iota}M,\iota T}$, Proposition 3.23 of \cite{[Z3]} states the following result.
\begin{prop}
Given such $\lambda$, $T$, and $M$, the generalized higher Specht polynomial $F_{M,T}$ from Definition \ref{Spechtdef} is obtained by substituting $x_{n+1}=0$ in the generalized higher Specht polynomial $F_{\hat{\iota}M,\iota T}$. Moreover, if the first row of $M$ contains at least $\lambda_{2}$ zeros, then $F_{\hat{\iota}M,\iota T}$ is determined as the unique polynomial that satisfies this property and upon which the action of $\tilde{C}(\iota T)$ is by $\widetilde{\operatorname{sgn}}$. \label{forstab}
\end{prop}
The second assertion in Proposition \ref{forstab}, which does not appear in Proposition 3.23 of \cite{[Z3]}, is proved by noting that the $\tilde{C}(\iota T)$ relates any monomial in $F_{\hat{\iota}M,\iota T}$ to one that is not divisible by $x_{n+1}$ and thus shows up in $F_{M,T}$ (see the proof of Theorem 2.17 of \cite{[Z2]}).

The case of $M=C\in\operatorname{CCT}(\lambda)$ in Proposition \ref{forstab}, showing up in Proposition 2.15 of \cite{[Z2]}, was the basis of the construction of certain homogeneous power series called stable higher Specht polynomials in Theorem 2.17 of that reference. This will be done more generally, and in the appropriate context (in fact, the construction is valid in general, but more notation is required---see Remark \ref{notation} below), in Proposition \ref{limSpecht} below.

\begin{rmk}
In fact, to any shape $\lambda$ one associates, via Definition 2.8 of \cite{[Z2]}, its minimal cocharge tableau $C^{0}$, which is $R(T)$-invariant for any $T$ of shape $\lambda$ and content $\mathbb{N}_{n}$, and where $F_{C^{0},T}$ is the classical Specht polynomial associated with $T$. It divided any other generalized higher Specht polynomial $F_{M,T}$ for the same $T$, with the quotient $Q_{M,T}:=F_{M,T}/F_{C^{0},T}$ being $\tilde{C}(T)$-invariant---see Corollary 2.9 of \cite{[Z2]} and Corollary 2.6 of \cite{[Z3]}. \label{quotSpecht}
\end{rmk}

Proposition \ref{forstab} admits an analogue for the quotients from Remark \ref{quotSpecht}---see part $(ii)$ of Corollary \ref{infquot} below.

\begin{ex}
To the polynomials from Example \ref{Spechtex} we associate the quotients $Q_{C,T}=x_{5}+x_{2}$ and \[Q_{M,T}=x_{1}x_{2}x_{3}x_{4}x_{5}(x_{3}x_{4}+x_{1}x_{4}+x_{1}x_{3})(x_{5}^{2}+x_{5}x_{2}+x_{2}^{2}).\] Going over to the tableaux from Example \ref{iotaex}, we get $F_{\hat{\iota}C,\iota T}=F_{C,T}$ (in correspondence with Lemma \ref{sameiota} below) and thus also $Q_{\hat{\iota}C,\iota T}=Q_{C,T}$ (see Corollary \ref{infquot} below). The polynomial $F_{\hat{\iota}M,\iota T}$ is the sum of $F_{M,T}$ (in which the mentioned multiplier from Example \ref{Spechtex} is no longer symmetric in $\mathbb{Q}[\mathbf{x}_{6}]$) and
\[x_{6}(x_{1}x_{3}^{3}x_{4}^{4}-x_{1}x_{4}^{3}x_{3}^{4}-x_{1}^{3}x_{3}x_{4}^{4}+x_{4}^{3}x_{3}x_{1}^{4}+x_{1}^{3}x_{3}^{2}x_{4}-x_{3}^{3}x_{1}^{4}x_{4})(x_{5}^{4}-x_{2}^{4})+\]
\[+x_{6}(x_{3}^{3}x_{4}^{4}-x_{4}^{3}x_{3}^{4}-x_{1}^{3}x_{4}^{4}+x_{4}^{3}x_{1}^{4}+x_{1}^{3}x_{3}^{2}-x_{3}^{3}x_{1}^{4})(x_{5}^{4}x_{2}-x_{2}^{4}x_{5}).\] Substituting $x_{6}=0$ annihilates the latter summands, and reduces $F_{\hat{\iota}M,\iota T}$ to $F_{M,T}$, as Proposition \ref{forstab} states. The quotient $Q_{\hat{\iota}M,\iota T}$ is $Q_{M,T}$ plus \[x_{1}x_{3}x_{4}x_{6}(x_{3}x_{4}+x_{1}x_{4}+x_{1}x_{3})(x_{5}^{3}+x_{5}^{2}x_{2}+x_{5}x_{2}^{2}+x_{2}^{3})+\]
\[+x_{2}x_{5}x_{6}(x_{3}^{2}x_{4}^{2}+x_{1}^{2}x_{4}^{2}+x_{1}^{2}x_{3}^{2}+x_{1}^{2}x_{3}x_{4}+x_{1}x_{3}^{2}x_{4}+x_{1}x_{3}x_{4}^{2})(x_{5}^{2}+x_{5}x_{2}+x_{2}^{2}),\] in which substituting $x_{6}=0$ yields just $Q_{M,T}$, as Corollary \ref{infquot} below predicts. \label{quotex}
\end{ex}

\medskip

We recall from Theorem 3.20 and Proposition 3.25 of \cite{[Z2]}, as well as Theorem 2.19 of \cite{[Z3]}, the following construction.
\begin{defn}
Let $I$ be a multi-set inside $\mathbb{N}_{n}\cup\{0\}$. Then the representation $R_{n,I}$ associated with $n$ and $I$ is defined to be $\bigoplus_{\lambda \vdash n}\bigoplus_{C\in\operatorname{CCT}(\lambda),\ \operatorname{Dsp}^{c}(C) \subseteq I}V_{C}^{\vec{h}_{C}^{I}}$, and the homogeneous representation $R_{n,I}^{\mathrm{hom}}$ corresponding to these parameters is $\bigoplus_{\lambda \vdash n}\bigoplus_{C\in\operatorname{CCT}(\lambda),\ \operatorname{Dsp}^{c}(C) \subseteq I}V_{C}^{\vec{h}(C,I)}$. \label{RnIdef}
\end{defn}
The fact that the sums in Definition \ref{RnIdef} are direct is established in the aforementioned results, as well as many other properties of these representations (in particular, $R_{n,I}^{\mathrm{hom}}$ is contained in the part $\mathbb{Q}[\mathbf{x}_{n}]_{d}$ that is homogeneous of degree $d:=\sum_{i \in I}i$, whence the name and notation). However, the property that we shall need here is the following one, for which we recall that the same $I$ can be used for any $n$ that is large enough.
\begin{prop}
If $n$ is larger than twice the maximal element of $I$, then $R_{n+1,I}^{\mathrm{hom}}$ equals $\bigoplus_{\lambda \vdash n}\bigoplus_{C\in\operatorname{CCT}(\lambda),\ \operatorname{Dsp}^{c}(C) \subseteq I}V_{\hat{\iota}C}^{\vec{h}(\hat{\iota}C,I)}$, with the summation indices being as for $R_{n,I}$. If $I$ is a set (rather than a multi-set) and does not contain 0, then $R_{n+1,I}$ is similarly given by $\bigoplus_{\lambda \vdash n}\bigoplus_{C\in\operatorname{CCT}(\lambda),\ \operatorname{Dsp}^{c}(C) \subseteq I}V_{\hat{\iota}C}^{\vec{h}_{\hat{\iota}C}^{I}}$, with the summation being as for $R_{n,I}$. \label{plus1iota}
\end{prop}
Proposition \ref{plus1iota} is proved as Lemma 4.20 and Corollary 4.22 of \cite{[Z2]} for sets $I$, and as Corollary 3.20 of \cite{[Z3]} in the general setting of the homogeneous representations, and the inequality in the condition in that proposition does not have to be strict (see Remark 4.21 of the former reference). Remark 3.7 of \cite{[Z3]} explains why the assertion about $R_{n+1,I}$ does not hold as nicely for multi-sets, and we will see in Theorem \ref{decomRinf} below that sets suffice for these representations. Remark 4.24 of \cite{[Z2]} and Remark 3.28 of \cite{[Z3]} relate Proposition \ref{plus1iota} to stable representations at the limit (see also Section 7 of \cite{[PR]}).

\begin{ex}
Take $I:=\{1,3,3\}$, so that $\hat{I}=\{1,3\}$. Then for $n=5$ we get
\[R_{5,I}^{\mathrm{hom}}=V_{00000}e_{1}e_{3}^{2} \oplus V_{\substack{0000 \\ 1\hphantom{111}}}e_{3}^{2} \oplus V_{\substack{0011 \\ 1\hphantom{122}}}e_{1}e_{3} \oplus V_{\substack{001 \\ 11\hphantom{2}}}e_{1}e_{3} \oplus V_{\substack{001 \\ 1\hphantom{12} \\ 2\hphantom{23}}}e_{3} \oplus V_{\substack{00 \\ 11 \\ 2\hphantom{3}}}e_{3}\] and \[R_{5,\hat{I}}=V_{00000}e_{1} \oplus V_{\substack{0000 \\ 1\hphantom{111}}}e_{1} \oplus V_{\substack{0011 \\ 1\hphantom{122}}} \oplus V_{\substack{001 \\ 11\hphantom{2}}} \oplus V_{\substack{001 \\ 1\hphantom{12} \\ 2\hphantom{23}}} \oplus V_{\substack{00 \\ 11 \\ 2\hphantom{3}}}.\] For $n=6$, the homogeneous representation $R_{6,I}^{\mathrm{hom}}$ equals \[V_{000000}e_{1}e_{3}^{2} \oplus V_{\substack{00000 \\ 1\hphantom{1111}}}e_{3}^{2} \oplus V_{\substack{00011 \\ 1\hphantom{1122}}}e_{1}e_{3} \oplus V_{\substack{0001 \\ 11\hphantom{12}}}e_{1}e_{3} \oplus V_{\substack{000 \\ 111}}e_{1}e_{3} \oplus V_{\substack{0001 \\ 1\hphantom{112} \\ 2\hphantom{223}}}e_{3} \oplus V_{\substack{000 \\ 11\hphantom{1} \\ 2\hphantom{22}}}e_{3},\] and we have \[R_{6,\hat{I}}=V_{000000}e_{1} \oplus V_{\substack{00000 \\ 1\hphantom{1111}}}e_{1} \oplus V_{\substack{00011 \\ 1\hphantom{1122}}} \oplus V_{\substack{0001 \\ 11\hphantom{12}}} \oplus V_{\substack{000 \\ 111}} \oplus V_{\substack{0001 \\ 1\hphantom{112} \\ 2\hphantom{223}}} \oplus V_{\substack{000 \\ 11\hphantom{1} \\ 2\hphantom{22}}}.\] Passing to $n=7$, the representations $R_{7,I}^{\mathrm{hom}}$ and $R_{7,\hat{I}}$ become \[V_{0000000}e_{1}e_{3}^{2} \oplus V_{\substack{000000 \\ 1\hphantom{11111}}}e_{3}^{2} \oplus V_{\substack{000011 \\ 1\hphantom{11122}}}e_{1}e_{3} \oplus V_{\substack{00001 \\ 11\hphantom{112}}}e_{1}e_{3} \oplus V_{\substack{0000 \\ 111\hphantom{1}}}e_{1}e_{3} \oplus V_{\substack{00001 \\ 1\hphantom{1112} \\ 2\hphantom{2223}}}e_{3} \oplus V_{\substack{0000 \\ 11\hphantom{11} \\ 2\hphantom{222}}}e_{3}\] and \[V_{0000000}e_{1} \oplus V_{\substack{000000 \\ 1\hphantom{11111}}}e_{1} \oplus V_{\substack{000011 \\ 1\hphantom{11122}}} \oplus V_{\substack{00001 \\ 11\hphantom{112}}} \oplus V_{\substack{0000 \\ 111\hphantom{1}}} \oplus V_{\substack{00001 \\ 1\hphantom{1112} \\ 2\hphantom{2223}}} \oplus V_{\substack{0000 \\ 11\hphantom{11} \\ 2\hphantom{222}}}\] respectively, with each summand from $R_{6,I}^{\mathrm{hom}}$ and $R_{6,\hat{I}}$ replaced by its $\hat{\iota}$-image, exemplifying Proposition \ref{plus1iota} (with the non-strict inequality in the assumption, as 6 equals twice the maximal element of $I$ and $\hat{I}$). \label{RnIex}
\end{ex}
Note that when going from $n=5$ to $n=6$ in Example \ref{plus1iota}, apart from $\hat{\iota}$-images we get the multiple of the representation $V_{\substack{000 \\ 111}}$, whose index is not a $\hat{\iota}$-image, showing the necessity of the assumption on $n$ in Proposition \ref{plus1iota}.

\medskip

Theorem 3.20 of \cite{[Z2]} and Theorem 2.19 of \cite{[Z3]} also give the following result.
\begin{thm}
Take some $k\geq1$ and $n\geq1$.
\begin{enumerate}[$(i)$]
\item The sum of $R_{n,I}^{\mathrm{hom}}$ over all multi-sets of size $k-1$ inside $\mathbb{N}_{n}\cup\{0\}$, with the multiplicity of 0 allowed and counted in the size, is direct inside $\mathbb{Q}[\mathbf{x}_{n}]$ and maps bijectively onto the quotient $R_{n,k,0}$ of that ring by the ideal generated by $x_{i}^{k}$, $1 \leq i \leq n$.
\item If $k \leq n$ then the sum of $R_{n,I}$ over subsets $I\subseteq\mathbb{N}_{n-1}$ of size $k-1$ is also direct inside $\mathbb{Q}[\mathbf{x}_{n}]$, and bijects onto the quotient $R_{n,k}$ of $R_{n,k,0}$ by the elementary symmetric functions $e_{r}$ for $n+1-k \leq r \leq n$.
\end{enumerate} \label{Rnkdecom}
\end{thm}
Allowing multi-sets in part $(ii)$ of Theorem \ref{Rnkdecom} also produces a direct sum that bijects onto $R_{n,k,0}$ from part $(i)$ there (without the restriction $k \leq n$), but this result does not have an infinite analogue in Theorem \ref{decomRinf} below.

We also recall the following result, which is given in Theorem 2.34 and Corollaries 3.19 and 3.27 of \cite{[Z3]}.
\begin{thm}
Fix some degree $d\geq0$, as well as a content $\mu$ of sum $d$.
\begin{enumerate}[$(i)$]
\item The representation $\mathbb{Q}[\mathbf{x}_{n}]_{d}$ decomposes as $\bigoplus_{\lambda \vdash n}\bigoplus_{M\in\operatorname{SSYT}_{d}(\lambda)}V_{M}$.
\item Inside part $(i)$ we have $\mathbb{Q}[\mathbf{x}_{n}]_{\mu}=\bigoplus_{\lambda \vdash n}\bigoplus_{M\in\operatorname{SSYT}_{\mu}(\lambda)}V_{M}$.
\item If $n>2d$ then the map taking $M$ to $\hat{\iota}M$ is a bijection between the irreducible summands in parts $(i)$ and $(ii)$ for $n$ and for $n+1$.
\end{enumerate} \label{FMTdecom}
\end{thm}
Also in part $(iii)$ of Theorem \ref{FMTdecom}, the non-strict inequality $n\geq2d$ suffices. Proposition \ref{ExtVMlim} and Corollary \ref{semisimp} below state two different infinite analogues of that theorem.

\begin{ex}
As in Examples 2.35 and 3.22 of \cite{[Z3]}, we recall the decompositions $\mathbb{Q}[\mathbf{x}_{3}]_{2}=V_{002} \oplus V_{\substack{00 \\ 2\hphantom{2}}} \oplus V_{011} \oplus V_{\substack{01 \\ 1\hphantom{2}}}$, $\mathbb{Q}[\mathbf{x}_{4}]_{2}=V_{0002} \oplus V_{\substack{000 \\ 2\hphantom{22}}} \oplus V_{0011} \oplus V_{\substack{001 \\ 1\hphantom{12}}} \oplus V_{\substack{00 \\ 11}}$, and $\mathbb{Q}[\mathbf{x}_{5}]_{2}=V_{00002} \oplus V_{\substack{0000 \\ 2\hphantom{222}}} \oplus V_{00011} \oplus V_{\substack{0001 \\ 1\hphantom{112}}} \oplus V_{\substack{000 \\ 11\hphantom{1}}}$, to which we add the expression $\mathbb{Q}[\mathbf{x}_{6}]_{2}=V_{000002} \oplus V_{\substack{00000 \\ 2\hphantom{2222}}} \oplus V_{000011} \oplus V_{\substack{00001 \\ 1\hphantom{1112}}} \oplus V_{\substack{0000 \\ 11\hphantom{11}}}$, all of which are as in part $(i)$ of Theorem \ref{FMTdecom}. The decompositions according to the two contents of sum 2, as in part $(ii)$ there, are clear. We also see that once $n$ reaches $2d=4$, passing from $n$ to $n+1$ simply replaces every $V_{M}$ by $V_{\hat{\iota}M}$ in the decomposition, as part $(iii)$ there states. \label{Qxnd2ex}
\end{ex}
In fact, we wrote the bases for the representations showing up in Example \ref{Qxnd2ex} in Example \ref{smallex}, where for those involving an entry 2 we simply take those with the same entry 1 in the latter example and replace each $x_{i}$ by $x_{i}^{2}$.

\section{Actions on Eventually Symmetric Functions \label{EvSymFunc}}

Theorem 2.17 of \cite{[Z2]} produces certain power series as limits of higher Specht polynomials. We now construct the framework for these power series, which will be the space inside which the limits of our stable representations will take place. In addition to $\mathbb{Q}[\mathbf{x}_{\infty}]$ from Definition \ref{Qxninfd}, we recall that $\Lambda$ stands for the ring of symmetric functions (over $\mathbb{Z}$), with its natural grading as $\bigoplus_{d=0}^{\infty}\Lambda_{d}$.
\begin{defn}
Let $\mathbb{Q}\ldbrack\mathbf{x}_{\infty}\rdbrack$ denote the ring of power series in the infinitely many variables $\{x_{i}\}_{i=1}^{\infty}$ over $\mathbb{Q}$.
\begin{enumerate}[$(i)$]
\item Let $S_{\mathbb{N}}$ denote the infinite symmetric group on $\mathbb{N}$.
\item If $S_{n}$ is embedded into $S_{\mathbb{N}}$ as the subgroup fixing every $m>n$, then the embedding of $S_{n}$ into $S_{n+1}$ is just an inclusion. We denote by $S_{\infty}$ the union $\bigcup_{n=1}^{\infty}S_{\infty}$ inside $S_{\mathbb{N}}$.
\item For any $n$, we write $S_{\mathbb{N}}^{(n)}$ for the subgroup of $S_{\mathbb{N}}$ that fixes every number $1 \leq i \leq n$.
\item Let $F$ be an element of $\mathbb{Q}\ldbrack\mathbf{x}_{\infty}\rdbrack$ that is homogeneous of some degree $d$. We say that $F$ is \emph{eventually symmetric} if there exists some $n$ such that $F$ is stabilized by all of $S_{\mathbb{N}}^{(n)}$. The set of eventually symmetric series that are homogeneous of degree $d$ is denoted by $\tilde{\Lambda}_{d}$.
\item A general element $F\in\mathbb{Q}\ldbrack\mathbf{x}_{\infty}\rdbrack$ is an \emph{eventually symmetric function} if it is of bounded degree and is stabilized by $S_{\mathbb{N}}^{(n)}$ for some $n$. The set of such series will be denoted by $\tilde{\Lambda}$.
\end{enumerate} \label{evsym}
\end{defn}

\begin{rmk}
We can now recall the conventions for the notation for higher Specht polynomials showing up previously in the literature. The original paper \cite{[ATY]} used pairs of semi-standard Young tableaux $T$ and $S$ of the same shape for indexing higher Specht polynomials, and wrote $F_{T}^{S}$ for the polynomial denoted here as $F_{\operatorname{ct}(S),T}$ (see also \cite{[Z2]}, and Remark 2.2 of \cite{[Z3]}). Then the higher Specht polynomial related with $F_{T}^{S}$ via Proposition \ref{forstab} is $F_{\iota T}^{\tilde{\iota}S}$ (with the notation $V^{S}_{\vec{h}}$ for the representation $V_{\operatorname{ct}(S)}^{\vec{h}}$, and a similar formulation for Proposition \ref{plus1iota}). Using the RSK algorithm, matching pairs of such tableaux with elements of $S_{n}$, it was shown in \cite{[Z1]} that the most natural way to do so is by writing $F_{w}$ for $F_{P(w)}^{\tilde{Q}(w)}$, where $P(w)$ and $Q(w)$ are the RSK tableaux as in \cite{[Sa]} and others, and $\tilde{Q}(w):=\operatorname{ev}Q(w)$ (or equivalently $Q(w_{0}ww_{0})$, where $w_{0}$ is the longest element of $S_{n}$). Indeed, then Proposition \ref{forstab} relates $F_{w}$ with $F_{w_{+}}$, where $w_{+} \in S_{n+1}$ is the image of $w$ stabilizing $n+1$, and the limit from Theorem 2.17 of \cite{[Z2]} can be associated with elements $\hat{w} \in S_{\infty}$. In general, we shall need the notation using Definition \ref{tabinf} below. \label{notation}
\end{rmk}

\medskip

The following properties of $\tilde{\Lambda}$ are clear from Definition \ref{evsym}.
\begin{lem}
The set $\tilde{\Lambda}$ is a subring of $\mathbb{Q}\ldbrack\mathbf{x}_{\infty}\rdbrack$, which is graded as $\bigoplus_{d=0}^{\infty}\tilde{\Lambda}_{d}$. It contains both $\mathbb{Q}[\mathbf{x}_{\infty}]$ and $\Lambda$. \label{ringES}
\end{lem}

\begin{proof}
The fact that the subgroups $\{S_{\mathbb{N}}^{(n)}\}_{n=1}^{\infty}$ form a decreasing sequence inside $S_{\mathbb{N}}$ and that the action of the latter on $\mathbb{Q}\ldbrack\mathbf{x}_{\infty}\rdbrack$ preserves degrees implies that an element $F=\sum_{d=0}^{h}F_{d}$ of bounded degree in $\mathbb{Q}\ldbrack\mathbf{x}_{\infty}\rdbrack$ lies in $\tilde{\Lambda}$ if and only if each $F_{d}$ lies in $\tilde{\Lambda}_{d}$. It also follows that $\tilde{\Lambda}$ is closed under the ring operations (and clearly contains $\mathbb{Q}$), yielding the first assertion.

We now note that elements of $\Lambda$ satisfy the condition form Definition \ref{evsym}. Similarly, a polynomial in $\mathbb{Q}[\mathbf{x}_{\infty}]$ has bounded degree, and it involves finitely many variables, and is thus in the image of $\mathbb{Q}[\mathbf{x}_{n}]$ inside $\mathbb{Q}[\mathbf{x}_{\infty}]$. But since $S_{\mathbb{N}}^{(n)}$ acts trivially on $\mathbb{Q}[\mathbf{x}_{n}]$ by definition, the second assertion follows as well. This completes the proof of the lemma.
\end{proof}
Lemma \ref{ringES} implies that the ring generated by $\Lambda$ and $\mathbb{Q}[\mathbf{x}_{\infty}]$, which we can write as $\Lambda_{\mathbb{Q}}[\mathbf{x}_{\infty}]$ (with $\Lambda_{\mathbb{Q}}:=\Lambda\otimes_{\mathbb{Z}}\mathbb{Q}$, since we adopt the usual notation of $\Lambda$ being the integral ring) is contained inside $\tilde{\Lambda}$. Moreover, we recall from the theory of symmetric functions that the elementary symmetric functions $\{e_{r}\}_{r=1}^{\infty}$ are free generators for $\Lambda$ over $\mathbb{Z}$ (hence of $\Lambda_{\mathbb{Q}}$ over $\mathbb{Q}$), namely we can write $\Lambda=\mathbb{Z}[\mathbf{e}_{\infty}]$ and $\Lambda_{\mathbb{Q}}=\mathbb{Q}[\mathbf{e}_{\infty}]$ by adapting the notation from Definition \ref{Qxninfd}, and then the latter ring is $\mathbb{Q}[\mathbf{x}_{\infty},\mathbf{e}_{\infty}]$. Moreover, we recall that when considering representations on power series in infinitely many variables (say homogeneous), there is a difference between the representations of the full infinite symmetric group $S_{\mathbb{N}}$ and the increasing union $S_{\infty}$ (see, e.g., Example \ref{difSinf} below). We thus turn to the following result.
\begin{prop}
The ring $\tilde{\Lambda}$ from Lemma \ref{ringES} has the following properties.
\begin{enumerate}[$(i)$]
\item The variables $\{x_{i}\}_{i=1}^{\infty}$ and the elementary symmetric functions $\{e_{r}\}_{r=1}^{\infty}$ combine to a set of algebraically independent variables.
\item The ring $\mathbb{Q}[\mathbf{x}_{\infty},\mathbf{e}_{\infty}]$ that they generate is all of $\tilde{\Lambda}$.
\item The action of $S_{\mathbb{N}}$, hence also of $S_{\infty}$, on $\mathbb{Q}\ldbrack\mathbf{x}_{\infty}\rdbrack$ preserves $\tilde{\Lambda}$.
\item A subspace of $\tilde{\Lambda}$ is a sub-representation with respect to the action of $S_{\infty}$ if and only if it is such for $S_{\mathbb{N}}$.
\end{enumerate} \label{samereps}
\end{prop}

\begin{proof}
The variables $\{x_{i}\}_{i=1}^{\infty}$ are algebraically independent by definition, and it is known that so are $\{e_{r}\}_{r=1}^{\infty}$ in $\Lambda$. Assume now that there exists an algebraic relation $f$ between the union of these sets, which is thus a polynomial in the $e_{r}$'s and finitely many $x_{i}$'s, and let $n$ be the maximal index of a variable $x_{i}$ showing up in $f$. Then $f$ is in the image of $\mathbb{Q}[\mathbf{x}_{n},\mathbf{e}_{\infty}]$, and we need to show that $f$ is the zero polynomial.

Assume otherwise, and then we can write $f$ as $\sum_{d=0}^{h}\sum_{\mu \vdash d}q_{\mu}(\mathbf{x}_{n})e_{\mu}$ for polynomials $q_{\mu}\in\mathbb{Q}[\mathbf{x}_{n}]$, where $e_{\mu}$ is the usual symmetric function $\prod_{r}e_{r}^{m_{r}}$ in which $m_{r}$ is the multiplicity of $r$ in $\mu$, and $q_{\mu}(\mathbf{x}_{n})$ is a shorthand for $q_{\mu}(x_{1},\ldots,x_{n})$. By reducing the value of $h$ if necessary, we may assume that there is some $\mu \vdash h$ with $q_{\mu}\neq0$. We thus obtain an equality comparing $\sum_{\mu \vdash h}q_{\mu}(\mathbf{x}_{n})e_{\mu}$ with an expression which has strictly lower degree in the $e_{r}$'s.

We now write each $e_{r}$ as $\sum_{t=0}^{r}e_{r-t}(\mathbf{x}_{n})e_{t}^{(n)}$, where $e_{t}^{(n)}$ is the $t$th elementary symmetric function in the variables $\{x_{m}\;|\;m>n\}$. After this substitution, every $e_{\mu}$ for $\mu \vdash h$ becomes $e_{\mu}^{(n)}$ (defined in an analogous manner) plus an element of $\mathbb{Q}[\mathbf{x}_{n},\mathbf{e}_{\infty}^{(t)}]$ (defined in the same manner), whose (usual weighted) degree in the $e_{t}^{(n)}$'s is strictly smaller than $h$. The terms involving $e_{\mu}$ for $\mu \vdash d<h$ also become elements of $\mathbb{Q}[\mathbf{x}_{n},\mathbf{e}_{\infty}^{(t)}]$ with the same degree assumption.

Altogether, $f=0$ becomes an equality comparing $\sum_{\mu \vdash h}q_{\mu}(\mathbf{x}_{n})e_{\mu}^{(n)}$ (with the same $q_{\mu}$'s) with some expression whose degree in the $e_{t}^{(n)}$'s is smaller than $h$. But $\{x_{i}\}_{i=1}^{n}\cup\{e_{t}^{(n)}\}_{t=1}^{\infty}$ is the union of two sets of algebraically independent expressions that are supported on disjoint sets of variables, hence they are algebraically independent themselves. Thus such a relation cannot hold, producing a contradiction, which implies that our initial polynomial $f$ cannot be non-zero. part $(i)$ is thus established.

for part $(ii)$, we first fix some $n$, and note that every symmetric function in the variables $\{x_{m}\;|\;m>n\}$ is stabilized by $S_{\infty}^{(n)}$ and is hence an element of $\tilde{\Lambda}$. To show that it is generated by $\mathbb{Q}[\mathbf{x}_{\infty}]$ and $\Lambda$, it suffices to consider generators of the algebra of the symmetric functions in these variables, which we take to be the power sum symmetric functions (we work over $\mathbb{Q}$, hence these are indeed generators). But the $r$th power sum $\sum_{m>n}x_{m}^{r}$ is the difference between the symmetric power sum $\sum_{i}x_{i}^{r}\in\Lambda$ and the polynomial $\sum_{i=1}^{n}x_{i}^{r}\in\mathbb{Q}[\mathbf{x}_{n}]\subseteq\mathbb{Q}[\mathbf{x}_{\infty}]$, proving that every such function lies in $\mathbb{Q}[\mathbf{x}_{\infty},\mathbf{e}_{\infty}]$.

Consider now a general element $F\in\tilde{\Lambda}$, and let $n$ be such that $S_{\infty}^{(n)}$ stabilizes $F$. We write $F$ as a sum of monomials, and every such monomial is the product of a monomial $q(\mathbf{x}_{n})$ times a monomial in $\{x_{m}\;|\;m>n\}$. The invariance under $S_{\infty}^{(n)}$ implies that the product of $q(\mathbf{x}_{n})$ with the image of the latter monomial under every permutation of the variables with larger indices must also show up in $F$, and gathering them together produces $q(\mathbf{x}_{n})$ times a (monomial) symmetric function in these variables.

We thus express $F$ as a (possibly infinite) sum of products of polynomials from $\mathbb{Q}[\mathbf{x}_{n}]$ and monomial symmetric functions in $\{x_{m}\;|\;m>n\}$. Recall that the degree of $F$ is bounded by Definition \ref{evsym}, and that for every degree $d$ there are finitely many monomial symmetric functions in $\{x_{m}\;|\;m>n\}$ that have degree at most $d$. Therefore $F$ is indeed generated, algebraically, by $\mathbb{Q}[\mathbf{x}_{n}]$ and those symmetric functions. As the latter were seen to be in $\tilde{\Lambda}$, and Lemma \ref{ringES} provides the same for the former, part $(ii)$ follows.

Next, take $F\in\tilde{\Lambda}$ let $n$ be such that $S_{\infty}^{(n)}$ stabilizes $F$, and consider an element $v \in S_{\mathbb{N}}$. Simple conjugation shows that if $m=\max\{v(i)\;|\;1 \leq i \leq n\}$ then $S_{\infty}^{(m)} \subseteq vS_{\infty}^{(n)}v^{-1}$, and since the latter group stabilizes $vF$, we deduce that $vF\in\tilde{\Lambda}$ as well. This proves part $(iii)$, with the assertion about the subgroup $S_{\infty} \subseteq S_{\mathbb{N}}$ following immediately.

The fact that $S_{\infty} \subseteq S_{\mathbb{N}}$ shows that one direction in part $(iv)$ is trivial. For the other direction, it suffices to show that if $F\in\tilde{\Lambda}$ and $v \in S_{\mathbb{N}}$, then the element $vF$, which lies in $\tilde{\Lambda}$ as well by part $(iii)$, can be related to $F$ via the action of $S_{\infty}$. For this we take $n$ and $m$ as in the above paragraph, and recall that $F$ can be written as $\sum_{d=0}^{h}\sum_{\mu \vdash d}q_{\mu}(\mathbf{x}_{n})e_{\mu}$ by our proof of part $(ii)$.

But then $vF$ equals $\sum_{h=0}^{d}\sum_{\mu \vdash h}q_{\mu}(x_{v(1)},\ldots,x_{v(n)})e_{\mu}$, and we can some element $w \in S_{m}$ with $w(i)=v(i)$ for all $1 \leq i \leq n$. By embedding $S_{m}$ into $S_{\infty}$, and identifying $w$ with its image there, we deduce that $w \in S_{\infty}$ and $wF=vF$, yielding part $(iv)$. This completes the proof of the proposition.
\end{proof}
In fact, as one can easily shows that $e_{t}^{(n)}$ is generated by $\{x_{i}\}_{i=1}^{\infty}$ (or, in fact, just $\{x_{i}\}_{i=1}^{n}$) and $\{e_{r}\}_{r=1}^{\infty}$ over $\mathbb{Z}$, we can strengthen Proposition \ref{samereps} to the statement that $\tilde{\Lambda}\cap\mathbb{Z}\ldbrack\mathbf{x}_{\infty}\rdbrack=\mathbb{Z}[\mathbf{x}_{\infty},\mathbf{e}_{\infty}]$. However, as we only work over $\mathbb{Q}$ in this paper, the result as asserted suffices for our purposes.

\begin{rmk}
Part $(iii)$ of Proposition \ref{samereps} allows us to consider representations of $S_{\mathbb{N}}$ and of $S_{\infty}$ inside $\tilde{\Lambda}$ (which are the same by part $(iv)$ there). By part $(ii)$ of that proposition, we may take a basis for $\Lambda$ over $\mathbb{Q}$ (say the monomial one), and then $\tilde{\Lambda}$ admits the same basis over $\mathbb{Q}[\mathbf{x}_{\infty}]$, and the representations are obtained through the action on $\mathbb{Q}[\mathbf{x}_{\infty}]$ alone (since $S_{\mathbb{N}}$ and $S_{\infty}$ stabilize elements of $\Lambda$). Most of the representations that we will construct will be supported on several parts in such a decomposition, but we will eventually show how to construct representations that respect this decomposition as well. \label{polstimessym}
\end{rmk}

\begin{ex}
Consider, inside the part of $\mathbb{Q}\ldbrack\mathbf{x}_{\infty}\rdbrack$ that is homogeneous of degree 1, the sum of all the variables with even indices, and the sum of all the variables having odd indices. These are relation by the action of the larger group $S_{\mathbb{N}}$, but not by the action of the smaller one $S_{\infty}$ (as each of its elements fixes all but finitely many numbers). This exemplifies the statement about the representations in Proposition \ref{samereps}, as well as shows that these two elements are not in $\tilde{\Lambda}$ (which is also clear from Definition \ref{evsym}). \label{difSinf}
\end{ex}

\medskip

The limit elements, like those from Theorem 2.17 of \cite{[Z2]}, are obtained by repeating the operations $\iota$ and $\hat{\iota}$ from Definition \ref{iotadef}, via the relation from Proposition \ref{plus1iota} at each stage. In order to describe the limit elements, we first present the objects that show up in their indices.
\begin{defn}
We introduce the following notions.
\begin{enumerate}[$(i)$]
\item An \emph{infinite Ferrers diagram}, denoted as $\hat{\lambda}\vdash\infty$, is a Ferrers diagram, in which the first row is infinite. Its \emph{finite part} is the complement of the first, infinite row.
\item An \emph{infinite standard Young tableau} $\hat{T}$ is any tableau whose shape is an infinite Ferrers diagram $\hat{\lambda}$, which satisfies the standard condition. The set of infinite standard Young tableaux of shape $\hat{\lambda}$ is denoted by $\operatorname{SYT}(\hat{\lambda})$.
\item An \emph{infinite semi-standard Young tableau} of shape $\hat{\lambda}\vdash\infty$ is a semi-standard filling of the finite part of $\hat{\lambda}$ by positive integers, plus a non-negative multiplicity for each $h>0$, which is positive only for finitely many values of $h$. The set of those infinite semi-standard Young tableaux will be denoted by $\operatorname{SSYT}(\hat{\lambda})$.
\item An \emph{infinite cocharge tableau} is an infinite semi-standard Young tableau $\hat{C}$ such that all the numbers $h>0$ with positive multiplicities show up in finite part of $\hat{C}$, and if $h>0$ that appears in the finite part of $\hat{C}$, then either its leftmost instance appears in a row lying below an instance of $h-1$ there, or $h-1$ has a positive multiplicity. We write $\operatorname{CCT}(\hat{\lambda})$ for the set of infinite cocharge tableau of shape $\hat{\lambda}\vdash\infty$.
\end{enumerate} \label{tabinf}
\end{defn}
In an infinite semi-standard Young tableau as described in Definition \ref{tabinf}, the first row is considered as if it contains, after infinitely many zeros, each number $h>0$ with the multiplicity given to it, in non-decreasing order. Thus the determination of $\operatorname{CCT}(\hat{\lambda})$ inside $\operatorname{SSYT}(\hat{\lambda})$ is the infinite analogue of part $(iv)$ of Lemma \ref{compctJ}. For $\hat{T}\in\operatorname{SYT}(\hat{\lambda})$ and $i\in\mathbb{N}$ we let $v_{\hat{T}}(i)$, $R_{\hat{T}}(i)$, and $C_{\hat{T}}(i)$ have the same meaning as with finite standard tableaux.

Before giving examples, we show how such objects are constructed.
\begin{lem}
The notions from Definition \ref{tabinf} are limits of finite ones, in the following sense.
\begin{enumerate}[$(i)$]
\item Given any $\lambda \vdash n$, applying the map $\lambda\mapsto\lambda_{+}$ repeatedly produces an infinite Ferrers diagram $\hat{\lambda}\vdash\infty$, whose finite part is based on $\{\lambda_{j}\}_{j\geq2}$. Every $\hat{\lambda}\vdash\infty$ is obtained from some $n$ and $\lambda \vdash n$ in this way.
\item For $T\in\operatorname{SYT}(\lambda)$, a repeated application of $\iota$ from Definition \ref{iotadef} produces an infinite standard Young tableau $\hat{T}\in\operatorname{SYT}(\hat{\lambda})$. Any such infinite tableau can be constructed in this manner for some $n$, $\lambda \vdash n$, and $T\in\operatorname{SYT}(\lambda)$.
\item A repeated action of $\hat{\iota}$ to some element $M\in\operatorname{SSYT}(\lambda)$ produces an infinite semi-standard Young tableau $\hat{M}\in\operatorname{SSYT}(\hat{\lambda})$. Given any element of $\operatorname{SSYT}(\hat{\lambda})$, there are $n$, $\lambda \vdash n$, and $M\in\operatorname{SSYT}(\lambda)$ such that starting this process with them yields the desired $\hat{M}$.
\item By starting with $M=C\in\operatorname{CCT}(\lambda)$, the limit object $\hat{M}=\hat{C}$ from part $(iii)$ is in $\operatorname{CCT}(\hat{\lambda})$. If an element $\hat{C}\in\operatorname{CCT}(\hat{\lambda})$ is obtained from a finite tableau $C\in\operatorname{SSYT}(\lambda)$, then $C\in\operatorname{CCT}(\lambda)$.
\end{enumerate} \label{limtab}
\end{lem}

\begin{proof}
As each application of $\lambda\mapsto\lambda_{+}$ simply lengthens the first row, one direction in part $(i)$ is obvious. Conversely, given $\hat{\lambda}\vdash\infty$, take any $\lambda$ with the required finite part (which is possible by choosing any $\lambda_{1}$ large enough), and it will produce $\hat{\lambda}$ by this process, yielding part $(i)$.

Next, as the standard condition is respected by $\iota$, and is verified by comparing a finite number of boxes at each place, one direction in part $(ii)$ follows. Observing that the finite part of $\hat{T}$ there is the part of $T$ that is below the first row, we conversely take $\hat{\lambda}\vdash\infty$ and $\hat{T}\in\operatorname{SYT}(\hat{\lambda})$, and consider the filling of its finite part. The first row contains the complement, inside $\mathbb{N}$, of the content of the finite part (ordered by the standard condition), and take $n$ to be larger than the maximal entry of this finite part, and let $T$ be the tableau consisting of the boxes $\{v_{\hat{T}}(i)\}_{i=1}^{n}$ with their entries. It is clearly standard, with the finite part of $\hat{T}$ below the first row, and every application of $\iota$ to it adds another box to its location in $\hat{T}$. This establishes part $(ii)$.

Given now an element $M\in\operatorname{SSYT}(\lambda)$, part $(vi)$ of Lemma \ref{rels} implies that the process described in part $(iii)$ pushes the content of the first row of $M$ infinitely to the right, and adds infinitely many zeros. The interpretation of the multiplicity of $h>0$ as representing this many copies of $h$ infinitely to the right (ordered increasingly over $h$) shows that the result is indeed an element of $\operatorname{SSYT}(\hat{\lambda})$, with its finite part being the same as that of $M$ below the first row.

Conversely, take $\hat{M}\in\operatorname{SSYT}(\hat{\lambda})$, and let $n$ be larger than the size of the finite part of $\hat{M}$ plus the sum of the multiplicities. Then set $M$ to be tableau of shape $\lambda \vdash n$, with the finite part of $\hat{M}$ below the first row, and where the first row contains each $h$ with the prescribed multiplicity (ordered by the semi-standard condition), and completed with zeros. As the condition on $n$ implies that all the entries above the second row would contain 0, we get $M\in\operatorname{SSYT}(\lambda)$, and applying $\hat{\iota}$ repeatedly to it would produce $\hat{M}$ as desired, proving part $(iii)$.

Finally, if $M$ in part $(iii)$ is $C\in\operatorname{CCT}(\lambda)$, then so are all its images after finitely many applications of $\hat{\iota}$ (by part $(viii)$ of Lemma \ref{rels}), which thus satisfy the condition from part $(iv)$ of Lemma \ref{compctJ}. It follows that the resulting element of $\operatorname{SSYT}(\hat{\lambda})$ lies in $\operatorname{CCT}(\hat{\lambda})$. Conversely, if $\hat{C}\in\operatorname{CCT}(\hat{\lambda})$ then our construction of $C\in\operatorname{SSYT}(\lambda)$ which will produce it via part $(iii)$ combines with the condition for $\hat{C}$ to be in $\operatorname{CCT}(\hat{\lambda})$ to show that $C$ actually lies in $\operatorname{CCT}(\lambda)$ by part $(iv)$ of Lemma \ref{compctJ} again, and part $(iv)$ follows. This proves the lemma.
\end{proof}

\medskip

\begin{ex}
Take $\lambda$, $T$, $M$, and $C$ as in Example \ref{ctJex}, with $\lambda_{+}$, $\iota T$, $\hat{\iota}M$, and $\hat{\iota}C$ from Example \ref{iotaex}. The corresponding infinite Ferrers diagram $\hat{\lambda}\vdash\infty$ has an infinite first row, plus a second row with two boxes and a third row with a single box. The tableaus $\hat{T}$, $\hat{M}$, and $\hat{C}$ obtained via Lemma \ref{limtab} are then \[\begin{ytableau} 1 & 2 & 6 & 7 & 8 & \cdots \\ 3 & 5 \\ 4 \end{ytableau},\ \ \begin{ytableau} 0 & 0 & 0 & 0 & 0 & \cdots \\ 3 & 4 \\ 4 \end{ytableau},\mathrm{\ and\ }\begin{ytableau} 0 & 0 & 0 & 0 & 0 & \cdots \\ 1 & 2 \\ 2 \end{ytableau},\] where in $\hat{T}$ the entries keep increasing to the right of the first, while in $\hat{M}$ and $\hat{C}$ they remain 0, and we recall that 1 should appear with multiplicity 2 in $\hat{M}$ (as it does in the first row of $M$ and $\hat{\iota}M$), and there are no non-zero multiplicities added in $\hat{C}$. \label{exinftab}
\end{ex}

\begin{rmk}
Recall that Proposition \ref{forstab} did not require $T$ to be standard. Taking the limit of a tableau of shape $\lambda \vdash n$ and content $\mathbb{N}_{n}$ via part $(ii)$ of Lemma \ref{limtab} produces a tableau of shape $\hat{\lambda}\vdash\infty$ and content $\mathbb{N}$, but only those tableaux in which the first row becomes standard after finitely many boxes are obtained in this manner. The set of the latter tableaux is stable under $S_{\infty}$, but not under $S_{\mathbb{N}}$, as considering the image of the standard infinite tableau with a single infinite row under any element of $S_{\mathbb{N}}$ connecting the two expressions from Example \ref{difSinf} to one another shows. This discrepancy plays a role in Remark \ref{onlystd} below. \label{std1row}
\end{rmk}

\begin{rmk}
Recall the original notation from Remark \ref{notation}, with the compatibility from \cite{[Z2]} that involves the evacuation process. One can, in fact, take limits of the operation $\tilde{\iota}$ on a standard Young tableau in order to get an \emph{infinite evacuation tableau} of shape $\hat{\lambda}\vdash\infty$, by filling the finite part with entries of the sort $\infty+i$ (in a standard manner), and requiring the first row to contain $\mathbb{N}$ plus finitely many other entries $\infty+i$ (with different $i$s). Indeed, following the details of $\operatorname{ev}$ from Section 3.9 of \cite{[Sa]}, one can show that it extends to a bijection from $\operatorname{SYT}(\hat{\lambda})$ onto the sets of such tableaux and vice versa, and that one can define a bijection $\operatorname{ct}$ from such tableaux onto $\operatorname{CCT}(\hat{\lambda})$, with inverse $\operatorname{ct}^{-1}$ (generalizing $\operatorname{ct}_{J}$ and $\operatorname{ct}_{J}^{-1}$, from Definition \ref{multisets} to this setting requires a small modification though). However, as these objects are less natural, we stick to the ones from Definition \ref{tabinf}. \label{infev}
\end{rmk}
The equivalent, in terms of Remark \ref{infev}, to the tableaux from Example \ref{exinftab} is with the first row containing $\mathbb{N}$, the second one having $\infty$ and $\infty+2$, and the box in the third row containing $\infty+1$. The infinite $\operatorname{ev}$ and infinite $\operatorname{ct}$ relate this tableau to $\hat{T}$ and $\hat{C}$ from that example respectively.

\medskip

Given an element $\hat{T}\in\operatorname{SYT}(\hat{\lambda})$, we can define the subgroups $R(\hat{T})$, $C(\hat{T})$, and $\tilde{C}(\hat{T})$ just like for finite standard tableaux. We also extend Definition \ref{defSSYT} and some related objects as follows.
\begin{defn}
Let $\hat{M}$ be an element of $\operatorname{SSYT}(\hat{\lambda})$ for some $\hat{\lambda}\vdash\infty$.
\begin{enumerate}[$(i)$]
\item We consider every $h>0$ as showing up in $\hat{M}$ with a multiplicity which is the sum of the multiplicity given in Definition \ref{tabinf} and the number of times that it shows up in the finite part of $\hat{M}$. The \emph{content} of $\hat{M}$ is the non-decreasing sequence in which each $h>0$ shows up with that multiplicity, ignoring the infinitely many zeros.
\item The entry sum $\Sigma(\hat{M})$ of $\hat{M}$ is defined to be the sum of the elements of the content of $\hat{M}$ as a sequence.
\item Given a content $\mu$, we denote by $\operatorname{SSYT}_{\mu}(\hat{\lambda})$ the set of infinite semi-standard Young tableau of shape $\hat{\lambda}$ and content $\mu$. Similarly $\operatorname{SSYT}_{d}(\hat{\lambda})$ stands for the set of elements of $\hat{M}\in\operatorname{SSYT}(\hat{\lambda})$ for which $\Sigma(\hat{M})=d$.
\item Let $k$ be such that the maximal entry in the content of $\hat{M}$ is $k-1$, and for every $1 \leq g<k$, denote by $i_{g}$ the number of entries in the content of $\hat{M}$ that equal at least $k-g$. We define $\operatorname{Dsp}^{c}(\hat{M})$ to be the multi-set obtained from gathering the non-decreasing sequence $(i_{g})_{g=1}^{k-1}$.
\item Given a content $\mu$, we denote by $\tilde{\Lambda}_{\mu}$ the space of all elements of $\tilde{\Lambda}$ that are supported on monomials of content $\mu$, in the sense of Definition \ref{Qxninfd}.
\end{enumerate} \label{infSSYT}
\end{defn}
Also here it is clear that $\operatorname{SSYT}_{d}(\hat{\lambda})$ from Definition \ref{infSSYT} is the disjoint union of $\operatorname{SSYT}_{\mu}(\hat{\lambda})$ over the contents $\mu$ of sum $d$. It is clear that $\hat{M}$ and $\hat{C}$ from Example \ref{exinftab} satisfy $\operatorname{Dsp}^{c}(\hat{M})=\{2,3,3,5\}$ and $\operatorname{Dsp}^{c}(\hat{C})=\{2,3\}$ as well as $\Sigma(\hat{M})=13$ and $\Sigma(\hat{C})=5$, just like $M$ and $C$ from Example \ref{ctJex} and $\hat{\iota}M$ and $\hat{\iota}C$ from Example \ref{iotaex}.

The notions from Definition \ref{infSSYT} have the following properties.
\begin{lem}
Consider $\hat{\lambda}\vdash\infty$, with elements $\hat{T}\in\operatorname{SYT}(\hat{\lambda})$ and $\hat{M}\in\operatorname{SSYT}(\hat{\lambda})$, and let $\lambda \vdash n$, $T\in\operatorname{SYT}(\lambda)$, and $M\in\operatorname{SYT}(\lambda)$ produce them via Lemma \ref{limtab}.
\begin{enumerate}[$(i)$]
\item The group $C(\hat{T})$ equals the image of $C(T)$ inside $S_{\infty}$ and is thus finite, while $R(\hat{T})$ and $\tilde{C}(\hat{T})$ contain $S_{\mathbb{N}}^{(n)}$ for this $n$.
\item The subgroup $R(\hat{T})\cap\tilde{C}(\hat{T})$ is of finite index inside $\tilde{C}(\hat{T})$, and its action on $C(\hat{T})$ yields $\tilde{C}(\hat{T})$ as the semi-direct product. Thus, as in the finite case, the character $\widetilde{\operatorname{sgn}}:\tilde{C}(\hat{T})\to\{\pm1\}$ is well-defined.
\item The multi-set $\operatorname{Dsp}^{c}(\hat{M})$ coincides with $\operatorname{Dsp}^{c}(M)$ from Definition \ref{sets}, and we have $\Sigma(\hat{M})=\Sigma(M)$. We thus have the equality $\Sigma(\hat{M})=\sum_{i\in\operatorname{Dsp}^{c}(\hat{M})}i$.
\item If $\hat{M}=\hat{C}\in\operatorname{CCT}(\hat{\lambda})$, then $\operatorname{Dsp}^{c}(\hat{C})$ is a (finite) subset of $\mathbb{N}$.
\item The part $\tilde{\Lambda}_{d}\subseteq\tilde{\Lambda}$ is the direct sum of $\tilde{\Lambda}_{\mu}$ over all contents $\mu$ of sum $d$.
\end{enumerate} \label{propinf}
\end{lem}

\begin{proof}
The columns of $\hat{T}$ are those of $T$ plus many additional columns of length 1. As the latter columns contribute trivial multipliers to $C(\hat{T})$, the comparison with $C(T)$ follows. Moreover, every integer $m>n$ appears in a column of length 1 inside the first row of $\hat{T}$, as it is added at a given application of $\iota$, so that the other assertion in part $(i)$ follows from the definition of $R(\hat{T})$ and $\tilde{C}(\hat{T})$.

The semi-direct product structure of $\tilde{C}(\hat{T})$ is an immediate consequence of the definition, and the finite index assertion in part $(ii)$ thus follows from the finiteness of $C(\hat{T})$ from part $(i)$. Letting that finite index subgroup act trivially and $C(\hat{T})$ operate by the usual sign produces $\widetilde{\operatorname{sgn}}$ as in part $(iii)$ of Lemma 2.1 of \cite{[Z1]}.

Next, part $(vii)$ of Lemma \ref{rels} shows that the $\operatorname{Dsp}^{c}$-multi-set is unaffected by $\hat{\iota}$, and so is the content up to the multiplicity of 0 (by, e.g., part $(vi)$ there). Hence so is the sum $\Sigma$, and part $(iii)$ is thus a consequence of these considerations via the proof of part $(iii)$ of Lemma \ref{limtab}.

Part $(iv)$ is either a consequence of part $(ii)$ the fact that $\operatorname{Dsp}^{c}(C)$ is a set wherever $C$ is a finite cocharge tableau, or follows by observing that all the numbers between 0 and $k-1$ (for $k$ as in Definition \ref{infSSYT}) show up in $\hat{C}$.

Finally, we decompose every element $F\in\tilde{\Lambda}_{d}$ as the series of its monomials, and gather those of content $\mu$ into a separate expression $F_{\mu}\in\mathbb{Q}\ldbrack\mathbf{x}_{\infty}\rdbrack$, so that $F=\sum_{\mu}F_{\mu}$ (a finite sum since there are finitely many contents of sum $d$). As the action of $S_{\mathbb{N}}$ preserves the contents of monomials, we deduce that $vF=\sum_{\mu}v(F_{\mu})$ for every $v \in S_{\mathbb{N}}$, where $v(F_{\mu})$ is the $\mu$-part $(vF)_{\mu}$ of $vF$ in this decomposition (and $vF\in\tilde{\Lambda}_{d}$, by part $(iii)$ of Proposition \ref{samereps} and the preservation of degrees under this action).

But there is some $n$ such that $vF=F$ for every $v \in S_{\mathbb{N}}^{(n)}$, by Definition \ref{evsym}. It follows that for each content $\mu$ of sum $d$ we get $v(F_{\mu})=(vF)_{\mu}=F_{\mu}$ for any such $v$, and we deduce from that definition that $F_{\mu}$ is in $\tilde{\Lambda}$, and thus in $\tilde{\Lambda}_{\mu}$ by Definition \ref{infSSYT}, as desired for part $(v)$. This proves the lemma.
\end{proof}
One can also define the set $\operatorname{Dsi}(\hat{T})$ for any $\hat{T}\in\operatorname{SYT}(\hat{\lambda})$, and an argument like the one proving part $(i)$ of Lemma \ref{propinf} shows that it equals $\operatorname{Dsi}(T)$ for the corresponding $T\in\operatorname{SYT}(\lambda)$ (and is thus finite). Using appropriate definitions, one can define a finite $\operatorname{Dsi}^{c}$-set of integers for every infinite evacuation tableau from Remark \ref{infev} in a way that Lemma \ref{evDsiAsi} applies for the (extended) evacuation process from this remark, and such that the corresponding $\operatorname{ct}$ map respects the equality between this $\operatorname{Dsi}^{c}$-set and the $\operatorname{Dsi}^{c}$-set of its image in $\operatorname{CCT}(\hat{\lambda})$ (which is a set by part $(iv)$ of Lemma \ref{propinf}).

\medskip

Our next goal is to define our stable objects, where we recall that $e_{r}$ is the $r$th symmetric function, as an element of $\Lambda$ (hence also of $\tilde{\Lambda}$, by Lemma \ref{ringES}). We also recall from part $(iv)$ of Lemma \ref{propinf} that $\operatorname{Dsp}^{c}(\hat{C})$ is a finite set, and thus for every finite multi-set of integers that contain it in the sense of Definition \ref{multisets} we can let $\operatorname{Asp}^{c}_{I}(\hat{C})$ define the complement, as in Definition \ref{sets}.
\begin{defn}
Fix $\hat{\lambda}\vdash\infty$, with $\hat{T}\in\operatorname{SYT}(\hat{\lambda})$ and $\hat{M}\in\operatorname{SSYT}(\hat{\lambda})$.
\begin{enumerate}[$(i)$]
\item For any $i\in\mathbb{N}_{n}$ such that $R_{\hat{T}}(i)>1$, we define $h_{i}$ to be the entry showing up inside the box $v_{\hat{M}}(i)$ of the finite part of $\hat{M}$. In addition we choose, for every $h>0$ having multiplicity $m_{h}>0$ in Definition \ref{tabinf} for $\hat{M}$, a collection of $m_{h}$ integers $i$ that show up in the first row of $\hat{T}$, such that the collections for different values of $h$ are disjoint, and set $h_{i}=h$ for every $i$ in the collection associated with $h$. We then define the monomial $p_{\hat{M},\hat{T}}:=\prod_{i=1}^{n}x_{i}^{h_{i}}\in\mathbb{Q}[\mathbf{x}_{n}]$.
\item We write $\operatorname{st}_{\hat{T}}\hat{M}$ for the stabilizer of $p_{\hat{M},\hat{T}}$ in $R(\hat{T})$, which contains some $S_{\mathbb{N}}^{(n)}$ by Lemma \ref{ringES} and part $(i)$ of Lemma \ref{propinf}.
\item The expression $\sum_{\tau \in R(\hat{T})/\operatorname{st}_{\hat{T}}\hat{M}}p_{\hat{M},\hat{T}}$ is a typically infinite sum of monic monomials. We define the \emph{stable generalized higher Specht polynomial} $F_{\hat{M},\hat{T}}$ to be its image under the operator $\sum_{\sigma \in C(\hat{T})}\operatorname{sgn}(\sigma)\sigma$, which is in $\mathbb{Q}[S_{\infty}]$ by part $(i)$ of Lemma \ref{propinf}. When $\hat{M}=\hat{C}\in\operatorname{CCT}(\hat{\lambda})$, we call $F_{\hat{M},\hat{T}}=F_{\hat{C},\hat{T}}$ a \emph{stable higher Specht polynomial}.
\item Given a finite subset $I\subseteq\mathbb{N}$, write $r_{i}:=|\{j\in\mathbb{N}\setminus\hat{I}\;|\;j<i\}|$ for every $i \in I$.
\item If $\hat{C}\in\operatorname{CCT}(\hat{\lambda})$ and the set $I$ contains the set $\operatorname{Dsp}^{c}(\hat{C})$ from Definition \ref{infSSYT} and part $(iv)$ of Lemma \ref{propinf}, then we define $h_{r}:=\big|\{i\in\operatorname{Asp}^{c}_{I}(\hat{C})\;|\;r_{i}=r\}\big|$ for any $r>0$, and let $\vec{h}_{\hat{C}}^{I}$ be the vector $\{h_{r}\}_{r=1}^{\infty}$. It has finitely many non-zero entries.
\item For such $I$ and $\hat{C}$, we define $F_{\hat{C},\hat{T}}^{I}$ to be the stable higher Specht polynomial $F_{\hat{C},\hat{T}}$ times the element $\prod_{i\in\operatorname{Asp}^{c}_{I}(C)}e_{r_{i}}=\prod_{i=1}^{n}e_{r_{i}}^{h_{i}}\in\Lambda\subseteq\tilde{\Lambda}$.
\item More generally, when $I$ is a multi-set containing $\operatorname{Dsp}^{c}(\hat{C})$ as in Definition \ref{multisets}, and we define $\vec{h}(\hat{C},I)$ to be the characteristic vector of the complement multi-set $\operatorname{Asp}^{c}_{I}(C)$.
\item Given $\hat{C}$ and such a multi-set $I$, we set $F_{\hat{C},\hat{T}}^{I,\mathrm{hom}}$ to be $F_{\hat{C},\hat{T}}$ times the symmetric function $\prod_{i\in\operatorname{Asp}^{c}_{I}(\hat{C})}e_{i}\in\Lambda\subseteq\tilde{\Lambda}$.
\item We write $V_{\hat{M}}$ for the subspace of $\mathbb{Q}\ldbrack\mathbf{x}_{\infty}\rdbrack$ spanned by all the stable generalized higher Specht polynomials $F_{\hat{M},\hat{T}}$ for $\hat{T}\in\operatorname{SYT}(\hat{\lambda})$. When $\hat{M}$ is an infinite cocharge tableau $\hat{C}\in\operatorname{CCT}(\hat{\lambda})$, and $\vec{h}$ is any vector $\{h_{r}\}_{r=1}^{\infty}$ of natural numbers with finitely many non-zero entries, we set $V_{\hat{C}}^{\vec{h}}$ to be the space of elements of $\mathbb{Q}\ldbrack\mathbf{x}_{\infty}\rdbrack$ that are products of an element of $V_{\hat{C}}$ and the symmetric function $\prod_{i=1}^{n}e_{r_{i}}^{h_{i}}\in\Lambda$.
\end{enumerate} \label{infSpecht}
\end{defn}
Note that the construction of $p_{\hat{M},\hat{T}}$ in Definition \ref{infSpecht} involves some choices, and is thus not unique. However, all the different choices produce polynomials that are in the same orbit of $R(\hat{T})$, and thus the stable generalized higher Specht polynomial $F_{\hat{M},\hat{T}}$ is well-defined, as is the expression before the action of $C(\hat{T})$.

We complete Definition \ref{infSpecht} by adding the following properties.
\begin{prop}
Let $\hat{\lambda}$, $\hat{T}$, $\hat{M}$, and $\hat{C}$ be as usual.
\begin{enumerate}[$(i)$]
\item The action of the group $\tilde{C}(\hat{T})$ on stable generalized higher Specht polynomial $F_{\hat{M},\hat{T}}$ is via the character $\widetilde{\operatorname{sgn}}$ from part $(ii)$ of Lemma \ref{propinf}.
\item If $\mu$ is the content of $\hat{M}$, then $F_{\hat{M},\hat{T}}$ lies in $\tilde{\Lambda}_{\mu}$, hence also in $\tilde{\Lambda}_{d}$ for $d=\Sigma(M)$. We therefore also have $V_{\hat{M}}\subseteq\tilde{\Lambda}_{\mu}\subseteq\tilde{\Lambda}_{d}\subseteq\tilde{\Lambda}$.
\item If $I$ contains $\operatorname{Dsp}^{c}(\hat{C})$ and $\vec{h}$ is any vector as in Definition \ref{infSpecht}, then $F_{\hat{C},\hat{T}}^{I}$, $F_{\hat{C},\hat{T}}^{I,\mathrm{hom}}$, and $V_{\hat{C}}^{\vec{h}}$ are also contained in $\tilde{\Lambda}$, and are homogeneous.
\item Take $n$ large enough such that if $\lambda \vdash n$, $T\in\operatorname{SYT}(\lambda)$, and $M\in\operatorname{SSYT}(\lambda)$ produce $\hat{\lambda}$, $\hat{T}$, and $\hat{M}$ via Lemma \ref{limtab}, and every box in the first row of $M$ that lies over another box contains 0. Then $F_{\hat{M},\hat{T}}$ is the unique element of $\tilde{\Lambda}$ on which the action of $\tilde{C}(\hat{T})$ is through $\widetilde{\operatorname{sgn}}$ and in which the substitution $x_{m}=0$ for all $m>n$ yields $F_{M,T}$.
\item Assume that $\hat{M}=\hat{C}\in\operatorname{CCT}(\hat{\lambda})$ and that some multi-set $I$ contains $\operatorname{Dsp}^{c}(\hat{C})$. Then, for $n$ as in part $(iv)$, with the same $T$ and the tableau $M=C\in\operatorname{CCT}(\lambda)$, the series $F_{\hat{C},\hat{T}}^{I,\mathrm{hom}}$ is the only homogeneous element of $\tilde{\Lambda}$ with the given action of $\tilde{C}(\hat{T})$ and such that substituting $x_{m}=0$ in it for every $m>n$ produces $F_{C,T}^{I,\mathrm{hom}}$. When $I$ is a set, the same assertion holds also for $F_{\hat{C},\hat{T}}^{I}$, with the substitution giving $F_{C,T}^{I}$.
\end{enumerate} \label{limSpecht}
\end{prop}

\begin{proof}
Part $(i)$ is proved like part $(i)$ of Theorem \ref{repsSpecht} (or Proposition 2.6 of \cite{[Z2]} and Proposition 2.5 of \cite{[Z3]}), by noting the effect of elements of $C(\hat{T})$ on the operator from $\mathbb{Q}[S_{\infty}]$ in Definition \ref{infSpecht} and the fact that $R(\hat{T})\cap\tilde{C}(\hat{T})$ normalizes $C(\hat{T})$. Part $(ii)$ follows from part $(i)$ via part $(i)$ of Lemma \ref{propinf} and the fact $\widetilde{\operatorname{sgn}}$ is trivial on the subgroup $S_{\mathbb{N}}^{(n)}$ appearing there. Part $(iii)$ is an immediate consequence of part $(ii)$ and Lemma \ref{ringES}, since the multiplying element from Definition \ref{infSpecht} is homogeneous.

Next, Proposition \ref{forstab} shows how $F_{M,T}$ determines $F_{\hat{\iota}M,\iota T}$ when $M$ contains enough zeros, and this condition is valid also for $\hat{\iota}M$. Hence this determination produces, in the limit, a unique homogenenous element of $\mathbb{Q}\ldbrack\mathbf{x}_{\infty}\rdbrack$ that satisfies the condition from part $(i)$ (and hence is an element of $\tilde{\Lambda}_{d}$) in which substituting $x_{m}=0$ for all $m>n$ yields $F_{M,T}$. But doing so in $F_{\hat{M},\hat{T}}$ yields the polynomial constructed from $p_{M,T}$ (which is in the $R(\hat{T})$-orbit of $F_{\hat{M},\hat{T}}$) and taking only the $R(\hat{T})$-images that do not produce monomials that are divisible by any $x_{m}$ with $m>n$. As this produces $F_{M,T}$ by Definition \ref{Spechtdef}, part $(iv)$ is also established.

Finally, we note that $F_{\hat{\iota}C,\iota T}^{I}$ and $F_{\hat{\iota}C,\iota T}^{I,\mathrm{hom}}$ are obtained from $F_{C,T}^{I}$ and $F_{C,T}^{I,\mathrm{hom}}$ respectively, by dividing by the largest symmetric divisor in $\mathbb{Q}[\mathbf{x}_{n}]$ (see Remark 3.3 of \cite{[Z2]} for why this is well-defined), applying the process from Proposition \ref{forstab}, and multiplying by the symmetric divisor with the same notation, but now from $\mathbb{Q}[\mathbf{x}_{n}]$ (and this determines that polynomial by our assumption on $n$). Taking the limit as in part $(iv)$ (which works for $F_{C,T}^{I}$ because $I$ is a set in this case, but not otherwise---see Remark 3.7 of \cite{[Z3]}), and comparing with $F_{\hat{C},\hat{T}}^{I}$ or $F_{\hat{C},\hat{T}}^{I,\mathrm{hom}}$, yields part $(v)$. This proves the proposition.
\end{proof}

\medskip

\begin{rmk}
In the notation from Remark \ref{notation}, we can denote the limits from part $(v)$ of Proposition \ref{limSpecht} by $F_{\hat{w},I}$ and $F_{\hat{w},I}^{\mathrm{hom}}$ for $\hat{w} \in S_{\infty}$, where $\hat{T}$ is the $P$-image of $\hat{w}$ under an infinite RSK algorithm, and $\hat{C}$ is the tableau obtained from its $Q$-image by applying $\operatorname{ev}$ and then $\operatorname{ct}$ in the infinite setting, as Remark \ref{infev} explains. In this case the set $\operatorname{Dsp}^{c}(\hat{C})$ coincides with the descent set of $\hat{w}$, as easily seen by comparing them with their counterparts associated with the finite objects. \label{FwI}
\end{rmk}

\begin{rmk}
One can extend part $(vi)$ of Proposition \ref{limSpecht} to the quotients from Remark \ref{quotSpecht}, but this requires considering the denominators as well---see Corollary \ref{infquot} below. \label{quotinf}
\end{rmk}

\begin{ex}
If $\hat{T}$, $\hat{M}$, and $\hat{C}$ are the infinite tableaux from Example \ref{exinftab}, then $F_{\hat{C},\hat{T}}$ from Definition \ref{infSpecht} is the same as $F_{C,T}$ as well as $F_{\hat{\iota}C,\iota T}$ from Examples \ref{Spechtex} and \ref{quotex} (see Lemma \ref{sameiota} below), and $Q_{\hat{C},\hat{T}}$ from Remark \ref{quotinf} equals $Q_{C,T}$ and $Q_{\hat{\iota}C,\iota T}$ (as in Corollary \ref{infquot} below). The series $F_{\hat{M},\hat{T}}$ is the sum of $F_{M,T}$ from Example \ref{Spechtex}, the counterparts \[\big[(x_{1}x_{3}^{3}x_{4}^{4}-x_{1}x_{4}^{3}x_{3}^{4}-x_{1}^{3}x_{3}x_{4}^{4}+x_{4}^{3}x_{3}x_{1}^{4}+x_{1}^{3}x_{3}^{2}x_{4}-x_{3}^{3}x_{1}^{4}x_{4})(x_{5}^{4}-x_{2}^{4})+\]
\[+(x_{3}^{3}x_{4}^{4}-x_{4}^{3}x_{3}^{4}-x_{1}^{3}x_{4}^{4}+x_{4}^{3}x_{1}^{4}+x_{1}^{3}x_{3}^{2}-x_{3}^{3}x_{1}^{4})(x_{5}^{4}x_{2}-x_{2}^{4}x_{5})\big]\cdot\sum_{i=6}^{\infty}x_{i},\]
of the expressions from Example \ref{quotex}, and the expression \[(x_{3}^{3}x_{4}^{4}-x_{4}^{3}x_{3}^{4}-x_{1}^{3}x_{4}^{4}+x_{4}^{3}x_{1}^{4}+x_{1}^{3}x_{3}^{2}-x_{3}^{3}x_{1}^{4})(x_{5}^{4}-x_{2}^{4})\sum_{i=6}^{\infty}\sum_{j=i+1}^{\infty}x_{i}x_{j}.\]
The quotient $Q_{\hat{M},\hat{T}}$ is the sum of $Q_{M,T}$, the expressions
\[\sum_{i=6}^{\infty}x_{i}\cdot\big[x_{1}x_{3}x_{4}(x_{3}x_{4}+x_{1}x_{4}+x_{1}x_{3})(x_{5}^{3}+x_{5}^{2}x_{2}+x_{5}x_{2}^{2}+x_{2}^{3})+\]
\[+x_{2}x_{5}(x_{3}^{2}x_{4}^{2}+x_{1}^{2}x_{4}^{2}+x_{1}^{2}x_{3}^{2}+x_{1}^{2}x_{3}x_{4}+x_{1}x_{3}^{2}x_{4}+x_{1}x_{3}x_{4}^{2})(x_{5}^{2}+x_{5}x_{2}+x_{2}^{2})\big]\] extending those from Example \ref{quotex}, and finally $\sum_{i=6}^{\infty}\sum_{j=i+1}^{\infty}x_{i}x_{j}$ times \[(x_{3}^{2}x_{4}^{2}+x_{1}^{2}x_{4}^{2}+x_{1}^{2}x_{3}^{2}+x_{1}^{2}x_{3}x_{4}+x_{1}x_{3}^{2}x_{4}+x_{1}x_{3}x_{4}^{2})(x_{5}^{3}+x_{5}^{2}x_{2}+x_{5}x_{2}^{2}+x_{2}^{3}).\] \label{Spechtinf}
\end{ex}

\medskip

We now turn to the spaces from Definition \ref{infSpecht}.
\begin{thm}
Consider $\hat{\lambda}\vdash\infty$, $\hat{M}\in\operatorname{SSYT}(\hat{\lambda})$, and $\hat{C}\in\operatorname{SSYT}(\hat{\lambda})$ as above.
\begin{enumerate}[$(i)$]
\item The space $V_{\hat{M}}$ admits $\{F_{\hat{M},\hat{T}}\;|\;\hat{T}\in\operatorname{SYT}(\hat{\lambda})\}$ as a basis. Similarly, the set $\{F_{\hat{C},\hat{T}}^{I}\;|\;\hat{T}\in\operatorname{SYT}(\hat{\lambda})\}$ forms a basis for $V_{\hat{C}}^{\vec{h}_{\hat{C}}^{I}}$ for every set $I$ containing $\operatorname{Dsp}^{c}(\hat{C})$ from Lemma \ref{propinf}, while if $I$ is a multi-set containing the latter set, then $V_{\hat{C}}^{\vec{h}(\hat{C},I)}$ has $\{F_{\hat{C},\hat{T}}^{I,\mathrm{hom}}\;|\;\hat{T}\in\operatorname{SYT}(\hat{\lambda})\}$ as a basis.
\item Both $V_{\hat{M}}$ and $V_{\hat{C}}^{\vec{h}}$ for every $\vec{h}$ are representations of both $S_{\mathbb{N}}$ and $S_{\infty}$.
\item The isomorphism type of the representation from part $(ii)$ depends only on the shape $\hat{\lambda}$.
\item The representation from part $(ii)$, of either group, is irreducible.
\item The irreducible representations associated with different infinite partitions $\hat{\lambda}\vdash\infty$ are not isomorphic.
\end{enumerate} \label{repsinf}
\end{thm}
In view of Theorem \ref{repsinf}, we can define an isomorphism type $\mathcal{S}^{\hat{\lambda}}$ of irreducible representations of either $S_{\infty}$ or $S_{\mathbb{N}}$, for every $\hat{\lambda}\vdash\infty$. for the relation between the representations from Theorem \ref{repsinf} and the general theory of representations of $S_{\infty}$, see Remark \ref{Thoma} below.

\begin{proof}
For every element $F\in\tilde{\Lambda}$, we write $F^{(n)}$ for its image under the substitution $x_{m}=0$ for all $m>n$. In particular, if $M\in\operatorname{SSYT}(\lambda)$ produces $\hat{M}$ via Lemma \ref{limtab}, and when $T\in\operatorname{SYT}(\lambda)$ yields some $\hat{T}\in\operatorname{SYT}(\hat{\lambda})$ through that lemma, then $F^{(n)}=F_{M,T}$ in case $F=F_{\hat{M},\hat{T}}$. For $C\in\operatorname{CCT}(\lambda)$ giving $\hat{C}$ in this fashion we get, via Definition \ref{infSSYT}, that if $F$ is $F_{\hat{C},\hat{T}}^{I}$ or $F_{\hat{C},\hat{T}}^{I,\mathrm{hom}}$ then $F^{(n)}$ equals $F_{C,T}^{I}$ or $F_{C,T}^{I,\mathrm{hom}}$ respectively.

We also note that $\lambda \vdash n$ yields $\hat{\lambda}$ via Lemma \ref{limtab} then $n$ determines $\lambda$. Therefore if we apply this operation for finitely many different values of elements of $\operatorname{SYT}(\hat{\lambda})$, say $\{\hat{T}_{j}\}_{j=1}^{e}$, then the corresponding tableaux $T_{j}$ are all of the same shape $\lambda \vdash n$.

Now, the fact that these eventually symmetric functions span $V_{\hat{M}}$, $V_{\hat{C}}^{\vec{h}_{\hat{C}}^{I}}$, and $V_{\hat{C}}^{\vec{h}(\hat{C},I)}$ follows directly from Definition \ref{infSpecht}. Assume now that a linear relation exists between them, which is thus of the sort $\sum_{j=1}^{e}a_{j}F_{j}$, where each $F_{j}$ arises from $\hat{M}$, or $\hat{C}$ and $I$, and some $\hat{T}_{j}\in\operatorname{SYT}(\hat{\lambda})$. Then substitute $x_{m}=0$ for all $m>n$ produces $\sum_{j=1}^{e}a_{j}F_{j}^{(n)}$, and our form of $F_{j}^{(n)}$ combines with Theorem \ref{repsSpecht} to show that this is a linear relation among elements which form a basis for a representation of $S_{n}$. We thus have $a_{j}=0$ for every $j$, yielding part $(i)$.

Part $(iv)$ of Proposition \ref{samereps} now shows that it suffices to prove part $(ii)$ for $S_{\infty}$. So take an element of $F$ in $V_{\hat{M}}$ or $V_{\hat{C}}^{\vec{h}}$, and an element $\hat{w} \in S_{\infty}$. We take $n$ to be large enough such that $\hat{w}$ is the image of some element $w \in S_{n}$, and such that the element $\hat{w}F\in\tilde{\Lambda}$ is determined by $(\hat{w}F)^{(n)}$.

But then $F$ is spanned by the basis elements from part $(i)$, and thus $F^{(n)}$ is generated by its images under the substitution $x_{m}=0$ for all $m>n$, which are part of the basis for $V_{M}$ or $V_{C}^{\vec{h}}$ as above. Moreover, it is clear that $(\hat{w}F)^{(n)}$ equals $wF^{(n)}$, hence also lies in the latter representation. We set $\{F_{j}\}_{j=1}^{e}$ to be the part of the basis from part $(i)$ such that $\{F_{j}^{(n)}\}_{j=1}^{e}$ are a basis of $V_{M}$ or $V_{C}^{\vec{h}}$ via Theorem \ref{repsSpecht} in this manner, and write $wF^{(n)}$ as a linear combination $\sum_{j=1}^{e}a_{j}F_{j}^{(n)}$ as above.

But $(\hat{w}F)^{(n)}=wF^{(n)}$ determines $\hat{w}F$ by our assumption, and the element $\sum_{j=1}^{e}a_{j}F_{j}$ of $V_{\hat{M}}$ or $V_{\hat{C}}^{\vec{h}}$ also has the same image under our substitution. It follows that $\hat{w}F=\sum_{j=1}^{e}a_{j}F_{j}$, proving part $(ii)$.

For part $(iii)$, we recall that in the finite setting, the isomorphism from Theorem \ref{repsSpecht} between $V_{M}$ or $V_{C}^{\vec{h}}$ and $\mathcal{S}^{\lambda}$ is based on taking $F_{M,T}$ or the appropriate multiple of $F_{C,T}$ to the basis element of $\mathcal{S}^{\lambda}$ that is associated with $T$. Hence if we replace $M$ or $C$ and $I$ by other element of $\operatorname{SSYT}(\lambda)$, then the map changing this index but leaves every $\hat{T}$ invariant respects the action of $S_{n}$.

Therefore we consider the proof of part $(ii)$, and change $\hat{M}$, or $\hat{C}$ and $I$, to another element of $\operatorname{SSYT}(\hat{\lambda})$. If we change this index but leave every $\hat{T}$ invariant, then after the substitution $F \mapsto F^{(n)}$ we obtain a map that respects the operation of $S_{n}$. But we saw that for large enough $n$, the image determine a given element, so that this isomorphism also respects the action of $S_{\infty}$. Combining this with the proof of part $(iv)$ of Proposition \ref{samereps} yields part $(iii)$ also for $S_{\mathbb{N}}$.

For part $(iv)$, take an element $F\neq0$ in our representation, and we have to show that any other element $G$ of that representation is generated by the images of $F$. We let $n$ be large enough so that both $F$ and $G$ are stabilized by $S_{\mathbb{N}}^{(n)}$, so in particular $F^{(n)}\neq0$ and $G$ is determined by $G^{(n)}$. Then the irreducibility of $V_{M}$ or $V_{C}^{\vec{h}}$ implies that $G^{(n)}$ is generated by the images of $F^{(n)}$ under the elements of $S_{n}$, namely we have $G^{(n)}=\sum_{w \in S_{n}}a_{w}wF^{(n)}$ with some scalar coefficients $\{a_{w}\}_{w \in S_{n}}$.

Given $w \in S_{n}$, let $\hat{w}$ be the image of $w$ in $S_{\infty}$, and then we saw above that $wF^{(n)}=(\hat{w}F)^{(n)}$ for every such $w$. Therefore our expression for $G^{(n)}$ becomes $\sum_{w \in S_{n}}(a_{w}\hat{w}F)^{(n)}$, and since the image of $S_{n}$ inside $S_{\mathbb{N}}$ normalizes $S_{\mathbb{N}}^{(n)}$, we deduce that $\hat{w}F$ is also stabilized by the latter subgroup for every $w \in S_{n}$. Therefore $\sum_{w \in S_{n}}a_{w}\hat{w}F$ is the only $S_{\mathbb{N}}^{(n)}$-invariant element of $\tilde{\Lambda}$ whose image under our substitution is $\sum_{w \in S_{n}}(a_{w}\hat{w}F)^{(n)}$.

But we saw that $G$ is also an element with the same properties, by the assumption on $n$ and the equality involving $G^{(n)}$. We thus get $G=\sum_{w \in S_{n}}a_{w}\hat{w}F$, so that $G$ is indeed generated by the images of $F$ under elements of $S_{\infty}$, proving the irreducibility of our representation over $S_{\infty}$, hence also over the larger group $S_{\mathbb{N}}$, as part $(iv)$ requires.

Finally, take a representation $V_{\hat{M}}$ associated with $\hat{M}\in\operatorname{SSYT}(\hat{\lambda})$ for some $\hat{\lambda}\vdash\infty$, and let $n$ be large enough such that the finite tableau $M\in\operatorname{SSYT}(\lambda)$, with $\lambda \vdash n$, that yields $\hat{M}$ through Lemma \ref{limtab} has zeros in the first row all over the finite part of $\hat{M}$. All the stable higher Specht polynomials $F_{\hat{M},\hat{T}}$ arising from $\hat{T}\in\operatorname{SYT}(\hat{\lambda})$ in which the numbers $m>n$ are in the first row are invariant under $S_{\mathbb{N}}^{(n)}$ (since it is contained in $\tilde{C}(\hat{T})$). For a tableau $\hat{T}\in\operatorname{SYT}(\hat{\lambda})$ not satisfying this condition, the latter subgroup relates the corresponding element $F_{\hat{M},\hat{T}}$ to infinitely many others, hence the part of $V_{\hat{M}}$ that is invariant under $S_{\mathbb{N}}^{(n)}$ is spanned by the former stable higher Specht polynomials.

We now note that a tableau $\hat{T}$ satisfies the condition from the previous paragraph if and only if it is the image of some $T\in\operatorname{SSYT}(\lambda)$ for the same $n$ and $\lambda$. Hence this part of $V_{\hat{M}}$ is the image of $V_{M}$ under the map from part $(iv)$ of Proposition \ref{limSpecht} (extended linearly), and this linear map respects the action of $S_{n}$ on $V_{M}$ and on $V_{\hat{M}}$ (via the embedding into $S_{\infty} \subseteq S_{\mathbb{N}}$). It follows that the part of $V_{\hat{M}}$ in question is isomorphic to $V_{M}$ as a representation of $S_{n}$.

But this implies that this part of $V_{\hat{M}}$ determines, as the isomorphism type of $V_{M}$, the partition $\lambda \vdash n$ yielding $\hat{\lambda}=\operatorname{sh}(\hat{M})$. This allows us to determine $\hat{\lambda}$ from $V_{\hat{M}}$ using only its isomorphism type, which establishes part $(v)$. This completes the proof of the theorem.
\end{proof}

\medskip

\begin{ex}
Let $\hat{\lambda}$ be with an infinite row and a second row containing two boxes, and take $\hat{M}\in\operatorname{SSYT}(\hat{\lambda})$ with second row containing two instances of 1 and 1 having multiplicity 1. If $\hat{T}\in\operatorname{SYT}(\hat{\lambda})$ contains the entries $k<l$ in the second row, with 1 above $k$ and $1<j<l$ above $l$ (with $j \neq k$), then $F_{\hat{M},\hat{T}}$ equals $(x_{k}-x_{1})(x_{l}-x_{j})\sum_{i}x_{i}$, where $i$ runs over $\mathbb{N}\setminus\{1,j,k,l\}$. If $n=5$, then the 5 tableaux in which $l\leq5$ arise from the 5 elements of $\operatorname{SYT}(\lambda)$ for $\lambda=32\vdash5$, and indeed all the indices $i>5$ show up in columns of length 1 in $\hat{T}$ and thus show up in a symmetric manner in $F_{\hat{M},\hat{T}}$, showing its invariance under $S_{\mathbb{N}}^{(5)}$. But if we take $l>5$, for example the case where $j=2$, $k=4$, and $l=7$, then $S_{\mathbb{N}}^{(5)}$ takes this $\hat{T}$ to any other tableau with these $j$ and $k$ and $l>5$. Due to the basis property, every $S_{\mathbb{N}}^{(5)}$-invariant element of $V_{\hat{M}}$ that contains this $F_{\hat{M},\hat{T}}$ would have to contain the basis element obtained by replacing $l=7$ by any other value $l>5$ with the same coefficient, and thus cannot be spanned by finitely many basis elements. \label{restofin}
\end{ex}
Note that substituting $x_{m}=0$ for $m>5$ in the expressions with $l\leq5$ in Example \ref{restofin} produces the basis elements for $V_{M}$ (where $M$ is the tableau of shape $\lambda$ yielding $M$ through Lemma \ref{limtab}), but doing so in the explicit case with $l=7$ there (or in any of its $S_{\mathbb{N}}^{(5)}$-images) produces $-(x_{4}-x_{1})x_{2}(x_{3}+x_{5})$, which is not in that representation.

\begin{rmk}
In Definition \ref{Spechtdef}, we allowed $T$ not to be standard, and in Theorem \ref{repsSpecht} reduced to standard $T$ in order to get a basis. One can obtain similar results with tableaux $\hat{T}$ of shape $\hat{\lambda}\vdash\infty$ and content $\mathbb{N}$, though some results are only valid for such tableaux in which the first row is eventually standard (see Remark \ref{std1row}). We thus stick to standard tableaux in the infinite setting, to avoid cumbersome formulations which do not increase our representations. \label{onlystd}
\end{rmk}

\medskip

The limits of the stable representations from Definition \ref{RnIdef} and Proposition \ref{plus1iota} are the following ones.
\begin{defn}
Consider the following objects.
\begin{enumerate}[$(i)$]
\item For any finite subset $I\subseteq\mathbb{N}$, we define the representation $R_{\infty,I}^{0}$ to be the sum $\sum_{\hat{\lambda}\vdash\infty}\sum_{\hat{C}\in\operatorname{CCT}(\hat{\lambda}),\ \operatorname{Dsp}^{c}(\hat{C}) \subseteq I}V_{\hat{C}}^{\vec{h}_{\hat{C}}^{I}}$.
\item If $I$ is any finite multi-set of positive integers, then $R_{\infty,I}^{\mathrm{hom},0}$ is the homogeneous representation $\sum_{\hat{\lambda}\vdash\infty}\sum_{\hat{C}\in\operatorname{CCT}(\hat{\lambda}),\ \operatorname{Dsp}^{c}(\hat{C}) \subseteq I}V_{\hat{C}}^{\vec{h}(\hat{C},I)}$.
\item Given a finite multi-set of non-negative integers, we write $R_{\infty,I}^{\mathrm{hom},0}$ for the representation from part $(ii)$ that is obtained by removing the multiplicity of 0 from $I$.
\end{enumerate} \label{RinfIdef}
\end{defn}
Note that parts $(ii)$ and $(iii)$ of Definition \ref{RinfIdef} are very close, since the multiplicity of 0 affects neither the containment of $\operatorname{Dsp}^{c}(\hat{C})$ nor the vector $\vec{h}(\hat{C},I)$ from Definition \ref{infSpecht} (and indeed, multiplying by a power of $e_{0}=1$ has no effect). We will use both conventions below. The reason for the superscript 0 will be explained in remark \ref{RinfIgen} below.

The representations from Definition \ref{RinfIdef} have the following properties.
\begin{lem}
The sums in Definition \ref{RinfIdef} are finite and direct, thus producing completely reducible representations of both $S_{\mathbb{N}}$ and $S_{\infty}$, with $R_{\infty,I}^{\mathrm{hom},0}$ being homogeneous of degree $\sum_{i \in I}i$. \label{RinfIprop}
\end{lem}
In the terminology from Remark \ref{FwI}, Lemma \ref{RinfIprop} implies that the representation $R_{\infty,I}^{0}$ and $R_{\infty,I}^{\mathrm{hom},0}$ admit $\{F_{\hat{w},I}\}_{\hat{w}}$ and $\{F_{\hat{w},I}^{\mathrm{hom}}\}_{\hat{w}}$ respectively as bases, where $\hat{w}$ runs over the elements of $S_{\infty}$ whose descent set is contained in $I$.

\begin{proof}
In Definition \ref{RnIdef} and the cited references we saw that $R_{n,I}$ and $R_{n,I}^{\mathrm{hom}}$, for finite $n$, are direct sums of representations, with the latter being homogeneous of the asserted degree. Proposition \ref{plus1iota} shows that for large enough $n$, applying $\hat{\iota}$ to the cocharge tableau in the subscript produces a bijection between the representations showing up for $n$ and those appearing for $n+1$, so that the number of representations participating in them (which is the same for $R_{n,I}$ and for $R_{n,I}^{\mathrm{hom}}$) stabilizes for large enough $n$.

Combining this with Lemma \ref{limtab} shows that the sums in Definition \ref{RinfIdef} are finite, with the number of representations being the value at which the number from the previous paragraph stabilizes. Take now an element of $V_{\hat{C}}^{\vec{h}_{\hat{C}}^{I}}$ or $V_{\hat{C}}^{\vec{h}(\hat{C},I)}$ for each $\hat{C}$ participating in the sum for $I$, and assume that they are linearly dependent. Recall the bases from Theorem \ref{repsinf}, this produces a linear relation of the form $\sum_{\hat{C}}\sum_{j=1}^{e_{\hat{C}}}a_{\hat{C},j}F_{\hat{C},\hat{T}_{j}}=0$.

As in the proof of that theorem, we take $n$ such that the stabilization of the number of representations occurs before it, and such that each $\hat{C}$ and each $\hat{T}_{j}$ arise via Lemma \ref{limtab} from finite tableaux $C$ and $T_{j}$ respectively, whose shapes are of size $n$. Applying the substitution $x_{m}=0$ for every $m>n$, and recalling that this takes each $F_{\hat{C},\hat{T}_{j}}$ to $F_{\hat{C},\hat{T}_{j}}^{(n)}=F_{C,T_{j}}$, our linear relation becomes one among a basis for the direct sum from Definition \ref{RnIdef}. As such relations do not exist, we get $a_{\hat{C},j}=0$ for every $\hat{C}$ and $j$, so the sum in question is indeed direct.

The homogeneity of $R_{\infty,I}^{\mathrm{hom},0}$ follows from that of $R_{n,I}^{\mathrm{hom}}$ via the same argument (or is quite clear from the definition), and the complete reducibility follows from the sum being finite and direct. This proves the lemma.
\end{proof}
Theorem 3.20 and Proposition 3.25 of \cite{[Z2]} and Theorem 2.19 of \cite{[Z3]} identify, for finite $n$, the representations $R_{n,I}$ and $R_{n,I}^{\mathrm{hom}}$ with ones of the sort $\mathbb{Q}[\mathcal{OP}_{n,I}]$ arising from the action of $S_{n}$ on a finite set $\mathcal{OP}_{n,I}$ of ordered partitions of $\mathbb{N}_{n}$ into sets of prescribed sizes. As $S_{\mathbb{N}}$ and $S_{\infty}$ are infinite, and so is the analogous set $\mathbb{Q}[\mathcal{OP}_{\infty,I}]$ of ordered partitions of $\mathbb{N}$ into finitely many sets with all but the last one being finite, and may encounter issues like those investigated in Theorems \ref{compred}, \ref{filtQxinf}, and \ref{filtrations} below, we leave investigating possible connections with the representations from Definition \ref{RinfIdef} and Lemma \ref{RinfIprop} for future research.

\begin{ex}
Consider the multi-set $I:=\{1,3,3\}$ from Example \ref{RnIex}, and the set $\hat{I}=\{1,3\}$. If, given $\hat{M}\in\operatorname{SSYT}(\hat{\lambda})$ with multiplicities $\{m_{h}\}_{h>0}$, we indicate in $V_{\hat{M}}$ the multiplicities by putting each $h>0$ in the superscript $m_{h}$ times, then $R_{\infty,I}^{\mathrm{hom},0}$ is given by \[V_{000\cdots}e_{1}e_{3}^{2} \oplus V_{\substack{000\cdots \\ 1\hphantom{11\cdots}}}e_{3}^{2} \oplus V_{\substack{000\cdots \\ 1\hphantom{11\cdots}}}^{11}e_{1}e_{3} \oplus V_{\substack{000\cdots \\ 11\hphantom{1\cdots}}}^{1}e_{1}e_{3} \oplus V_{\substack{000\cdots \\ 111\hphantom{\cdots}}}e_{1}e_{3} \oplus V_{\substack{000\cdots \\ 1\hphantom{11\cdots} \\ 2\hphantom{22\cdots}}}^{1}e_{3} \oplus V_{\substack{000\cdots \\ 11\hphantom{1\cdots} \\ 2\hphantom{22\cdots}}}e_{3},\] and we have \[R_{\infty,\hat{I}}=V_{000\cdots}e_{1} \oplus V_{\substack{000\cdots \\ 1\hphantom{11\cdots}}}e_{1} \oplus V_{\substack{000\cdots \\ 1\hphantom{11\cdots}}}^{11} \oplus V_{\substack{000\cdots \\ 11\hphantom{1\cdots}}}^{1} \oplus V_{\substack{000\cdots \\ 111\hphantom{\cdots}}} \oplus V_{\substack{000\cdots \\ 1\hphantom{11\cdots} \\ 2\hphantom{22\cdots}}}^{1} \oplus V_{\substack{000\cdots \\ 11\hphantom{1\cdots} \\ 2\hphantom{22\cdots}}}.\] \label{RinfIex}
\end{ex}

\medskip

Recall the direct sums from Theorem \ref{Rnkdecom}, which we now wish to generalize to the infinite setting. Note that for fixed $k$, the indices $r$ with $n+1-k \leq r \leq n$ all tend to $\infty$ with $n$, so that in this sense the limit of $R_{n,k}$ as $n\to\infty$ is ``the same'' as that of $R_{n,k,0}$. However, these limits should be quotients of $\tilde{\Lambda}$, where symmetric functions become separate from polynomials (see, e.g., Example \ref{exd1} below for this separation). The quotient in the infinite case, and the analogue of the sets $H_{C}^{k}$ from Theorem 3.4 of \cite{[Z2]} and $H_{C}^{k,0}$ from Theorem 2.12 of \cite{[Z3]} in this setting, are thus the following ones.
\begin{defn}
Let $k\geq1$ be any integer.
\begin{enumerate}[$(i)$]
\item The symbol $R_{\infty,k}$ stands for the quotient of $\tilde{\Lambda}$ by the monomials $x_{i}^{k}$, $i\geq1$, and by the power sums $p_{l}=\sum_{i=1}^{\infty}x_{i}^{l}\in\Lambda$ for $l \geq k$.
\item As $R_{\infty,k}$ is a graded ring, we write it as $\bigoplus_{d=0}^{\infty}R_{\infty,k,d}$, where $R_{\infty,k,d}$ for the part of $R_{\infty,k}$ that is homogeneous of degree $d$.
\item Given $\hat{\lambda}\vdash\infty$ and $\hat{C}\in\operatorname{CCT}(\hat{\lambda})$, the symbol $H_{\hat{C}}^{k}$ stands for the set of vectors $\vec{h}=\{h_{r}\}_{r=1}^{\infty}$ of non-negative integers such that $\sum_{r=1}^{\infty}h_{r}<k-|\operatorname{Dsp}^{c}(\hat{C})|$.
\end{enumerate} \label{defRinfk}
\end{defn}
Indeed, if $x_{i}^{l}=0$ for every $i$ in a quotient, it makes sense that $p_{l}=0$ as well. The quotient $R_{\infty,k}$ from Definition \ref{defRinfk} is a quotient of the extension to $\mathbb{Q}$ of the algebra considered in Proposition 7.3 of \cite{[PR]}, and the set $H_{\hat{C}}^{k}$ is, of course, empty in case $|\operatorname{Dsp}^{c}(\hat{C})| \geq k$.

\begin{rmk}
Simple relations inside $\Lambda$ show that any monomial symmetric function involving an exponent which is $k$ or more vanishes in the quotient $R_{\infty,k}$ from Definition \ref{defRinfk}. This combines with the presentations of any monomial symmetric function in terms of products of elementary symmetric functions to show that products of less than $k$ elementary symmetric functions form a basis for the image of $\Lambda$ inside $R_{\infty,k}$. This is in correspondence with the decompositions of the representation $R_{\infty,k}^{0}$ from Theorem \ref{decomRinf} below into representations involving multipliers from $\Lambda$, where these multipliers are products of less than $k$ elementary symmetric functions. \label{propkeis}
\end{rmk}

\medskip

In order to consider Theorem \ref{Rnkdecom} in the infinite case, we first prove an analogue of Lemma 3.17 of \cite{[Z2]} and Lemma 2.16 of \cite{[Z3]} in this setting.
\begin{lem}
Consider some $k\geq1$, $\hat{\lambda}\vdash\infty$, and $\hat{C}\in\operatorname{CCT}(\hat{\lambda})$.
\begin{enumerate}[$(i)$]
\item The map taking any multi-set $I$ of size $k-1$ that contains $\operatorname{Dsp}^{c}(\hat{C})$ in the sense of Definition \ref{multisets} to the vector $\vec{h}(C,I)$ from Definition \ref{Spechtdef} is a bijection from this collection of multi-sets onto the set $H_{\hat{C}}^{k}$ from Definition \ref{defRinfk}.
\item We have a containment $H_{\hat{C}}^{k-1} \subseteq H_{\hat{C}}^{k}$, and $I$ maps into that subset via part $(i)$ if and only if it contains 0.
\item The map sending a subset $I\subseteq\mathbb{N}$ contains $\operatorname{Dsp}^{c}(\hat{C})$ to the vector $\vec{h}_{C}^{I}$ is also a bijection between the collection of such sets and $H_{\hat{C}}^{k}$.
\end{enumerate} \label{bijvecs}
\end{lem}

\begin{proof}
The argument again follows the proofs of Lemma 3.17 of \cite{[Z2]} and Lemma 2.16 of \cite{[Z3]}. Given $I$ as in either part $(i)$ or part $(iii)$, if $\vec{h}$ is the associated vector then $\sum_{r=1}^{\infty}h_{r}$ is, by Definition \ref{Spechtdef}, the number of elements of $\operatorname{Asp}^{c}_{I}(\hat{C})$ from Definition \ref{sets} (which is a set in part $(iii)$, but in part $(i)$ may be a multi-set) yielding a non-zero contribution. This sum is thus bounded by the size of this set or multi-set, and as this size is $k-1-|\operatorname{Dsp}^{c}(\hat{C})|$, the images of both maps are in $H_{\hat{C}}^{k}$.

Consider now an element $\vec{h}=\{h_{r}\}_{r=1}^{n} \in H_{\hat{C}}^{k}$, and thus its coordinate sum is smaller than $k-|\operatorname{Dsp}^{c}(\hat{C})|$, and we complete it by defining $h_{0}\geq0$ in such a way that the equality $\sum_{r=0}^{n}h_{r}=k-1-|\operatorname{Dsp}^{c}(\hat{C})|$ holds. We need to find a unique multi-set $I$ satisfying $\vec{h}(\hat{C},I)=\vec{h}$ for part $(i)$, and for proving part $(iii)$ we have to find a set $I\subseteq\mathbb{N}$ for which $\vec{h}_{\hat{C}}^{I}=\vec{h}$.

But Definitions \ref{Spechtdef} and \ref{sets} imply that $\vec{h}(\hat{C},I)=\vec{h}$ is equivalent to the $r$th entry in the completion of $\vec{h}$ being the multiplicity with which $r$ appears in $\operatorname{Asp}^{c}_{I}(\hat{C})$, which determines the latter multi-set uniquely. As its size is thus $k-1-|\operatorname{Dsp}^{c}(\hat{C})|$, and $I$ is determined by the latter multi-set by simply adding $\operatorname{Dsp}^{c}(\hat{C})$ as a multi-set, and thus has size $k-1$, part $(i)$ is thus established.

The condition in Definition \ref{defRinfk} becomes stronger when $k$ is replaced by $k-1$, and $\vec{h}(\hat{C},I)$ was seen in the proof of part $(i)$ (or in Definition \ref{Spechtdef}) to be the characteristic function of $\operatorname{Asp}^{c}_{I}(\hat{C})$, with the multiplicity of 0 appropriately modified. This proves the containment in part $(ii)$, and as the multi-sets corresponding to an element of $H_{\hat{C}}^{k-1} \subseteq H_{\hat{C}}^{k}$ via part $(i)$ for $k$ and for $k-1$ differ by one instance of 0, so does the other assertion there.

Finally, we denote $D:=\operatorname{Dsp}^{c}(C)\subseteq\mathbb{N}$ as well as $d:=|D|+1$, and we recall from Definition \ref{infSpecht} that $r_{i}$ only increases with $i$ and depends on the elements that are smaller than $i$ (whether they are in $I$ or not). We need to find the subset $A\subseteq\mathbb{N} \setminus D$, of size $k-d$, whose entries yield the $r_{i}$'s according to the completion of $\vec{h}$ (which is indeed of entry sum $k-d$), observe that there is only one such subset, and then our $I$ is $A \cup D$.

We do this as in the proof of Lemma 3.17 of \cite{[Z1]}, by recalling that the $r_{i}$'s consist of $h_{0}$ instances of 0, then $h_{1}$ instances of 1, and so forth until we reach the maximal index $r$ for which $h_{r}>0$. Then the only set $A$ with the desired properties consists of the smallest $h_{0}$ elements of $\mathbb{N} \setminus D$ (or skip those if $h_{0}=0$), then skips one and takes the next $h_{1}$ elements there (unless $h_{1}=0$), and so on, finishing with we pass over all the positive entries of $\vec{h}$ (or equivalently when we have gathered $k-d$ elements). With this set $I$, and only with it, we have $\vec{h}_{\hat{C}}^{I}=\vec{h}$, as part $(iii)$ requires. This completes the proof of the lemma.
\end{proof}

\medskip

We can now prove the analogue of Theorem \ref{Rnkdecom} for the infinite case.
\begin{thm}
Denote by $R_{\infty,k}^{0}$ the sum $\sum_{\hat{\lambda}\vdash\infty}\bigoplus_{\hat{C}\in\operatorname{CCT}(\hat{\lambda})}\sum_{\vec{h} \in H_{\hat{C}}^{k}}V_{\hat{C}}^{\vec{h}}$.
\begin{enumerate}[$(i)$]
\item The sum defining $R_{\infty,k}^{0}$ is direct inside $\tilde{\Lambda}$, and even in $R_{\infty,k}$.
\item The sum of the representations $R_{\infty,I}^{0}$ over all sets $I\subseteq\mathbb{N}$ of size $k-1$ is also direct in $\tilde{\Lambda}$ and also in $R_{\infty,k}$, and equals $R_{\infty,k}^{0}$.
\item Taking the sum of $R_{\infty,I}^{\mathrm{hom},0}$ over all the multi-sets of size $k-1$ (which may include zeros) also yields a direct sum in $\tilde{\Lambda}$ and $R_{\infty,k}$, which equals $R_{\infty,k}^{0}$.
\item Viewing $R_{\infty,k}$ as a quotient of the finer quotient $R_{\infty,k+1}$, the sum producing $R_{\infty,k}^{0}$ is partial to the one yielding $R_{\infty,k+1}^{0}$. The same holds for the expressions from part $(iii)$ for these two values.
\item By setting $R_{\infty,k,d}^{0}:=R_{\infty,k}^{0} \cap R_{\infty,k,d}$ for any degree $d$, it decomposes as the finite direct sum of those $R_{\infty,I}^{\mathrm{hom},0}$'s from part $(iii)$ for which the multi-set $I$ satisfies $\sum_{i \in I}i=d$.
\end{enumerate} \label{decomRinf}
\end{thm}

\begin{proof}
We first observe that part $(iii)$ of Lemma \ref{bijvecs} shows that decomposing the terms in the sum from part $(ii)$ according to Definition \ref{RinfIdef} yields that of part $(i)$, and part $(i)$ of that lemma relates the sum from part $(iii)$ to that of part $(i)$ in the same manner. Hence the first three parts are equivalent.

For proving these parts, we take a linear relation among the $R_{\infty,I}^{0}$'s (or the $R_{\infty,I}^{\mathrm{hom},0}$'s), and recall that Lemma \ref{RinfIprop} and Theorem \ref{repsinf} transform it into an expression of the sort $\sum_{j=1}^{e_{\hat{C}}}a_{j}F_{\hat{C}_{j},\hat{T}_{j}}^{I}=0$ (or $\sum_{j=1}^{e_{\hat{C}}}a_{j}F_{\hat{C}_{j},\hat{T}_{j}}^{I,\mathrm{hom}}=0$). As in the proof of these results, we can take $n$ large enough so that all of the (finitely many) indices are obtained from finite tableaux of size $n$ through Lemma \ref{limtab}, and obtain a linear relation among the expressions from Definition \ref{RnIdef}. As Theorem \ref{Rnkdecom} implies that no such linear relations exist (also in $R_{n,k}$ or $R_{n,k,0}$), we get $a_{j}=0$ for all $j$, thus yielding parts $(i)$, $(ii)$, and $(iii)$.

The first assertion in part $(iv)$ follows from part $(ii)$ of Lemma \ref{bijvecs}, and the second one is a consequence of that part together with the fact that $R_{\infty,I}^{\mathrm{hom},0}$ from Definition \ref{RinfIdef} is unaffected by adding or removing 0 from $I$. Part $(v)$ is an immediate consequence of part $(iii)$ and the degree of homogeneity of $R_{\infty,I}^{\mathrm{hom},0}$. This completes the proof of the theorem.
\end{proof}

\begin{rmk}
Part $(iv)$ of Theorem \ref{decomRinf} presents the direct sum living inside $R_{\infty,k+1}$ as the sum of one that projects bijectively onto $R_{\infty,k}$ plus some additional terms. These terms, arising from multi-sets of size $k$ that do not contain 0, do not vanish when projected onto $R_{\infty,k}$ in general. Indeed, they are products of stable higher Specht polynomials with products of elementary symmetric functions, and as Remark \ref{propkeis} states, such expressions will be congruent in $R_{\infty,k}$ to linear combinations of expressions involving products of less elementary symmetric functions, and not to 0. We will see in part $(iii)$ of Theorem \ref{images} below how stable generalized higher Specht polynomials yields expressions that do vanish under these projections. \label{monfunc}
\end{rmk}

We obtain the following infinite analogue of Corollary 2.27 of \cite{[Z3]}.
\begin{cor}
Take $\hat{\lambda}\vdash\infty$ and $d\geq0$, and consider pairs involving an element $\hat{C}\in\operatorname{CCT}(\hat{\lambda})$ combined with a multi-set $I$, involving only positive integers, that sum to $d$ and contain, in the sense of Definition \ref{multisets}, the set $\operatorname{Dsp}^{c}(\hat{C})$. \begin{enumerate}[$(i)$]
\item Each such pair contributes an irreducible representation inside $\tilde{\Lambda}_{d}$ that is isomorphic to $\mathcal{S}^{\hat{\lambda}}$, and the sum of these representations is finite and direct.
\item The sum of these representations over $\hat{\lambda}\vdash\infty$ is also finite and direct.
\end{enumerate} \label{multCI}
\end{cor}

\begin{proof}
Take some $k>d$, and note that the relations by which we divide $\tilde{\Lambda}$ in order to get $R_{\infty,k}$ in Definition \ref{defRinfk} do not affect the homogeneous part of degree $d$. Hence $\tilde{\Lambda}_{d}$ maps bijectively onto the part of $R_{\infty,k}$ that is homogeneous of degree $d$, and we observe that part $(iii)$ of Theorem \ref{decomRinf} concerns homogeneous representations.

We thus restrict attention to those multi-sets $I$ satisfying $\sum_{i \in I}i=d$, so that $R_{\infty,I}^{\mathrm{hom},0}$ is contained in $\tilde{\Lambda}_{d}$ or in the said part of $R_{\infty,k}$, and recall from Definition \ref{RinfIdef} that removing the number of zeros in $I$ does not affect $R_{\infty,I}^{\mathrm{hom},0}$. Since every multi-set of entry sum $d<k$ has at most $d \leq k-1$ non-zero entries, and can be completed to multi-set of size $k-1$, taking the sum of the decompositions of the corresponding $R_{\infty,I}^{\mathrm{hom},0}$ yields the asserted direct sum.

The fact that the sums associated with different $\hat{\lambda}\vdash\infty$ land in different isotypical components (by part $(v)$ of Theorem \ref{repsinf} implies the directness of the second sum, and the finiteness of both follows from Lemma \ref{RinfIprop} and the fact that there are finitely many multi-sets of positive integers summing to $d$. This proves the corollary.
\end{proof}

\medskip

We also prove an infinite analogue of Theorem \ref{FMTdecom}.
\begin{prop}
The sum $\tilde{\Lambda}_{d}^{0}:=\sum_{\hat{\lambda}\vdash\infty}\sum_{\hat{M}\in\operatorname{SSYT}_{d}(\hat{\lambda})}V_{\hat{M}}$ is finite and direct inside $\tilde{\Lambda}_{d}$ for every $d\geq0$, and the partial sum $\tilde{\Lambda}_{\mu}^{0}:=\bigoplus_{\hat{\lambda}\vdash\infty}\bigoplus_{\hat{M}\in\operatorname{SSYT}_{\mu}(\hat{\lambda})}V_{\hat{M}}$, for any a content $\mu$ of sum $d$, is its intersection with $\tilde{\Lambda}_{\mu}$. \label{ExtVMlim}
\end{prop}

\begin{proof}
As in the proofs of Theorem \ref{repsinf}, Lemma \ref{RinfIprop}, and Theorem \ref{decomRinf}, we consider a linear relation between elements of these representations, and express each summand in them via the basis from the former theorem. This relation involves finitely many terms, and we can take $n$ large enough such that all of these terms arise from tableaux of size $n$ through Lemma \ref{limtab}. Applying the usual substitution $x_{m}=0$ for all $m>n$ turns this linear relation into one between the bases of the representations appearing in Theorem \ref{FMTdecom}, thus forcing all the coefficients to vanish because the latter theorem implies that no non-trivial linear relations exist in this case.

This proves that the first sum is direct, and it is finite either by part $(iii)$ of Theorem \ref{FMTdecom} or by easily verifying that $\operatorname{SSYT}_{d}(\hat{\lambda})$ from Definition \ref{infSSYT} is finite. The assertion involving $\mu$ is an immediate consequence of, e.g., part $(ii)$ of Proposition \ref{limSpecht}. This completes the proof of the proposition.
\end{proof}
The direct sums from Proposition \ref{ExtVMlim} are sub-representations of $\tilde{\Lambda}_{d}$ and $\tilde{\Lambda}_{\mu}$, which are proper for $d\geq1$---see Example \ref{exd1} and Theorem \ref{compred} below.
\begin{ex}
The case $d=2$ in Proposition \ref{ExtVMlim} is the limit of the expressions from Example \ref{Qxnd2ex}, namely we have, in the notation from Example \ref{RinfIex}, the equality $\tilde{\Lambda}_{2}^{0}=V_{000\cdots}^{2} \oplus V_{\substack{000\cdots \\ 2\hphantom{22\cdots}}} \oplus V_{000\cdots}^{11} \oplus V_{\substack{000\cdots \\ 1\hphantom{11\cdots}}}^{1} \oplus V_{\substack{000\cdots \\ 11\hphantom{1\cdots}}}$. \label{exd2}
\end{ex}

\section{Completely Reducible Sub-Representations \label{MaxCompRed}}

Proposition \ref{samereps} implies that eventually symmetric functions are generated by symmetric functions and polynomials. Our next aim is to modify the decomposition from Theorem \ref{FMTdecom} and Proposition \ref{ExtVMlim} into one whose elements are products of polynomials and symmetric functions as in Remark \ref{polstimessym}, rather than linear combinations of such products.
\begin{lem}
Let $F=F_{M,T}$ be the higher Specht polynomial associated by Definition \ref{Spechtdef} with $M\in\operatorname{SSYT}(\lambda)$ and $T\in\operatorname{SYT}(\lambda)$ for some $\lambda \vdash n$.
\begin{enumerate}[$(i)$]
\item $F$ is also the higher Specht polynomial $F_{\hat{\iota}M,\iota T}$ with the shape $\lambda_{+} \vdash n+1$, if and only if $F$ is also a stable higher Specht polynomial, if and only if the tableau $M$ has no non-zero entries in its first row.
\item The conditions from part $(i)$ are unaffected by replacing $\lambda$, $M$, and $T$ by $\lambda_{+}$, $\iota T$, and $\hat{\iota}M$ respectively. They hold if and only if the stable higher Specht polynomial obtained as the limit of the applications of this process via part $(iv)$ of Proposition \ref{limSpecht} is a polynomial.
\item Given $\hat{M}\in\operatorname{SSYT}(\hat{\lambda})$ as in Definition \ref{infSSYT}, where $\hat{\lambda}\vdash\infty$, the stable higher Specht polynomial $F_{\hat{M},\hat{T}}\in\tilde{\Lambda}$ is in $\mathbb{Q}[\mathbf{x}_{\infty}]$ if and only if $\hat{M}$ involves no non-zero multiplicities in Definition \ref{tabinf}.
\end{enumerate} \label{sameiota}
\end{lem}

\begin{proof}
The argument proving Proposition \ref{forstab}, which is given in detail in the proofs of Proposition 2.15 of \cite{[Z2]} and Proposition 3.23 of \cite{[Z3]}, shows that passing from $F_{M,T}$ to $F_{\hat{\iota}M,\iota T}$ amounts to letting the operator associated with the larger group $R(\iota T)$ act on the same monomial before the operation of the anti-symmetric operator corresponding to $C(T) \cong C(\iota T)$.

By recalling the formula for $\hat{\iota}M$ from part $(vi)$ of Lemma \ref{rels}, it follows that when $M$ contains no non-zero entries in the first row, the same applies for $\hat{\iota}M$, and the operators associated with $R(T)$ and $R(\iota T)$ are the same (because of the division by the sizes of the stabilizers in Definition \ref{Spechtdef}). Hence in this case we get $F_{\hat{\iota}M,\iota T}=F$ as well, and since we can apply the same argument for $\lambda_{+}$, $\iota T$, and $\hat{\iota}M$, and do so repeatedly, we deduce that in this case the stable higher Specht polynomial $F_{\hat{M},\hat{T}}$, where $\hat{M}$ and $\hat{T}$ are the limits from Lemma \ref{limtab}, equals $F$ as well (and is thus a polynomial). It is clear that in this case $\hat{M}$ involves no non-zero multiplicities. This yields one direction in parts $(i)$ and $(ii)$.

Assuming now that there is some non-zero entry in the first row of $M$, the definition of $F_{\hat{\iota}M,\iota T}$ implies that it will contain some monomials that are divisible by $x_{n+1}$, and thus cannot equal $F$. As $F_{\hat{M},\hat{T}}$, for the same limit, yields $F_{\hat{\iota}M,\iota T}$ (with these monomials) when one substitutes $x_{m}=0$ for all $m>n+1$, it cannot equal $F$ either. By applying this argument for $\lambda_{+}$, $\iota T$, and $\hat{\iota}M$, with the latter containing the non-zero entries from the first row of $M$ via part $(vi)$ of Lemma \ref{rels}, we deduce that the next polynomial will also contain monomials that are divisible by $x_{n+2}$. Going over to the limit, we deduce that for every index $m>n$ there exists a monomial in $F_{\hat{M},\hat{T}}$ that is divisible by $x_{m}$, so that the latter element of $\tilde{\Lambda}$ cannot be in $\mathbb{Q}[\mathbf{x}_{\infty}]$. The other directions in parts $(i)$ and $(ii)$ thus follow as well.

We now note that Lemma \ref{limtab} shows that $\hat{M}\in\operatorname{SSYT}(\hat{\lambda})$ is always such a limit, with $F_{\hat{M},\hat{T}}$ constructed via part $(iv)$ of Proposition \ref{limSpecht}, and the non-zero multiplicities appearing in $M$ are precisely the non-zero entries in the first row of any (finite) element of $\operatorname{SSYT}(\lambda)$ producing $\hat{M}$ (and we can use the same subscript $d$ by part $(iii)$ of Lemma \ref{propinf}). Combining this observation with our argument establishes part $(iii)$. This proves the lemma.
\end{proof}

We now complete Proposition \ref{limSpecht}, as Remark \ref{quotinf} predicts. Let $\hat{\lambda}\vdash\infty$, $\hat{T}\in\operatorname{SYT}(\hat{\lambda})$, and $\hat{M}\in\operatorname{SSYT}(\hat{\lambda})$ be as in that proposition, and take $n$ large enough, with the finite $\lambda \vdash n$, $T\in\operatorname{SYT}(\lambda)$, and $M\in\operatorname{SSYT}(\lambda)$ yielding $\hat{\lambda}$, $\hat{T}$, and $\hat{M}$ via Lemma \ref{limtab}. We let $C^{0}\in\operatorname{CCT}(\lambda)$ be the minimal cocharge tableau from Remark \ref{quotSpecht}, which via Definition 2.8 of \cite{[Z2]} contains $i-1$ at every box in the $i$th row, and we can similarly define the minimal infinite cocharge tableau $\hat{C}^{0}\in\operatorname{CCT}(\hat{\lambda})$ in the same manner for the finite part, and with no positive multiplicities in Definition \ref{tabinf}. We then get the following result.
\begin{cor}
In this setting, the following assertions hold.
\begin{enumerate}[$(i)$]
\item the minimal cocharge tableau for $\lambda_{+}$ is $\hat{\iota}C^{0}$, and $F_{C^{0},T}$, $F_{\hat{\iota}C^{0},\iota T}$, and $F_{\hat{C}^{0},\hat{T}}$ are the same polynomial.
\item Substituting $x_{n+1}=0$ in the $\tilde{C}(\iota T)$-invariant quotient $Q_{\hat{\iota}M,\iota T}$ yields $Q_{M,T}$.
\item The stable generalized higher Specht polynomial $F_{\hat{M},\hat{T}}$ is divisible by $F_{\hat{C}^{0},\hat{T}}$, and the quotient $Q_{\hat{M},\hat{T}}$ is invariant under $\tilde{C}(\hat{T})$.
\item The quotient from $Q_{\hat{M},\hat{T}}$ is the unique homogeneous $\tilde{C}(\hat{T})$-invariant element of $\tilde{\Lambda}$ in which setting each $x_{m}$ with $m>n$ to be 0 yields $Q_{M,T}$.
\end{enumerate} \label{infquot}
\end{cor}

\begin{proof}
Part $(i)$ follows directly from Lemma \ref{sameiota}. Part $(ii)$ is a consequence of Proposition \ref{forstab}, since part $(i)$ shows that we divide both sides by the same polynomial. Knowing $F_{\hat{C}^{0},\hat{T}}$, the proof of Corollary 2.9 of \cite{[Z2]} and Corollary 2.6 of \cite{[Z3]} yields part $(iii)$. Finally, part $(iv)$ is obtained by the same limiting process as in the proof of part $(iv)$ of Proposition \ref{limSpecht}, since again we divide and multiply by the same denominator via part $(i)$. This proves the corollary.
\end{proof}
We saw a case of part $(ii)$ of Corollary \ref{infquot} in Example \ref{quotex} above, with one for parts $(iii)$ and $(iv)$ showing up in Example \ref{Spechtinf}.

\medskip

Lemma \ref{sameiota} shows that representations associated with elements of $\operatorname{SSYT}(\lambda)$ whose first row contains only zeros, or $\operatorname{SSYT}(\hat{\lambda})$ having no non-zero multiplicities, are spanned by higher Specht polynomials with better properties. Based on this observation, we define yet another type of representations based on higher Specht polynomials and symmetric functions.
\begin{defn}
Take $\lambda \vdash n$, $M\in\operatorname{SSYT}(\lambda)$, $\hat{\lambda}\vdash\infty$ and $\hat{M}\in\operatorname{SSYT}(\hat{\lambda})$.
\begin{enumerate}[$(i)$]
\item Let $M^{0}$ to be the semi-standard Young tableau of shape $\lambda$ obtained from $M$ by replacing all the first row by zeros, write $m_{M}$ for the monomial symmetric function in $n$ variables whose index consists of the (non-zero) entries in the first row of $M$, set $f_{M}$ to be the number of non-zero entries in that row, and define $\tilde{V}_{M}:=V_{M^{0}}m_{M}$.
\item For any $d\geq0$ and $f\geq0$, we denote by $\mathbb{Q}[\mathbf{x}_{n}]_{d}^{f}$ the sub-representation of $\mathbb{Q}[\mathbf{x}_{n}]_{d}$ that is given by $\bigoplus_{\lambda \vdash n}\bigoplus_{M\in\operatorname{SSYT}_{d}(\lambda),\ f_{M} \leq f}V_{M}$. In particular we write $\mathbb{Q}[\mathbf{x}_{n}]_{d}^{-1}=\{0\}$ for every $d\geq0$.
\item Given a content $\mu$ of sum $d$ and some $f\geq-1$, we write $\mathbb{Q}[\mathbf{x}_{n}]_{\mu}^{f}$ for the intersection$\bigoplus_{\lambda \vdash n}\bigoplus_{M\in\operatorname{SSYT}_{\mu}(\lambda),\ f_{M} \leq f}V_{M}$ of $\mathbb{Q}[\mathbf{x}_{n}]_{d}^{f}$ with $\mathbb{Q}[\mathbf{x}_{n}]_{\mu}$, and in particular $\mathbb{Q}[\mathbf{x}_{n}]_{\mu}^{-1}=\{0\}$.
\item Given $\hat{\lambda}\vdash\infty$ and $\hat{M}\in\operatorname{SSYT}(\hat{\lambda})$, we denote by $\hat{M}^{0}$ the infinite semi-standard Young tableau of shape $\hat{\lambda}$ which one gets by removing all the multiplicities from $\hat{M}$, define $m_{\hat{M}}$ to be the monomial element of $\Lambda$ whose index contains each $h>0$ with the same multiplicity as it appears in $\hat{M}$, denote by $f_{\hat{M}}$ the sum of these multiplicities, and set $\tilde{V}_{\hat{M}}:=V_{\hat{M}^{0}}m_{\hat{M}}$.
\end{enumerate} \label{repswithmon}
\end{defn}

\begin{ex}
The $f$-values of the representations showing up for $n=6$ and $n=7$ in Example \ref{RnIex}, as well as those from Example \ref{RinfIex}, are 0, 0, 2, 1, 0, 1, and 0 respectively. In those from Example \ref{Qxnd2ex} with $n$ being 4, 5, or 6, and the ones appearing in Example \ref{exd2}, the values are 1, 0, 2, 1, and 0. \label{fvals}
\end{ex}

\begin{rmk}
Part $(vi)$ Lemma \ref{rels} clearly implies, via Definition \ref{repswithmon}, that $f_{\hat{\iota}M}=f_{M}$ for every $\lambda \vdash n$ and $M\in\operatorname{SSYT}(\lambda)$, and that $(\hat{\iota}M)^{0}=\hat{\iota}(M^{0})$. Moreover, part $(iii)$ of Lemma \ref{limtab} shows that if $\hat{M}\in\operatorname{SSYT}(\hat{\lambda})$ is related to $M$ (or to $\hat{\iota}M$), then $f_{\hat{M}}=f_{M}$ as well. \label{fiota}
\end{rmk}

\begin{rmk}
It is clear from Definition \ref{repswithmon} that the number of non-zero entries in the content of $p_{M,T}$ from Definition \ref{Spechtdef} (resp. $p_{\hat{M},\hat{T}}$ from Definition \ref{infSpecht}) is the number of boxes in $M$ (resp. $\hat{M}$) that are not in the first row plus $f_{M}$ (resp. $f_{\hat{M}}$), and the equality $f_{M}=0$ (resp. $f_{\hat{M}}=0$) is equivalent to $M$ (resp. $\hat{M}$) satisfying the conditions from Lemma \ref{sameiota}. Hence if $M\in\operatorname{SSYT}_{\mu}(\lambda)$ for some content $\mu$ then $f_{M}$ is determined by $\mu$ and $\lambda$ (and $\mu$ and $\hat{\lambda}\vdash\infty$ determine $f_{\hat{M}}$ for every $\hat{M}\in\operatorname{SSYT}_{\mu}(\hat{\lambda})$), so that $\mathbb{Q}[\mathbf{x}_{n}]_{\mu}^{f}$ is a collection of isotypical parts of $\mathbb{Q}[\mathbf{x}_{n}]_{\mu}$. More explicitly, if $\mu$ has $\ell$ non-zero entries then the $\mathcal{S}^{\lambda}$-isotypical component of $\mathbb{Q}[\mathbf{x}_{n}]_{\mu}$ will appear in $\mathbb{Q}[\mathbf{x}_{n}]_{\mu}^{f}$ if $\sum_{i=1}^{\ell(\lambda)}\lambda_{i}\geq\ell-f$, and not otherwise. \label{mudetf}
\end{rmk}

If we decompose quotients like $R_{n,k,0}$ or $R_{n,k}$ from Theorem \ref{Rnkdecom} into their homogeneous components, then the component of degree $d$ is filtered by the images of $\mathbb{Q}[\mathbf{x}_{n}]_{d}^{f}$ for the various values of $f$. While we do not need these filtrations in the finite $n$ case, their infinite analogues will show up and be useful later---see Definition \ref{Lambdafilt} below.

\medskip

We will use the following technical lemma.
\begin{lem}
Take a content $\mu$ of sum $d$, partitions $\lambda \vdash n$ and $\hat{\lambda}\vdash\infty$, finite tableaux $M\in\operatorname{SSYT}_{\mu}(\lambda)$ and $T\in\operatorname{SYT}(\lambda)$, and infinite ones $\hat{M}\in\operatorname{SSYT}_{\mu}(\hat{\lambda})$ and $\hat{T}\in\operatorname{SYT}(\hat{\lambda})$ as in Definition \ref{tabinf}.
\begin{enumerate}[$(i)$]
\item If $M^{0}$ and $m_{M}$ are as in Definition \ref{repswithmon}, then $F_{M^{0},T}m_{M}$ equals $F_{M,T}$ plus a linear combination of monomials, all of whose contents have less non-zero entries than $\mu$.
\item For $\hat{M}^{0}$ and $m_{\hat{M}}$ from that definition, the difference between $F_{\hat{M}^{0},\hat{T}}m_{\hat{M}}$ and $F_{\hat{M},\hat{T}}$ is supported on (possibly infinite) monomials of contents with less non-zero entries than $\mu$.
\end{enumerate} \label{difmodf-1}
\end{lem}

\begin{proof}
The monomial $p_{M^{0},T}$ is not divisible by any variable that is associated with an entry of the first row of $T$, and is invariant under the subgroup of $R(T)$ acting only on the first row. Therefore if $R^{>1}(T)$ is the subgroup of $R(T)$ acting only on the rows of $T$ that are not the first one, and $\operatorname{st}_{T}^{>1}M$ is the stabilizer of $p_{M^{0},T}$ (or of $M^{0}$) there, then $F_{M^{0},T}$ can also be given by $\sum_{\sigma \in C(T)}\sum_{\tau \in R^{>1}(T)/\operatorname{st}_{T}^{>1}M}p_{M,T}\operatorname{sgn}(\sigma)\sigma\tau p_{M^{0},T}$. As the action of elements of $\mathbb{Q}[S_{n}]$ commutes with multiplication by the symmetric function $m_{M}$, the product from part $(i)$ is the same as $\sum_{\sigma \in C(T)}\sum_{\tau \in R^{>1}(T)/\operatorname{st}_{T}^{>1}M}\operatorname{sgn}(\sigma)\sigma\tau(p_{M^{0},T}m_{M})$.

Note that for $\hat{M}$ and $\hat{T}$ in part $(ii)$, Lemma \ref{sameiota} shows that $F_{\hat{M}^{0},\hat{T}}$ is again a polynomial, the group $R^{>1}(\hat{T})$ is finite (while $R(\hat{T})$ is not), and $C(\hat{T})$ is also finite (see part $(i)$ Lemma \ref{propinf}), so that the product in part $(ii)$ is given by $\sum_{\sigma \in C(\hat{T})}\sum_{\tau \in R^{>1}(\hat{T})/\operatorname{st}_{\hat{T}}^{>1}\hat{M}}\operatorname{sgn}(\sigma)\sigma\tau(p_{\hat{M}^{0},\hat{T}}m_{\hat{M}})$, and also in this case we could replace the sum over cosets by the sum over the group divided by the (finite) size of the stabilizer.

The rest of the proof is based on decomposing $m_{M}$ and $m_{\hat{M}}$ into expressions that are based on the variables from the first row of $T$ or $\hat{T}$, which we write as $\mathbf{x}_{T}^{1}$ or $\mathbf{x}_{\hat{T}}^{1}$, times those that are based on other variables, written collectively as $\mathbf{x}_{T}^{>1}$ or $\mathbf{x}_{\hat{T}}^{>1}$. The only difference between $T$ and $\hat{T}$ is be that the set of variables $\mathbf{x}_{T}^{1}$ is finite, $\mathbf{x}_{\hat{T}}^{1}$ is infinite. We thus prove part $(i)$ only, and the proof of part $(ii)$ carries over verbatim, after changing $T$ and $M$ into $\hat{T}$ and $\hat{M}$ respectively.

Now, $m_{M}=m_{M}(\mathbf{x}_{\infty})$ is a symmetric function in the disjoint union $\mathbf{x}_{T}^{1}\cup\mathbf{x}_{T}^{>1}$ of two sets of variables. If $\alpha$ is the index of our monomial symmetric function as a multi-set, which has $f_{M}$ non-zero entries, then $m_{M}$ decomposes as the sum, over presentations of $\alpha$ as a multi-set sum $\beta+\gamma$, of the products $m_{\beta}(\mathbf{x}_{T}^{1})m_{\gamma}(\mathbf{x}_{T}^{>1})$.

Since $p_{M^{0},T}$ is divisible by all the variables from $\mathbf{x}_{T}^{>1}$, multiplying it by a monomial from $m_{\gamma}(\mathbf{x}_{T}^{>1})$ does not add non-zero entries to its content, and a monomial in $m_{\beta}(\mathbf{x}_{T}^{1})m_{\gamma}(\mathbf{x}_{T}^{>1})$ is thus divisible by all the variables from $\mathbf{x}_{T}^{>1}$ plus some variables from $\mathbf{x}_{T}^{1}$, the number of the latter being the number of non-zero entries of $\beta$. If $\beta\neq\alpha$ then this number is strictly smaller than $f_{M}$, so that all of these terms yield linear combinations of monomials having content with strictly less non-zero entries.

It thus suffices to prove that the part arising from $\beta=\alpha$ (and $\gamma=0$) is precisely $F_{M,T}$. We thus take one monomial from $m_{\alpha}(\mathbf{x}_{T}^{1})$, and then the multiplier $m_{\alpha}(\mathbf{x}_{T}^{1})$ is the image of that monomial under the action of the permutation group of the first row of $T$ modulo the stabilizer of that monomial. Moreover, the product of $p_{M^{0},T}$ with that monomial lies, by definition, in the orbit of $p_{M,T}$ under the latter permutation group, and we can choose the monomial so that the product is precisely $p_{M,T}$.

As $R(T)$ is the direct product of the latter permutation group with $R^{>1}(T)$, and the latter stabilizer combines with the stabilizer of $p_{M^{0},T}$ inside $R^{>1}(T)$, we deduce that the resulting part of $\sum_{\tau \in R^{>1}(T)/\operatorname{st}_{T}^{>1}M}\tau(p_{M^{0},T}m_{M})$ is the same as $\sum_{\tau \in R(T)/\operatorname{st}_{T}M}\tau p_{M,T}$. Applying the operator involving the action of $C(T)$ on the result produces $F_{M,T}$ via part $(i)$ of Theorem \ref{repsinf}, as desired. This completes the proof of the lemma.
\end{proof}

\begin{ex}
Consider the representations from Example \ref{Qxnd2ex} for $n\geq4$, or those from Example \ref{exd2} for the infinite setting. The indices $M$ of two of those satisfy the conditions from Lemma \ref{sameiota}, corresponding to the two vanishing $f$-values in Example \ref{fvals}. The one-line representations there, one with $f$-value 1 and one with $f$-value 2, are trivial representations (multiplied by the symmetric functions $p_{2}$ and $e_{2}$ in the number of variables we work with), yielding the equality from Lemma \ref{difmodf-1}. In the remaining representation $V_{\substack{0001 \\ 1\hphantom{112}}}$ (with the number of zeros varying with $n$), of $f$-value 1, take $T$ to be the standard tableau of the corresponding shape with 1 and 2 in the first column, and then the corresponding basis element of $V_{M}$ is $(x_{2}-x_{1})\sum_{i\geq3}x_{i}$, which is supported on monomials that involve two non-zero exponents. However, for $\tilde{V}_{M}$ the associated basis element is $(x_{2}-x_{1})\sum_{i\geq3}x_{i}$, and the difference is $x_{2}^{2}-x_{1}^{2}$, indeed supported on monomials in which only one exponent is non-zero. \label{exmodf-1}
\end{ex}
For a larger example, see Example \ref{relsd6ex} below.

\medskip

\begin{prop}
Take any two integers $d\geq0$ and $n\geq1$.
\begin{enumerate}[$(i)$]
\item The space $\mathbb{Q}[\mathbf{x}_{n}]_{d}$ equals to the direct sum $\bigoplus_{\lambda \vdash n}\bigoplus_{M\in\operatorname{SSYT}_{d}(\lambda)}\tilde{V}_{M}$.
\item The first direct sum from Proposition \ref{ExtVMlim} for this $d$ also coincides with the direct sum $\bigoplus_{\hat{\lambda}\vdash\infty}\bigoplus_{\hat{M}\in\operatorname{SSYT}_{d}(\hat{\lambda})}\tilde{V}_{\hat{M}}$.
\end{enumerate} \label{decomVM0mM}
\end{prop}

\begin{proof}
Since both $V_{M}$ and $\tilde{V}_{M}$ are isomorphic to $\mathcal{S}^{\lambda}$ when $M\in\operatorname{SSYT}_{d}(\lambda)$, taking the $\mathcal{S}^{\lambda}$-isotypical parts in Theorem \ref{FMTdecom} implies that part $(i)$ is equivalent to the equality $\bigoplus_{M\in\operatorname{SSYT}_{d}(\lambda)}\tilde{V}_{M}=\bigoplus_{M\in\operatorname{SSYT}_{d}(\lambda)}V_{M}$ holding for any $\lambda \vdash n$. Similarly, part $(ii)$ is equivalent to $\bigoplus_{\hat{M}\in\operatorname{SSYT}_{d}(\hat{\lambda})}\tilde{V}_{\hat{M}}$ being the same as $\bigoplus_{\hat{M}\in\operatorname{SSYT}_{d}(\hat{\lambda})}V_{\hat{M}}$, since parts $(iii)$, $(iv)$, and $(v)$ of Theorem \ref{repsinf} show that both are then also isotypical parts of the representation in question.

Recalling the parameters from Definition \ref{repswithmon}, we will show by induction on $f$ that the $\mathcal{S}^{\lambda}$-isotypical part of $\mathbb{Q}[\mathbf{x}_{n}]_{d}^{f}$ is also given by $\bigoplus_{M\in\operatorname{SSYT}_{d}(\lambda), f_{M} \leq f}\tilde{V}_{M}$, with the sum being direct. As in the proof of Lemma \ref{difmodf-1}, we will prove part $(i)$ using the notations for finite $n$, and the proof of part $(ii)$ carries over verbatim, except for replacing $\mathbb{Q}[\mathbf{x}_{n}]_{d}^{f}$ by the analogue of its definition (we do not write it as $\tilde{\Lambda}_{d}^{f}$, since this symbol will have a different meaning in Theorem \ref{filtrations} below).

It is clear from Definition \ref{repswithmon} that if $f_{M}=0$ then $M^{0}=M$ and $m_{M}=1$, so that $\tilde{V}_{M}=V_{M}$ and the assertion for $f=0$ is trivial (with the direct sum following from Theorem \ref{FMTdecom}). We therefore assume that $f>0$ and the result holds for $f-1$, namely we can describe the $\mathcal{S}^{\lambda}$-isotypical part of $\mathbb{Q}[\mathbf{x}_{n}]_{d}^{f-1}$ also by the appropriate direct sum of the $\tilde{V}_{M}$'s.

But then Remark \ref{mudetf} implies that the $\mathcal{S}^{\lambda}$-isotypical part of $\mathbb{Q}[\mathbf{x}_{n}]_{d}^{f-1}$ is the subspace of the $\mathcal{S}^{\lambda}$-isotypical part of $\mathbb{Q}[\mathbf{x}_{n}]_{d}$ that is spanned by all the monomials such that the number of non-zero entries in their content is less than $f$ plus the number of boxes of $\lambda$ that are not in the first row. This means that Lemma \ref{difmodf-1} implies that for $M\in\operatorname{SSYT}_{d}(\lambda)$ with $f_{M}=f$, and for $T\in\operatorname{SYT}(\lambda)$, the expressions $F_{M^{0},T}m_{M}$ and $F_{M,T}$ are the same modulo $\mathbb{Q}[\mathbf{x}_{n}]_{d}^{f-1}$. Since the direct sum property from Theorem \ref{FMTdecom} is equivalent to $F_{M,T}$ for $M$ with $f_{M}=f$ and such $T$ being linearly independent modulo $\mathbb{Q}[\mathbf{x}_{n}]_{d}^{f-1}$, which is the same for the $F_{M^{0},T}m_{M}$'s, we deduce the direct sum property, and it is also clear that they generate the same space modulo $\mathbb{Q}[\mathbf{x}_{n}]_{d}^{f-1}$.

This proves our claim for all $f$, and since the direct sum from Theorem \ref{FMTdecom} is finite, we reach the full space $\mathbb{Q}[\mathbf{x}_{n}]_{d}$ for large enough $f$, yielding part $(i)$. Recalling that the sum from Proposition \ref{ExtVMlim} is finite as well, this part of the proof also works in the infinite case, and part $(ii)$ is also established. This completes the proof of the proposition.
\end{proof}
Note that with as $\mathbb{Q}[\mathbf{x}_{n}]_{d}^{-1}=\{0\}$ in Definition \ref{repswithmon}, also the base case $f=0$ in the proof of Proposition \ref{decomVM0mM} can be seen as working modulo $\mathbb{Q}[\mathbf{x}_{n}]_{d}^{f-1}$.

In fact, part $(ii)$ in either Lemma \ref{difmodf-1} or Proposition \ref{decomVM0mM} follows from the respective part $(i)$ by taking limits as usual, but we saw that the proofs apply for both assertions in the same manner. For the case $d=2$ in Proposition \ref{decomVM0mM}, Example \ref{exmodf-1} shows how four summands coincide with those from Theorem \ref{FMTdecom} (or Proposition \ref{ExtVMlim} in the infinite setting), and the difference between the corresponding summand in each is supported on one of the coinciding ones.

\medskip

In this paper, as well as its predecessors \cite{[Z2]} and \cite{[Z3]}, we constructed several decompositions of representations. We now gather them together, and compare their properties.
\begin{defn}
Fix $\lambda \vdash n$ and integers $d\geq0$, $k\geq1$, and $0 \leq s\leq\min\{k,n\}$.
\begin{enumerate}[$(i)$]
\item We denote by $A_{d}(\lambda)$ the set of pairs $(C,I)$, where $C\in\operatorname{CCT}(\lambda)$ and $I$ is a multi-set inside $\mathbb{N}_{n}\cup\{0\}$ that contains $\operatorname{Dsp}^{c}(C)$ in the sense of Definition \ref{multisets}, does not contain 0, and satisfies $\sum_{i \in I}i=d$ (the sum with multiplicities as usual).
\item The subset $A_{d}^{k,s}(\lambda)$ of $A_{d}(\lambda)$ consists of those $(C,I)$ where $|I|<k$ and the maximal element of the complement $\operatorname{Asp}^{c}_{I}(C)=I\setminus\operatorname{Dsp}^{c}(C)$ from Definition \ref{sets} does not exceed $n-s$. Write $A_{d}^{k}(\lambda)$ for $A_{d}^{k,k}(\lambda)$.
\item We denote by $\operatorname{SSYT}_{d}^{k}(\lambda)$ the set of those elements of $\operatorname{SSYT}_{d}(\lambda)$ all of whose entries are smaller than $k$.
\end{enumerate} \label{Adlambda}
\end{defn}

Using the notions from Definition \ref{Adlambda}, the results are now as follows.
\begin{thm}
Take $d\geq0$, $n\geq1$, and $\lambda \vdash n$.
\begin{enumerate}[$(i)$]
\item The $\mathcal{S}^{\lambda}$-isotypical part of $\mathbb{Q}[\mathbf{x}_{n}]_{d}$ decomposes as $\bigoplus_{M\in\operatorname{SSYT}_{d}(\lambda)}V_{M}$, as $\bigoplus_{M\in\operatorname{SSYT}_{d}(\lambda)}\tilde{V}_{M}$, and as $\bigoplus_{(C,I) \in A_{d}(\lambda)}V_{C}^{\vec{h}(C,I)}$.
\item For any $k\geq1$ and $0 \leq s\leq\min\{k,n\}$, the sum $\bigoplus_{(C,I) \in A_{d}^{k,s}(\lambda)}V_{C}^{\vec{h}(C,I)}$ bijects onto the $\mathcal{S}^{\lambda}$-isotypical part of the part of $R_{n,k,s}$ that is homogeneous of degree $d$. If $s<\min\{k,n\}$ then any $V_{C}^{\vec{h}(C,I)}$ for $(C,I) \in A_{d}^{k,s}(\lambda) \setminus A_{d}^{k,s+1}(\lambda)$ vanishes under the natural map from $R_{n,k,s}$ onto $R_{n,k,s+1}$.
\item Given $k\geq1$, the $\mathcal{S}^{\lambda}$-isotypical part of the part of $R_{n,k,0}$ that is homogeneous of degree $d$ is the bijective image of both $\bigoplus_{M\in\operatorname{SSYT}_{d}^{k}(\lambda)}V_{M}$ and $\bigoplus_{M\in\operatorname{SSYT}_{d}^{k}(\lambda)}\tilde{V}_{M}$. If $M\in\operatorname{SSYT}_{d}^{k+1}(\lambda)\setminus\operatorname{SSYT}_{d}^{k}(\lambda)$ then the associated summand $V_{M}$ from $R_{n,k+1,0}$ vanishes when projected onto $R_{n,k,0}$.
\item The map taking any generator $F_{M^{0},T}m_{M}$ of $\tilde{V}_{M}$, with $T\in\operatorname{SYT}(\lambda)$, to the same expression but with $m_{M}$ now symmetric in $n+1$ variables, embeds $\tilde{V}_{M}$ into the restriction of $\tilde{V}_{\hat{\iota}M}$ from $S_{n+1}$ to $S_{n}$. When $n>2d$ the image of $\bigoplus_{M\in\operatorname{SSYT}_{d}(\lambda)}\tilde{V}_{M}$ under this operation generates, under the action of $S_{n+1}$, the $\mathcal{S}^{\lambda_{+}}$-isotypical part of $\mathbb{Q}[\mathbf{x}_{n+1}]_{d}$, and by summing over all $\lambda \vdash n$ this process produces all of $\mathbb{Q}[\mathbf{x}_{n+1}]_{d}$.
\end{enumerate} \label{propbases}
\end{thm}
The case $s=k \leq n$ in part $(ii)$ of Theorem \ref{propbases} means that the $\mathcal{S}^{\lambda}$-isotypical part of the degree $d$ part of $R_{n,k}$ is the bijective image of $\bigoplus_{(C,I) \in A_{d}^{k}(\lambda)}V_{C}^{\vec{h}(C,I)}$. The second assertion in part $(iv)$ there holds, in fact, for all $n\geq2d$.

\begin{proof}
The first two expressions in part $(i)$ follow directly from Theorem \ref{FMTdecom} and Proposition \ref{decomVM0mM}, and the third one is Corollary 2.27 of \cite{[Z3]}. Next, part $(iii)$ of Theorem 2.19 of \cite{[Z3]} and the decomposition of each $R_{n,I}^{\mathrm{hom}}$ amounts, by taking the $\mathcal{S}^{\lambda}$-isotypical part and replacing every multi-set of size $k-1$ by the possibly smaller one not containing 0, to the assertion of part $(ii)$ with $s=0$. The general case follows by noting that the summands that remain non-zero in $R_{n,k,s}$ are those not containing $e_{r}$ with $r>n-s$, and the vanishing under the natural map is also clear by the same argument (because $R_{n,k,s+1}$ is the quotient of $R_{n,k,s}$ by the ideal generated by $e_{n-s}$ by definition).

We now apply part $(i)$ (or Theorem \ref{FMTdecom}), and recall from part $(ii)$ of Theorem \ref{repsSpecht} that if $M\in\operatorname{SSYT}_{\mu}(\lambda)$ for some content $\mu$ (of sum $d$) then $V_{M}\subseteq\mathbb{Q}[\mathbf{x}_{n}]_{\mu}\subseteq\mathbb{Q}[\mathbf{x}_{n}]_{d}$. Note that if $\mu$ contains an entry that equals $k$ or more then $\mathbb{Q}[\mathbf{x}_{n}]_{\mu}$ vanishes in $R_{n,k,0}$, and otherwise no relations appear between elements of $\mathbb{Q}[\mathbf{x}_{n}]_{\mu}$ in that quotient. since when considering the differences in Lemma \ref{difmodf-1} for $M\in\operatorname{SSYT}_{d}^{k}(\lambda)$, some are based on contents of elements that are in $\operatorname{SSYT}_{d}^{k}(\lambda)$ but some are not, but in the quotient $R_{n,k,0}$ the two asserted sums, which are clearly isomorphic, have the same image and therefore they are isomorphic to it together, part $(iii)$ is also established.

Finally, for every element of $\mathbb{Q}[\mathbf{x}_{n+1}]_{d}$ there exists a unique symmetric divisor of maximal degree (see, e.g., Remark 3.3 of \cite{[Z2]}), which for $F_{M^{0},T}m_{M}$ is $m_{M}$ (hence so it is for every element of $\tilde{V}_{M}$), making our map well-defined on these polynomials. Lemma \ref{sameiota} yields the equality $F_{M^{0},T}=F_{\hat{\iota}(M^{0}),\iota T}$, and Definition \ref{repswithmon} and the formula for $\hat{\iota}$ in Lemma \ref{rels} show that $\hat{\iota}(M^{0})=(\hat{\iota}M)^{0}$ and that adding the variable $x_{n+1}$ to $m_{M}$ yields the monomial symmetric function $m_{\hat{\iota}M}$. Hence the image of $\tilde{V}_{M}$ under this map is contained in $\tilde{V}_{\hat{\iota}M}$, yielding the first assertion in part $(iv)$ by the irreducibility of the latter representation. Since part $(iii)$ of Theorem \ref{FMTdecom} implies that for $n>2d$ the decomposition of $\mathbb{Q}[\mathbf{x}_{n+1}]_{d}$ from Proposition \ref{decomVM0mM} is $\bigoplus_{\lambda \vdash n}\bigoplus_{M\in\operatorname{SSYT}_{d}(\lambda)}\tilde{V}_{\hat{\iota}M}$, with $\tilde{V}_{\hat{\iota}M}\cong\mathcal{S}^{\lambda_{+}}$ since $\operatorname{sh}(\hat{\iota}M)=\lambda_{+}$ for any $M\in\operatorname{SSYT}_{d}(\lambda)$, the remaining assertions follow as well. This completes the proof of the theorem.
\end{proof}
In fact, the subspace of $\tilde{V}_{\hat{\iota}M}$ that is obtained by the map from part $(iv)$ of Theorem \ref{propbases} is spanned by those basis elements that are not divisible by $x_{n+1}m_{\hat{\iota}M}$, since $F_{(\hat{\iota}M)^{0},\tilde{T}}$ is not divisible by $x_{n+1}$ precisely when $\tilde{T}\in\operatorname{SYT}(\lambda_{+})$ is $\iota T$ for some $T\in\operatorname{SYT}(\lambda)$ (in order for $x_{n+1}$ to appear in the first row of $\tilde{T}$, and thus at its end).

\medskip

\begin{ex}
We consider the case where $n=7$, $\lambda=421\vdash7$, and $d=6$. Then the decompositions from part $(i)$ of Theorem \ref{propbases} are \[V_{\substack{0000 \\ 11\hphantom{11} \\ 4\hphantom{444}}} \oplus V_{\substack{0000 \\ 12\hphantom{22} \\ 3\hphantom{333}}} \oplus V_{\substack{0000 \\ 13\hphantom{33} \\ 2\hphantom{444}}} \oplus V_{\substack{0001 \\ 11\hphantom{12} \\ 3\hphantom{333}}} \oplus V_{\substack{0001 \\ 12\hphantom{22} \\ 2\hphantom{333}}} \oplus V_{\substack{0002 \\ 11\hphantom{13} \\ 2\hphantom{224}}} \oplus V_{\substack{0011 \\ 11\hphantom{22} \\ 2\hphantom{233}}}\] (ordered by the $f$-values 0, 0, 0, 1, 1, 1, and 2), their $\tilde{V}_{M}$-equivalents, and \[V_{\substack{0000 \\ 11\hphantom{11} \\ 2\hphantom{222}}}e_{1}^{2} \oplus V_{\substack{0000 \\ 11\hphantom{11} \\ 2\hphantom{222}}}e_{2} \oplus V_{\substack{0000 \\ 12\hphantom{22} \\ 2\hphantom{333}}}e_{1} \oplus V_{\substack{0001 \\ 11\hphantom{12} \\ 2\hphantom{222}}}e_{1} \oplus V_{\substack{0001 \\ 12\hphantom{22} \\ 2\hphantom{333}}} \oplus V_{\substack{0002 \\ 11\hphantom{13} \\ 2\hphantom{224}}} \oplus V_{\substack{0011 \\ 11\hphantom{22} \\ 2\hphantom{233}}},\] corresponding to the multi-sets $\{1,1,1,3\}$, $\{1,2,3\}$ twice, $\{1,1,4\}$, $\{2,4\}$ twice, and $\{1,5\}$ respectively. When $k=4$ we remove the first summand in the latter expression, for $k=3$ we also take out the next three, while for large $k$, if $s=6$ then we remove the second summand, and for $s=7=n$ we again keep only the last three summands. When considering the first expansion, if $k=4$ we omit the first summand, and for $k=3$ we remove the next three as well, exemplifying parts $(ii)$ and $(iii)$ of that theorem. While the inequality $n>2d$ (or $n\geq2d$) from part $(iv)$ does not hold, with our choice of $\lambda$ the components associated with $\lambda_{+}=521\vdash7$ are obtained from these decompositions by replacing every $M$ or $C$ by its $\hat{\iota}$-image. \label{421d6}
\end{ex}

\begin{rmk}
Theorem \ref{propbases} can be interpreted as considering the properties of the three decompositions of $\mathbb{Q}[\mathbf{x}_{n}]_{d}$ from part $(i)$ there. The one from part $(ii)$ is more adapted to changing the parameter $s$ in the quotients $R_{n,k,s}$, that from part $(iii)$ behaves more naturally when changing the parameter $k$ in $R_{n,k,0}$, and the basis from part $(iv)$ is more natural when letting $n$ vary (either in the quotients $R_{n,k,0}$, or in the original space). In fact, recall from Remark 2.30 of \cite{[Z3]} that $R_{n,k,0}$ can be seen as the cohomology ring of the complete variety $(\mathbb{P}^{k-1})^{n}$, and the natural pairing arising from this interpretation is based on the coefficient of $e_{n}^{k-1}$ in the product. Thus $V_{M}$ for $M\in\operatorname{SSYT}_{d}^{k}(\lambda)$, can, by degree considerations, have non-zero pairings with $V_{N}$ only if $N\in\operatorname{SSYT}_{n(k-1)-d}^{k}(\lambda)$, but the content puts additional restrictions. More precisely, since $V_{M}\subseteq\mathbb{Q}[\mathbf{x}_{n}]_{\mu}$, if we order $\mu$ as an increasing sequence of length $n$ (including 0) that is strictly bounded by $k$, then a necessary condition for $V_{N}$ to pair non-trivially with $M$ is that when ordering the content of $N$ as a \emph{decreasing} sequence, then adding the two sequences entry-wise yield the constant sequence $k-1$ of length $n$. Hence the basis from part $(iii)$ of Theorem \ref{propbases} is closer to the orthogonality mentioned in Remark 2.30 of \cite{[Z3]}. \label{difbases}
\end{rmk}

It is possible that the condition from Remark \ref{difbases} for a non-zero pairing between $V_{M}$ and $V_{N}$ is also sufficient, though I did not verify this.

\medskip

The notions from Definition \ref{Adlambda} has the following infinite analogous.
\begin{defn}
Consider $\hat{\lambda}\vdash\infty$ and two integers $d\geq0$ and $k\geq1$.
\begin{enumerate}[$(i)$]
\item We set $A_{d}(\hat{\lambda})$ for the set of pairs $(\hat{C},I)$ in which $\hat{C}\in\operatorname{CCT}(\hat{\lambda})$ and $I$ is a multi-set of positive integers containing $\operatorname{Dsp}^{c}(\hat{C})$ as above, and for which the equality $\sum_{i \in I}i=d$ holds.
\item We write $A_{d}^{k}(\hat{\lambda})$ for the subset of $A_{d}(\hat{\lambda})$ in which we impose the extra condition $|I|<k$.
\item The set of elements of $\operatorname{SSYT}_{d}(\hat{\lambda})$ from Definition \ref{infSSYT} whose entries are all smaller than $k$ will be denote by $\operatorname{SSYT}_{d}^{k}(\hat{\lambda})$.
\item We write $\operatorname{SSYT}^{k}(\hat{\lambda})$ for the disjoint union $\bigcup_{d=0}^{\infty}\operatorname{SSYT}_{d}^{k}(\hat{\lambda})$.
\end{enumerate} \label{Adinf}
\end{defn}

Their properties are described in the following analogue of Theorem \ref{propbases}.
\begin{thm}
Fix $d\geq0$ and $\hat{\lambda}\vdash\infty$.
\begin{enumerate}[$(i)$]
\item The image, inside $\tilde{\Lambda}_{d}$, of the sum from part $(ii)$ of Corollary \ref{multCI} coincides with $\tilde{\Lambda}_{d}^{0}$ from Proposition \ref{ExtVMlim}.
\item The latter image is completely reducible, and its $\mathcal{S}^{\hat{\lambda}}$-isotypical part equals $\bigoplus_{\hat{M}\in\operatorname{SSYT}_{d}(\hat{\lambda})}V_{\hat{M}}$ from Proposition \ref{ExtVMlim}, $\bigoplus_{\hat{M}\in\operatorname{SSYT}_{d}(\hat{\lambda})}\tilde{V}_{\hat{M}}$ via Definition \ref{repswithmon}, and also $\bigoplus_{(\hat{C},I) \in A_{d}(\hat{\lambda})}V_{\hat{C}}^{\vec{h}(\hat{C},I)}$ using Definitions \ref{infSpecht} and \ref{Adinf}.
\item All the three direct sums $\bigoplus_{(\hat{C},I) \in A_{d}^{k}(\hat{\lambda})}V_{\hat{C}}^{\vec{h}(\hat{C},I)}$, $\bigoplus_{\hat{M}\in\operatorname{SSYT}_{d}^{k}(\hat{\lambda})}V_{\hat{M}}$, and $\bigoplus_{\hat{M}\in\operatorname{SSYT}_{d}^{k}(\hat{\lambda})}\tilde{V}_{\hat{M}}$ from part $(ii)$ map bijectively onto their image inside the quotient $R_{\infty,k}$ from Definition \ref{defRinfk}, which is $R_{\infty,k}^{0}$ from part $(v)$ of Theorem \ref{decomRinf} that is homogeneous of degree $d$. The sub-representation $V_{\hat{M}}$ of $R_{\infty,k+1}^{0}$ that is associated with $\hat{M}\in\operatorname{SSYT}_{d}^{k+1}(\hat{\lambda})\setminus\operatorname{SSYT}_{d}^{k}(\hat{\lambda})$ has a vanishing image in $R_{\infty,k}^{0}$.
\item Any element of a sub-representation $\tilde{V}_{\hat{M}}$ from part $(ii)$ is a product of a polynomial and a symmetric function from $\Lambda$.
\end{enumerate} \label{images}
\end{thm}
Part $(iii)$ of Theorem \ref{images} yields the decomposition mentioned in Remark \ref{monfunc}, while part $(iv)$ concerns the property from Remark \ref{polstimessym}.

\begin{proof}
We recall from part $(iii)$ of Theorem \ref{FMTdecom} that when $n>2d$, the decompositions for $\mathbb{Q}[\mathbf{x}_{n+1}]_{d}$ are obtained from those of $\mathbb{Q}[\mathbf{x}_{n}]_{d}$ by taking $\hat{\iota}$-images. Indeed, this is Proposition \ref{plus1iota} for the representations of the form $V_{C}^{\vec{h}(C,I)}$, part $(iii)$ of Theorem \ref{FMTdecom} for the $V_{M}$'s, and for the $\tilde{V}_{M}$'s this is via the map from part $(iv)$ of Theorem \ref{propbases}. We thus write the decompositions of $\mathbb{Q}[\mathbf{x}_{n+1}]_{d}$ using $\hat{\iota}$-images of tableaux $M\in\operatorname{SSYT}_{d}(\lambda)$ and $C\in\operatorname{CCT}(\lambda)$, for partitions $\lambda \vdash n$.

Then, inside $V_{\hat{\iota}M}$, or $\tilde{V}_{\hat{\iota}M}$, or $V_{\hat{\iota}C}^{\vec{h}(\hat{\iota}C,I)}$, any basis element that is based on $\iota T$ for $T\in\operatorname{SYT}(\lambda)$ is determined, via Proposition \ref{forstab}, by its image in $\mathbb{Q}[\mathbf{x}_{n}]_{d}$ obtained from the substitution $x_{n+1}=0$. This means that the linear relation spanning an element of a sub-representation of $\mathbb{Q}[\mathbf{x}_{n}]_{d}$ in terms of the components of the other decompositions keeps its form when we replace each index by its image under $\iota$ or $\hat{\iota}$.

We thus take a basis element of one of the representations $V_{\hat{M}}$, $\tilde{V}_{\hat{M}}$, or $V_{\hat{C}}^{\vec{h}(\hat{C},I)} \subseteq R_{\infty,I}^{\mathrm{hom},0}$ inside $\tilde{\Lambda}_{d}$, for some $\hat{\lambda}\vdash\infty$ and  $\hat{M}\in\operatorname{SSYT}_{d}(\hat{\lambda})$ or $\hat{C}\in\operatorname{CCT}(\hat{\lambda})$ and a multi-set $I$, and let $\hat{T}\in\operatorname{SYT}(\hat{\lambda})$ be the corresponding index of that basis element. We take $n>2d$ such that $\hat{T}$ and $\hat{M}$ or $\hat{C}$ are images of tableaux $T$ and $M$ or $C$ of shape $\lambda \vdash n$, and recall from Propositions \ref{limSpecht} and \ref{forstab} that for such $n$, the associated element $F_{\hat{M},\hat{T}}$, or $F_{\hat{M}^{0},\hat{T}}m_{\hat{M}}$, or $F_{\hat{C},\hat{T}}^{I,\mathrm{hom}}$ is determined by its finite counterpart $F_{M,T}$, or $F_{M^{0},T}m_{M}$, or $F_{C,T}^{I,\mathrm{hom}}$ respectively.

We thus express the finite counterpart in terms of the other decompositions of $\mathbb{Q}[\mathbf{x}_{n}]_{d}$, and apply the same considerations for the parts of the bases of the other decompositions inside $\tilde{\Lambda}_{d}$ that are based on elements of $\operatorname{SYT}(\hat{\lambda})$ that arise from tableaux in $\operatorname{SYT}(\lambda)$ with our $\lambda \vdash n$. By what we have shown, the linear relations among the sub-representations of $\mathbb{Q}[\mathbf{x}_{n}]_{d}$ remain valid for the corresponding elements inside $\tilde{\Lambda}_{d}$, showing that each basis element showing up in one decomposition there is spanned by those from the other decompositions. Hence the image is the same, establishing part $(i)$.

Part $(ii)$ is now an immediate consequence of part $(i)$, via taking the isotypical components based on the irreducible representations determined in Theorem \ref{repsinf}, and the first assertion in part $(iii)$ follows by combining with Definition \ref{Adinf} and Theorem \ref{decomRinf}. Noting that if $\hat{M}\in\operatorname{SSYT}_{d}^{k+1}(\hat{\lambda})\setminus\operatorname{SSYT}_{d}^{k}(\hat{\lambda})$ has content $\mu$ then $\tilde{\Lambda}_{\mu}$, which contains $V_{\hat{M}}$ by part $(ii)$ of Proposition \ref{limSpecht}, maps bijectively onto its image in $R_{\infty,k+1}$, but its image in $R_{\infty,k+1}$ vanishes entirely, produces the remaining assertion there.

Finally, the assertion of part $(iv)$ holds for basis elements of $\tilde{V}_{\hat{M}}$ by Definition \ref{repswithmon} and Lemma \ref{sameiota}. The fact that the symmetric function is the same one $m_{\hat{M}}\in\Lambda$ for all such basis elements, and all the multipliers are polynomials from $\mathbb{Q}[\mathbf{x}_{\infty}]$, yields the assertion in general. This proves the theorem.
\end{proof}
In fact, one could imitate the proof of Theorem \ref{propbases} and obtain all the parts of Theorem \ref{images} directly from Theorem \ref{decomRinf}, Propositions \ref{ExtVMlim} and \ref{decomVM0mM}, and Corollary \ref{multCI}, without invoking the theorem itself. 

\begin{rmk}
We remark that in both Theorem \ref{propbases} and Theorem \ref{images} we could get two additional decompositions, one by replacing the product $\prod_{r}e_{r}^{h_{r}}$ where $\{h_{r}\}_{r\geq1}$ is the vector $\vec{h}(C,I)$ or $\vec{h}(\hat{C},I)$ for an element of $A_{d}(\lambda)$ or $A_{d}(\hat{\lambda})$ by the ``maximal'' monomial symmetric function showing up in it, and another by replacing $m_{M}$ or $m_{\hat{M}}$ for an element of $\operatorname{SSYT}_{d}(\lambda)$ or $\operatorname{SSYT}_{d}(\hat{\lambda})$ by the product of elementary symmetric functions in which it shows up as the ``maximal'' monomial one. However, these decompositions do not have the properties that were established in these results, so we content ourselves with the three decompositions already mentioned. \label{otherdecoms}
\end{rmk}

\begin{ex}
For $\hat{\lambda}$ from Example \ref{exinftab}, recall from Example \ref{421d6} that when going from $\lambda=421\vdash7$ and $d=6$ to $\lambda_{+}=521\vdash8$ (both of which yield $\hat{\lambda}$ via Lemma \ref{limtab}), the representations are already obtained via replacing indices by their $\hat{\iota}$-images. By replacing each summand there by the one whose index is the infinite semi-standard Young tableau (or cocharge tableau) constructed through Lemma \ref{limtab}, we get the decompositions of the $\mathcal{S}^{\hat{\lambda}}$-isotypical component of $\tilde{\Lambda}_{6}$ and their properties, as in Theorem \ref{images}. \label{infexd6}
\end{ex}

\begin{ex}
Consider the decomposition from Example \ref{infexd6} that is based on the first one from Example \ref{421d6}, and denote the seven corresponding elements of $\operatorname{SSYT}(\hat{\lambda})$ by $\hat{M}_{j}$, $1 \leq j \leq 7$. We take $\hat{T}\in\operatorname{SYT}(\hat{\lambda})$ to be the tableau in which the first column contains 1, 2, and 3 and the second one contains 4 and 5, and we wish to make the comparison from Lemma \ref{difmodf-1} and Proposition \ref{decomVM0mM} explicit. If $1 \leq j \leq 3$ then $f_{\hat{M}_{j}}=0$, so that $F_{\hat{M}_{j},\hat{T}}=F_{\hat{M}_{j}^{0},\hat{T}}$ and $m_{\hat{M}_{j}}=1$. Next, $m_{\hat{M}_{4}}=e_{1}$, the product of $F_{\hat{M}_{4}^{0},\hat{T}}$ with the part $\sum_{i=6}^{\infty}x_{i}$ of $e_{1}$ yield $F_{\hat{M}_{4},\hat{T}}$, the part $x_{1}+x_{2}+x_{3}$ yields $F_{\hat{M}_{1},\hat{T}}$ plus another contribution, and the latter combines with the product arising from $x_{4}+x_{5}$ to produce $F_{\hat{M}_{2},\hat{T}}$. We have $m_{\hat{M}_{5}}=e_{1}$ as well, the products of $F_{\hat{M}_{5}^{0},\hat{T}}$ with $\sum_{i=6}^{\infty}x_{i}$ and with $x_{4}+x_{5}$ produce $F_{\hat{M}_{5},\hat{T}}$ plus a part of $F_{\hat{M}_{3},\hat{T}}$, that with $x_{1}+x_{2}+x_{3}$ yields a part of $F_{\hat{M}_{2},\hat{T}}$, and the remaining parts of $F_{\hat{M}_{2},\hat{T}}$ and $F_{\hat{M}_{3},\hat{T}}$ cancel with each other. The next case involves $m_{\hat{M}_{6}}=p_{2}$, multiplying $F_{\hat{M}_{6}^{0},\hat{T}}$ by $\sum_{i=6}^{\infty}x_{i}^{2}$ and with $x_{4}^{2}+x_{5}^{2}$ gives $F_{\hat{M}_{6},\hat{T}}$ plus the same part of $F_{\hat{M}_{3},\hat{T}}$, and if we write $x_{1}^{2}+x_{2}^{2}+x_{3}^{2}$ as $h_{3}(x_{1},x_{2},x_{3})-e_{3}(x_{1},x_{2},x_{3})$ then the second summand completes $F_{\hat{M}_{3},\hat{T}}$ and the first one yields $F_{\hat{M}_{1},\hat{T}}$. Finally, from $m_{\hat{M}_{7}}=e_{2}$, the product of $F_{\hat{M}_{7}^{0},\hat{T}}$ with $\sum_{i=6}^{\infty}\sum_{j=i+1}^{\infty}x_{i}x_{j}$ equals $F_{\hat{M}_{7},\hat{T}}$, the one with $(x_{4}+x_{5})\sum_{i=6}^{\infty}x_{i}$ and $x_{4}x_{5}$ combine to give $F_{\hat{M}_{5},\hat{T}}$, from $(x_{1}+x_{2}+x_{3})\sum_{i=6}^{\infty}x_{i}$ we get $F_{\hat{M}_{4},\hat{T}}$, and those arising from $e_{3}(x_{1},x_{2},x_{3})$ and $(x_{1}+x_{2}+x_{3})(x_{4}+x_{5})$ produce $F_{\hat{M}_{5},\hat{T}}$. \label{relsd6ex}
\end{ex}
The equalities from Example \ref{relsd6ex} are $F_{\hat{M}_{4}^{0},\hat{T}}e_{1}=F_{\hat{M}_{4},\hat{T}}+F_{\hat{M}_{2},\hat{T}}+F_{\hat{M}_{1},\hat{T}}$, $F_{\hat{M}_{5}^{0},\hat{T}}e_{1}=F_{\hat{M}_{5},\hat{T}}+F_{\hat{M}_{3},\hat{T}}+F_{\hat{M}_{2},\hat{T}}$, $F_{\hat{M}_{6}^{0},\hat{T}}p_{2}=F_{\hat{M}_{6},\hat{T}}+F_{\hat{M}_{3},\hat{T}}+F_{\hat{M}_{1},\hat{T}}$, and $F_{\hat{M}_{7}^{0},\hat{T}}e_{2}=F_{\hat{M}_{7},\hat{T}}+F_{\hat{M}_{5},\hat{T}}+F_{\hat{M}_{4},\hat{T}}+F_{\hat{M}_{2},\hat{T}}$, together with the three trivial ones.

\medskip

We deduce an immediate consequence of Theorems \ref{propbases} and \ref{images}.
\begin{cor}
Let $d\geq0$, $\lambda \vdash n$, and $\hat{\lambda}\vdash\infty$ be fixed.
\begin{enumerate}[$(i)$]
\item We have the equalities $|A_{d}(\lambda)|=|\operatorname{SSYT}_{d}(\lambda)|$ and $|A_{d}(\hat{\lambda})|=|\operatorname{SSYT}_{d}(\hat{\lambda})|$.
\item Multiplying the first number from part $(i)$ by the dimension of $\mathcal{S}^{\lambda}$ and summing over $\lambda$ yields $\binom{n+d-1}{d}$.
\item For $k\geq1$ we also get $|A_{d}^{k,0}(\lambda)|=|\operatorname{SSYT}_{d}^{k}(\lambda)|$ and $|A_{d}^{k}(\hat{\lambda})|=|\operatorname{SSYT}_{d}^{k}(\hat{\lambda})|$.
\item Given a content $\mu$ of sum $d$, let $m_{h}$ be the multiplicity of $h\geq0$ in $\mu$, so that $\sum_{h}m_{h}=n$. Then multiplying $|\operatorname{SSYT}_{\mu}(\lambda)|$ by the dimension of $\mathcal{S}^{\lambda}$ and taking the sum over $\lambda \vdash n$ produces $n!\big/\prod_{h}m_{h}!$.
\end{enumerate} \label{lincombM}
\end{cor}

\begin{proof}
The equalities in parts $(i)$ and $(iii)$ follow from the fact that in each one there is a (completely reducible) representation, and an isotypical part of it, such that both sides count, by parts $(i)$, $(iii)$, and $(ii)$ of Theorem \ref{propbases} (the latter with $s=0$) and parts $(ii)$ and $(iii)$ of Theorem \ref{images}, the multiplicity of the corresponding irreducible representation there. Part $(ii)$ is established when we compare the dimension of $\mathbb{Q}[\mathbf{x}_{n}]_{d}$ with the sum of the dimensions in the representations in any of its decompositions. For part $(iv)$ we recall the decomposition of $\mathbb{Q}[\mathbf{x}_{n}]_{\mu}$ from part $(ii)$ of Theorem \ref{FMTdecom}, note that its dimension is the asserted value, and apply the same argument proving part $(ii)$. This proves the corollary.
\end{proof}
The first equality in part $(i)$ of Corollary \ref{lincombM} is already a consequence of Propositions 1.12 and 1.13 of \cite{[Z3]}, and parts $(ii)$ and $(iv)$ there follow as well.

\begin{rmk}
Due to the equalities from Corollary \ref{lincombM}, it may be interesting to see whether some natural bijections exist between these sets. The current proof does not give one, but rather yields, for every $\lambda \vdash n$ and $d$, an invertible matrix with entries from $A_{d}(\lambda)\times\operatorname{SSYT}_{d}(\lambda)$ that expresses the basis element $F_{C,T}^{I,\mathrm{hom}}$ for $(C,I) \in A_{d}(\lambda)$ and $T\in\operatorname{SYT}(\lambda)$ as a linear combination of $F_{M,T}$ over $M\in\operatorname{SSYT}_{d}(\lambda)$. This matrix is independent of $T$ since it represents a map of representations, and it follows from the proof of Theorem \ref{images} that the matrix is the same one for all $n>2d$, and coincides with the one indexed on $A_{d}(\hat{\lambda})\times\operatorname{SSYT}_{d}(\hat{\lambda})$ with $\hat{\lambda}\vdash\infty$. Replacing the $F_{M,T}$'s and the $F_{\hat{M},\hat{T}}$'s by the respective elements $F_{M^{0},T}m_{M}\in\tilde{V}_{M}$ or $F_{\hat{M}^{0},\hat{T}}m_{\hat{M}}\in\tilde{V}_{\hat{T}}$ multiplies this by another matrix, which can be taken to be upper triangular via the proofs of Lemma \ref{difmodf-1} and Proposition \ref{decomVM0mM}. \label{noexpbij}
\end{rmk}

The condition $n>2d$ for the stability of the matrix in Remark \ref{noexpbij} (which can be weakened to $n\geq2d$ as always) is the one arising from Theorem \ref{FMTdecom} and Proposition \ref{plus1iota}, as well as part $(iv)$ of Theorem \ref{propbases}. The bijectivity between the numbers of representations there is inverted by the substitution $x_{n+1}=0$, and without this inequality, this substitution, also in polynomials that are obtained from tableaux containing vanishing entries, may behave as if it comes from expressions arising from tableaux that are no longer semi-standard. In such cases calculations like those appearing in the proof of Lemma 2.32 of \cite{[Z3]}, which may be non-trivial, may show up.

We now turn to the simplest examples.
\begin{ex}
For $d=0$ there is only the empty multi-set $I$, and the (finite or infinite) semi-standard Young tableau having a single row of zeros, which is also a cocharge tableau (finite or infinite). Then all the expressions from Theorems \ref{propbases} and \ref{images} produce the trivial representation $\mathbb{Q}$, and the numbers from Corollary \ref{lincombM} all equal 1. This is so in all the quotients as well. \label{exd0}
\end{ex}
As the case from Example \ref{exd0} is trivial, we turn to the first non-trivial one.
\begin{ex}
When $d=1$ and $n\geq2$, all the decompositions from Theorem \ref{propbases} are the same, since $\mathbb{Q}[\mathbf{x}_{n}]_{1}$ is the direct sum of two non-isomorphic irreducible representations. One of them is the standard one, associated with the (cocharge) tableau of shape $(n-1,1)$ and a single entry 1 in the second row, and the other one is $\mathbb{Q}e_{1}$, either arising from the cocharge tableau from Example \ref{exd0} with the (multi-)set $I:=\{1\}$, or associated, either as $V_{M}$ or as $\tilde{V}_{M}$, with the tableau of length 1 have $n-1$ zeros and a single instance of 1. When $n=1$ the former representation disappears, and the matrix from Remark \ref{noexpbij} is the identity matrix (of size 1 when $n=1$ and size 2 otherwise). Going to Theorem \ref{images}, we get the standard representation as an irreducible representation of co-dimension 1 inside $\mathbb{Q}[\mathbf{x}_{\infty}]_{1}$, while the trivial one is spanned by $e_{1}\in\Lambda_{1}\subseteq\tilde{\Lambda}_{1}$, with the same identity matrix of size 2 in Remark \ref{noexpbij}. Note that an element of $\mathbb{Q}[\mathbf{x}_{\infty}]_{1}$ that is not inside this irreducible sub-representation generates all of $\mathbb{Q}[\mathbf{x}_{\infty}]_{1}$. The same occurs in the quotients $R_{n,k,s}$ and $R_{\infty,k}$ for $k\geq2$, except for $s=n \leq k$ where the $e_{1}$ part disappears, and everything vanishes if $k=1$. \label{exd1}
\end{ex}

\medskip

Example \ref{exd1} shows that in general, the image from part $(i)$ of Theorem \ref{images} is not all of $\tilde{\Lambda}_{d}$, and the representation $R_{\infty,k}^{0}$ from Theorem \ref{decomRinf} is not all of $R_{\infty,k}$ from Definition \ref{defRinfk}, because $\tilde{\Lambda}_{d}$ and the latter quotients are not completely reducible (hence so is $\tilde{\Lambda}$. However, the fact that all the completely reducible sub-representations that we constructed as sums in this way ended up being the same hints at the following result.
\begin{thm}
Fix $d\geq0$, and recall the notation from Definition \ref{repswithmon}.
\begin{enumerate}[$(i)$]
\item The maximal completely reducible sub-representation of $\mathbb{Q}[\mathbf{x}_{\infty}]_{d}$ is the direct sum $\mathbb{Q}[\mathbf{x}_{\infty}]_{d}^{0}:=\bigoplus_{\hat{\lambda}\vdash\infty}\bigoplus_{\hat{M}\in\operatorname{SSYT}_{d}(\hat{\lambda}),\ \hat{M}^{0}=\hat{M}}V_{\hat{M}}$.
\item For a content $\mu$ of sum $d$, the intersection of the sub-representation from part $(ii)$ with $\mathbb{Q}[\mathbf{x}_{\infty}]_{\mu}$ is $\mathbb{Q}[\mathbf{x}_{\infty}]_{\mu}^{0}:=\bigoplus_{\hat{\lambda}\vdash\infty}\bigoplus_{\hat{M}\in\operatorname{SSYT}_{\mu}(\hat{\lambda}),\ \hat{M}^{0}=\hat{M}}V_{\hat{M}}$. It is the maximal completely reducible sub-representation of $\mathbb{Q}[\mathbf{x}_{\infty}]_{\mu}$.
\item The sub-representation $\tilde{\Lambda}_{d}^{0}=\bigoplus_{\hat{\lambda}\vdash\infty}\bigoplus_{\hat{M}\in\operatorname{SSYT}_{d}(\hat{\lambda})}V_{\hat{M}}$ from Proposition \ref{ExtVMlim} and part $(i)$ of Theorem \ref{images} is the maximal completely reducible sub-representation of $\tilde{\Lambda}_{d}$.
\item The subspace $\tilde{\Lambda}_{\mu}^{0}=\bigoplus_{\hat{\lambda}\vdash\infty}\bigoplus_{\hat{M}\in\operatorname{SSYT}_{\mu}(\hat{\lambda})}V_{\hat{M}}$ from Proposition \ref{ExtVMlim} is the maximal completely reducible sub-representation there.
\item The representation $R_{\infty,k,d}^{0}$ from part $(v)$ of Theorem \ref{decomRinf} is the maximal completely reducible sub-representation of the homogenous part $R_{\infty,k,d}$ of the quotient $R_{\infty,k}$ from Definition \ref{defRinfk} for every $k\geq1$. It is the image of $\tilde{\Lambda}_{d}^{0}$ under the projection from $\tilde{\Lambda}$ onto $R_{\infty,k}$.
\end{enumerate} \label{compred}
\end{thm}

\begin{proof}
The asserted direct sum in part $(i)$ is clearly completely reducible, and contained in $\mathbb{Q}[\mathbf{x}_{\infty}]_{d}$ by Lemma \ref{sameiota}. We thus consider an irreducible representation $U$ of $S_{\infty}$ or $S_{\mathbb{N}}$ inside $\mathbb{Q}[\mathbf{x}_{\infty}]_{d}$, and we have to show that it is contained in the asserted sum.

So take some element $F \in U$. Since $F$ is a polynomial, it contains only finitely many variables, and let $n$ be large enough such that $F\in\mathbb{Q}[\mathbf{x}_{n}]$ and $n>2d$. Since $S_{n} \subseteq S_{\infty} \subseteq S_{\mathbb{N}}$, the space generated by the $S_{n}$-images of $F$ is contained in $U$. By Theorem \ref{FMTdecom}, the latter space is, as a representation of $S_{n}$, a sub-representation of $\bigoplus_{\lambda \vdash n}\bigoplus_{M\in\operatorname{SSYT}_{d}(\lambda)}V_{M}$, and we need to see which of these may give a non-zero contribution to our subspace of $U$.

But that theorem shows, via Proposition \ref{forstab}, that when we increase $n$ to $n+1$, if $V_{M}$ had a contribution to this part of $U$ then $V_{\hat{\iota}M}$ will have a similar contribution to the corresponding part of $U$. By taking the limit as in that corollary, we deduce that $U$ must be contained in the direct sum from part $(iv)$ there, which we write as $\bigoplus_{\hat{M}\in\operatorname{SSYT}_{d}(\hat{\lambda})}\tilde{V}_{\hat{M}}$ via part $(ii)$ of Theorem \ref{images}.

We now recall from Remark \ref{polstimessym} that $\tilde{\Lambda}$ can be written as the free module $\bigoplus_{\alpha}\mathbb{Q}[\mathbf{x}_{\infty}]m_{\alpha}$ over $\mathbb{Q}[\mathbf{x}_{\infty}]$, where $m_{\alpha}$ runs over all the monomial elements of $\Lambda$. As every element of $\tilde{V}_{\hat{M}}$ is a polynomial times $m_{\hat{M}}\in\Lambda$ (either by Definition \ref{repswithmon} or via part $(iv)$ of Theorem \ref{images}), we deduce that $\tilde{V}_{\hat{M}}$ is contained in the part $\mathbb{Q}[\mathbf{x}_{\infty}]m_{\hat{M}}$ of that direct sum.

The direct sum property implies that if an element of the direct sum from Theorem \ref{images} contains no multiple of $m_{\alpha}$ for some index $\alpha$, then it gets no contribution from any $\tilde{V}_{\hat{M}}$ with $m_{\hat{M}}=m_{\alpha}$. In particular, the intersection of $\mathbb{Q}[\mathbf{x}_{\infty}]$ with our direct sum consists of the direct sum of those $\tilde{V}_{\hat{M}}$ for which $m_{\hat{M}}=1$, namely for which $f_{\hat{M}}=0$. Since we assumed that our representation $U$ is contained in that intersection, it is indeed contained in the asserted direct sum, as desired. This proves part $(i)$, from which part $(ii)$ follows via a decomposition as in Theorem \ref{FMTdecom}.

Turning to part $(iii)$, we consider the decomposition of $\tilde{\Lambda}$ as $\bigoplus_{\alpha}\mathbb{Q}[\mathbf{x}_{\infty}]m_{\alpha}$ again, and recall from Remark \ref{polstimessym} that the action of $S_{\infty}$ or $S_{\mathbb{N}}$ affects only the coefficients in this presentation. Hence any homomorphism $\pi:\tilde{\Lambda}\to\mathbb{Q}[\mathbf{x}_{\infty}]$ of $\mathbb{Q}[\mathbf{x}_{\infty}]$-modules that sends each monomial symmetric function $m_{\alpha}\in\Lambda$ to some constant $c_{\alpha}\in\mathbb{Q}$ respects the group action. For any $\alpha$, we define $\pi_{\alpha}$ to be the equivariant homomorphism obtained in this manner, where $m_{\alpha}$ is taken to the constant 1, and any other monomial element is sent to 0.

Consider now $F\in\tilde{\Lambda}$, which we can write in that decomposition as $\sum_{\alpha}F_{\alpha}m_{\alpha}$ with $F_{\alpha}\in\mathbb{Q}[\mathbf{x}_{\infty}]$ for polynomials $F_{\alpha}\in\mathbb{Q}[\mathbf{x}_{\infty}]$ such that $F_{\alpha}=0$ for all but finitely many $\alpha$'s (in fact, since $F\in\tilde{\Lambda}$ has a finite degree, we can restrict to monomial elements of $\Lambda$ up to that degree, and there are finitely many of those). Since applying $\pi_{\alpha}$, for any $\alpha$, to both sides shows that $F_{\alpha}=\pi_{\alpha}(F)$, we deduce the formula $F=\sum_{\alpha}\pi_{\alpha}(F)m_{\alpha}$.

Let now $U$ be an irreducible representation inside $\tilde{\Lambda}_{d}$, and take $F \in U$. Then each $m_{\alpha}$ for $\alpha$ with $\pi_{\alpha}(F)\neq0$ has degree at most $d$, and if we write this degree as $d-e$ for some $0 \leq e \leq d$ then $\pi_{\alpha}(F)$ must have degree $e$ (since all the components $\sum_{\alpha}\pi_{\alpha}(F)m_{\alpha}$ must be in $\tilde{\Lambda}_{d}$ once the sum $F$ is there). Moreover, as $\pi_{\alpha}$ is a map of representations and $U$ is irreducible, the image $\pi_{\alpha}(U)\subseteq\mathbb{Q}[\mathbf{x}_{\infty}]_{e}$ must either vanish or be an irreducible representation that is isomorphic to $U$.

But part $(i)$ implies that if $\pi_{\alpha}(U)$ is non-zero, then it is contained in the direct sum $\bigoplus_{\hat{\lambda}\vdash\infty}\bigoplus_{\hat{N}\in\operatorname{SSYT}_{e}(\hat{\lambda}),\ \hat{N}^{0}=\hat{N}}V_{\hat{N}}$. We thus take $F \in U$, and write $\pi_{\alpha}(F)\in\mathbb{Q}[\mathbf{x}_{\infty}]_{e}$ as the sum over $\hat{\lambda}\vdash\infty$ and over such tableaux $\hat{N}$ of elements $\pi_{\alpha,\hat{N}}(F) \in V_{\hat{N}}\subseteq\mathbb{Q}[\mathbf{x}_{\infty}]_{e}$. Our presentation of $F$ and of $\pi_{\alpha}(F)$ implies the equality $F=\sum_{\alpha}\sum_{\hat{\lambda}\vdash\infty}\sum_{\hat{N}}\pi_{\alpha,\hat{N}}(F)m_{\alpha}$.

So take such $\alpha$, $\hat{\lambda}$, and $\hat{N}$, and consider the infinite semi-standard Young tableau $\hat{M}$ of shape $\hat{\lambda}$ that has the same entries as $\hat{N}$, and the multiplicity of each $h$ in Definition \ref{tabinf} is the one in which it appears in $\alpha$. Then we have $\hat{M}\in\operatorname{SSYT}_{d}(\hat{\lambda})$, and as Definition \ref{repswithmon} yields $\hat{N}=\hat{M}^{0}$ and $m_{\hat{M}}=m_{\alpha}$, we deduce that multiplying any element of $V_{\hat{N}}$ by $m_{\alpha}$ produces an element of $\tilde{V}_{\hat{M}}$.

It thus follows that for every such $\alpha$, $\hat{\lambda}$, and $\hat{N}$, the product $\pi_{\alpha,\hat{N}}(F)m_{\alpha}$ lies in one of the summands of the direct sum $\bigoplus_{\hat{M}\in\operatorname{SSYT}_{d}(\hat{\lambda})}\tilde{V}_{\hat{M}}$ from, e.g., part $(ii)$ of Proposition \ref{decomVM0mM}. Hence the sum $F$ over $\hat{N}$, $\hat{\lambda}$, and $\alpha$ also lies inside that sum, which is the desired sub-representation of $\tilde{\Lambda}_{d}$ by part $(i)$ of Theorem \ref{images} and clearly completely reducible by part $(ii)$ there. As we just showed that every irreducible sub-representation of $\tilde{\Lambda}_{d}$ is contained in it, the maximality from part $(iii)$ is also established.

part $(iv)$ follows from part $(iii)$ via the intersection from Proposition \ref{ExtVMlim} just like part $(ii)$ did from part $(i)$. Finally, the argument proving part $(iii)$ of Theorem \ref{images} implies that given a content $\mu$, the image of $\tilde{\Lambda}_{\mu}$ under the projection onto $R_{\infty,k}$ vanishes of $\mu$ contains an entry that equals $k$ or larger, and is an isomorphic to $\tilde{\Lambda}_{\mu}$ as representations otherwise. Hence part $(v)$ also follows from parts $(iii)$ and $(iv)$. This proves the theorem.
\end{proof}
Note that the symbols $\mathbb{Q}[\mathbf{x}_{\infty}]_{d}^{0}$ and $\mathbb{Q}[\mathbf{x}_{\infty}]_{\mu}^{0}$ from parts $(i)$ and $(ii)$ of Theorem \ref{compred} are similar to $\tilde{\Lambda}_{d}^{0}$ and $\tilde{\Lambda}_{\mu}^{0}$, all of which will be generalized in Definitions \ref{Qxinfdf} and \ref{Lambdafilt} below, and should not be confused with $\mathbb{Q}[\mathbf{x}_{n}]_{d}^{0}$ and $\mathbb{Q}[\mathbf{x}_{n}]_{d}^{0}$, for finite $n$ and $f=0$, from Definition \ref{repswithmon}.

\begin{rmk}
The general theory of the representations of $S_{\infty}$ involves the Thoma parameters, initiated in \cite{[T]} (see more in the book \cite{[BO]}). From this point of view, all the representations we constructed should have the same Thoma parameter as the trivial representation (as they resemble the limits from the Vershik--Kerov Theorem, as in, e.g., Theorem 6.16 of \cite{[BO]}). Moreover, as Proposition \ref{samereps} shows, our representations are closer to those that are called \emph{tame} in, e.g., \cite{[O]} and \cite{[MO]}, but note that these papers discuss unitary representations. Indeed, as unitary representations are always completely reducible, and Theorem \ref{compred} shows that $\tilde{\Lambda}_{d}$ for $d\geq1$ does not have this property, our representations are not within direct grasp of these theories. \label{Thoma}
\end{rmk}

\section{The Structure of the Full Representations \label{FullRep}}

Theorem \ref{compred} implies that increasing $\mathbb{Q}[\mathbf{x}_{\infty}]_{d}^{0}$, $\mathbb{Q}[\mathbf{x}_{\infty}]_{\mu}^{0}$, $\tilde{\Lambda}_{d}^{0}$, and $\tilde{\Lambda}_{\mu}^{0}$ into larger representations of $\mathbb{Q}[\mathbf{x}_{\infty}]_{d}$, $\mathbb{Q}[\mathbf{x}_{\infty}]_{\mu}$, $\tilde{\Lambda}_{d}$, and $\tilde{\Lambda}_{\mu}$ respectively cannot be done by adding irreducible components. As the completely reducible sub-representations are direct limits as $n\to\infty$ of $\mathbb{Q}[\mathbf{x}_{n}]_{d}$ or $\mathbb{Q}[\mathbf{x}_{n}]_{\mu}$ and the larger ones are, in some sense, inverse limits of them, we introduce another notion in the finite $n$ setting.
\begin{defn}
Take a general polynomial $F\in\mathbb{Q}[\mathbf{x}_{n}]_{d}$, and by writing the latter space as in Theorem \ref{FMTdecom}, we can write $F$ as $\sum_{\lambda \vdash n}\sum_{M\in\operatorname{SSYT}_{d}(\lambda)}F_{M}$, with $F_{M} \in V_{M}\subseteq\mathbb{Q}[\mathbf{x}_{n}]_{d}$ for every $\lambda \vdash n$ and $M\in\operatorname{SSYT}_{d}(\lambda)$. We then define the \emph{$n$-support} $\operatorname{supp}_{n}F$ of $F$ to be $\bigcup_{\lambda \vdash n}\{M\in\operatorname{SSYT}_{d}(\lambda)\;|\;F_{M}\neq0\}$. \label{nsup}
\end{defn}

\begin{rmk}
It is clear from Definition \ref{nsup} that $F$ lies in $\mathbb{Q}[\mathbf{x}_{n}]_{\mu}$ for some content $\mu$ of sum $d$ if and only if $\operatorname{supp}_{n}F\cap\operatorname{SSYT}_{d}(\lambda)\subseteq\operatorname{SSYT}_{\mu}(\lambda)$ for every $\lambda \vdash n$. Combining with Definition \ref{repswithmon}, we obtain that $F$ lies in $\mathbb{Q}[\mathbf{x}_{n}]_{d}^{0}$ for some $f\geq0$ if and only if every tableau $M\in\operatorname{supp}_{n}F$ satisfies $f_{M} \leq f$. \label{contsup}
\end{rmk}

By taking $F\in\mathbb{Q}[\mathbf{x}_{n}]_{d}$ and viewing it as contained in the larger space $\mathbb{Q}[\mathbf{x}_{n+1}]_{d}$, we would like to relate the $n$-support of $F$ with its $(n+1)$-support. Lemma \ref{sameiota} yields the following result.
\begin{lem}
If $F\in\mathbb{Q}[\mathbf{x}_{n}]_{d}^{0}$ then $\operatorname{supp}_{n+1}F=\{\hat{\iota}M\;|\;M\in\operatorname{supp}_{n}F\}$. In particular $\mathbb{Q}[\mathbf{x}_{n}]_{d}^{0}$ is contained in $\mathbb{Q}[\mathbf{x}_{n+1}]_{d}^{0}$, and if $n>2d$ then the former generates the latter over $\mathbb{Q}[S_{n+1}]$. \label{Qnd0iota}
\end{lem}

\begin{proof}
We write $F=\sum_{M\in\operatorname{supp}_{n}F}F_{M}$ as in Definition \ref{nsup}, and get $f_{M}=0$ for all $M\in\operatorname{supp}_{n}F$ via Remark \ref{contsup}. We span each $F_{M}$ for such $M$, with $\operatorname{sh}(M)=\lambda \vdash n$, via the basis $F_{M,T}$ with $T\in\operatorname{SYT}(\lambda)$ (via part $(ii)$ of Theorem \ref{repsSpecht}), and then Lemma \ref{sameiota} implies that each such $F_{M,T}$ also equals $F_{\hat{\iota}M,\iota T}$. The expansion of $F_{M}$ in the former basis thus presents it as an element of $V_{\hat{\iota}M}$ (which is non-zero since $M\in\operatorname{supp}_{n}F$), and thus the expansion of $F$ implies the first assertion. The containment is an immediate consequence of this assertion (Remark \ref{contsup} again), and the generation follows from part $(iii)$ of Theorem \ref{FMTdecom} and Remark \ref{fiota}. This proves the lemma.
\end{proof}
As always, the inequality $n\geq2d$ suffices for the generation in Lemma \ref{Qnd0iota}.

\medskip

In order to consider larger values of $f$, we begin with some examples.
\begin{ex}
Fix some large $n$, and recall the bases from Example \ref{smallex}. We take some basis elements with $f>0$ showing up there, consider them as elements of $\mathbb{Q}[\mathbf{x}_{n+1}]$, and decompose them using the same bases but now for $n+1$. We then get the equality \[\sum_{i=1}^{n}x_{i}=\frac{n}{n+1}\sum_{i=1}^{n+1}x_{i}+\frac{1}{n+1}\sum_{j=2}^{n}(x_{j}-x_{1})-\frac{n}{n+1}(x_{n+1}-x_{1}),\] so that in particular when we take the form spanning the 1-dimensional complement of $\mathbb{Q}[\mathbf{x}_{n}]_{1}^{0}$ inside $\mathbb{Q}[\mathbf{x}_{n}]_{1}$ and consider its image in $\mathbb{Q}[\mathbf{x}_{n+1}]_{1}$, this image has contributions from the two irreducible representations there. Similarly, $(x_{2}-x_{1})\sum_{i=3}^{n}x_{i}$ decomposes as
\[\frac{n-2}{n-1}(x_{2}-x_{1})\sum_{i=3}^{n+1}x_{i}+\frac{1}{n-1}\sum_{j=4}^{n}(x_{2}-x_{1})(x_{j}-x_{3})-\frac{n-2}{n-1}(x_{2}-x_{1})(x_{n+1}-x_{3}),\] and the for element $\sum_{i=1}^{n}\sum_{j=i+1}^{n}x_{i}x_{j}$, with $f=2$, we have the expression \[\frac{n-1}{n+1}\sum_{i=1}^{n+1}\sum_{j=i+1}^{n+1}x_{i}x_{j}+\frac{1}{n+1}\sum_{j=2}^{n}(x_{j}-x_{1})\sum_{\substack{k=2 \\ k \neq j}}^{n+1}x_{k}-\frac{n}{n+1}(x_{n+1}-x_{1})\sum_{k=2}^{n}x_{k}.\] \label{suprep}
\end{ex}

\begin{ex}
In degree 3 we recall the simpler basis from Remark \ref{simpbasis}, using which we express $(x_{2}^{2}-x_{1}^{2})\sum_{k=3}^{n}x_{k}+(x_{2}^{2}x_{1}-x_{2}x_{1}^{2})$ as \[\frac{n^{2}-n-1}{n^{2}-1}\bigg[\bigg((x_{2}^{2}-x_{1}^{2})\sum_{k=3}^{n+1}x_{k}+(x_{2}^{2}x_{1}-x_{2}x_{1}^{2})\bigg)-(x_{2}^{2}-x_{1}^{2})(x_{n+1}-x_{3})\bigg]+\] \[+\frac{1}{n^{2}-1}\bigg[\sum_{j=4}^{n+1}(x_{2}-x_{1})(x_{j}^{2}-x_{3}^{2})-\bigg((x_{2}-x_{1})\sum_{k=3}^{n+1}x_{k}^{2}-(x_{2}^{2}x_{1}-x_{2}x_{1}^{2})\bigg)\bigg]+\] \[+\frac{1}{n^{2}-1}\sum_{j=4}^{n}(x_{2}^{2}-x_{1}^{2})(x_{j}-x_{3})+\frac{x_{3}^{2}(x_{2}-x_{1})-x_{3}(x_{2}^{2}-x_{1}^{2})+(x_{2}^{2}x_{1}-x_{2}x_{1}^{2})}{n+1},\] and similarly $(x_{2}-x_{1})\sum_{k=3}^{n}x_{k}^{2}-(x_{2}^{2}x_{1}-x_{2}x_{1}^{2})$ equals \[\frac{n^{2}-n-1}{n^{2}-1}\bigg[\bigg((x_{2}-x_{1})\sum_{k=3}^{n+1}x_{k}^{2}-(x_{2}^{2}x_{1}-x_{2}x_{1}^{2})\bigg)-(x_{2}-x_{1})(x_{n+1}^{2}-x_{3}^{2})\bigg]+\]
\[+\frac{1}{n^{2}-1}\bigg[\sum_{j=4}^{n+1}(x_{2}^{2}-x_{1}^{2})(x_{j}-x_{3})-\bigg((x_{2}^{2}-x_{1}^{2})\sum_{k=3}^{n+1}x_{k}+(x_{2}^{2}x_{1}-x_{2}x_{1}^{2})\bigg)\bigg]+\] \[+\frac{1}{n^{2}-1}\sum_{j=4}^{n}(x_{2}-x_{1})(x_{j}^{2}-x_{3}^{2})-\frac{x_{3}^{2}(x_{2}-x_{1})-x_{3}(x_{2}^{2}-x_{1}^{2})+(x_{2}^{2}x_{1}-x_{2}x_{1}^{2})}{n+1}.\] For the basis element with $f=2$ we get \[\sum_{i=1}^{n}\sum_{\substack{j=1 \\ j \neq i}}^{n}x_{i}^{2}x_{j}=\frac{1}{n+1}\sum_{j=2}^{n}\bigg[(x_{j}^{2}-x_{1}^{2})\sum_{\substack{k=2 \\ k \neq j}}^{n+1}x_{k}+(x_{j}-x_{1})\sum_{\substack{k=2 \\ k \neq j}}^{n+1}x_{k}^{2}\bigg]+\] \[+\frac{n-1}{n+1}\sum_{i=1}^{n+1}\sum_{\substack{j=1 \\ j \neq i}}^{n+1}x_{i}^{2}x_{j}-\frac{n}{n+1}\bigg[(x_{n+1}^{2}-x_{1}^{2})\sum_{k=2}^{n}x_{k}+(x_{n+1}-x_{1})\sum_{k=2}^{n}x_{k}^{2}\bigg].\] \label{supd3}
\end{ex}
The elements in Examples \ref{suprep} and \ref{supd3} are basis elements for representations $V_{M}$ for $f_{M}>0$, and thus grow with $n$ in the number of variables appearing in them. Here are some examples of elements whose $n$-support contains more than one element, and that do not grow with $n$.
\begin{ex}
For any $n\geq1$ the monomial $x_{1}$ decomposes as the expression $\frac{1}{n+1}\sum_{i=1}^{n+1}x_{i}-\frac{1}{n+1}\sum_{j=2}^{n+1}(x_{j}-x_{1})$, and similarly if $n\geq3$ then $(x_{2}-x_{1})x_{3}$ equals $\frac{1}{n-1}(x_{2}-x_{1})\sum_{i=3}^{n+1}x_{i}-\frac{1}{n-1}\sum_{j=4}^{n+1}(x_{2}-x_{1})(x_{j}-x_{3})$. We also have \[x_{1}x_{2}=\frac{2}{n(n-1)}\sum_{j=3}^{n}\sum_{k=j+1}^{n+1}(x_{j}-x_{1})(x_{k}-x_{2})+\sum_{j=4}^{n+1}\frac{n+4-2j}{n(n-1)}(x_{2}-x_{1})(x_{j}-x_{3})+\] \[+\frac{1}{n+1}(x_{2}-x_{1})\sum_{k=3}^{n+1}x_{k}-\frac{2}{n^{2}-1}\sum_{j=3}^{n+1}(x_{2}-x_{1})\sum_{\substack{k=2 \\ k \neq j}}^{n+1}x_{k}+\frac{2}{n(n+1)}\sum_{i=1}^{n+1}\sum_{j=i+1}^{n+1}x_{i}x_{j},\] using the bases from Example \ref{smallex}.
\label{supgen}
\end{ex}
A degree 3 analogue of Example \ref{supgen} that uses the basis from Remark \ref{simpbasis} is as follows
\begin{ex}
Three monomials of degree 3 with $f=1$ are expanded as
\[(x_{2}^{2}-x_{1}^{2})x_{3}=-\frac{1}{n+1}[x_{3}^{2}(x_{2}-x_{1})-x_{3}(x_{2}^{2}-x_{1}^{2})+(x_{2}^{2}x_{1}-x_{2}x_{1}^{2})]+\] \[+\frac{n}{n^{2}-1}\bigg[\bigg((x_{2}^{2}-x_{1}^{2})\sum_{k=3}^{n+1}x_{k}+(x_{2}^{2}x_{1}-x_{2}x_{1}^{2})\bigg)-\sum_{j=4}^{n+1}(x_{2}^{2}-x_{1}^{2})(x_{j}-x_{3})\bigg]+\] \[+\frac{1}{n^{2}-1}\bigg[\sum_{j=4}^{n+1}(x_{2}-x_{1})(x_{j}^{2}-x_{3}^{2})-\bigg((x_{2}-x_{1})\sum_{k=3}^{n+1}x_{k}^{2}-(x_{2}^{2}x_{1}-x_{2}x_{1}^{2})\bigg)\bigg],\] \[(x_{2}-x_{1})x_{3}^{2}=+\frac{1}{n+1}[x_{3}^{2}(x_{2}-x_{1})-x_{3}(x_{2}^{2}-x_{1}^{2})+(x_{2}^{2}x_{1}-x_{2}x_{1}^{2})]+\] \[+\frac{1}{n^{2}-1}\bigg[\bigg((x_{2}^{2}-x_{1}^{2})\sum_{k=3}^{n+1}x_{k}+(x_{2}^{2}x_{1}-x_{2}x_{1}^{2})\bigg)-\sum_{j=4}^{n+1}(x_{2}^{2}-x_{1}^{2})(x_{j}-x_{3})\bigg]+\] \[+\frac{n}{n^{2}-1}\bigg[\sum_{j=4}^{n+1}(x_{2}-x_{1})(x_{j}^{2}-x_{3}^{2})-\bigg((x_{2}-x_{1})\sum_{k=3}^{n+1}x_{k}^{2}-(x_{2}^{2}x_{1}-x_{2}x_{1}^{2})\bigg)\bigg],\] and
\[x_{2}^{2}x_{1}-x_{2}x_{1}^{2}=\frac{n-1}{n+1}[x_{3}^{2}(x_{2}-x_{1})-x_{3}(x_{2}^{2}-x_{1}^{2})+(x_{2}^{2}x_{1}-x_{2}x_{1}^{2})]+\] \[+\frac{1}{n+1}\bigg[\bigg((x_{2}^{2}-x_{1}^{2})\sum_{k=3}^{n+1}x_{k}+(x_{2}^{2}x_{1}-x_{2}x_{1}^{2})\bigg)-\sum_{j=4}^{n+1}(x_{2}^{2}-x_{1}^{2})(x_{j}-x_{3})\bigg]+\] \[+\frac{1}{n+1}\bigg[\sum_{j=4}^{n+1}(x_{2}-x_{1})(x_{j}^{2}-x_{3}^{2})-\bigg((x_{2}-x_{1})\sum_{k=3}^{n+1}x_{k}^{2}-(x_{2}^{2}x_{1}-x_{2}x_{1}^{2})\bigg)\bigg],\] while with $f=2$ we get the expansion \[x_{2}^{2}x_{1}+x_{2}x_{1}^{2}=\frac{2}{n(n-1)}\sum_{j=3}^{n}\sum_{k=j+1}^{n+1}\big[(x_{j}-x_{1})(x_{k}^{2}-x_{2}^{2})+(x_{j}^{2}-x_{1}^{2})(x_{k}-x_{2})\big]+\] \[+\!\sum_{j=4}^{n+1}\frac{n+4-2j}{n(n-1)}\big[(x_{2}-x_{1})(x_{j}^{2}-x_{3}^{2})+(x_{2}^{2}-x_{1}^{2})(x_{j}-x_{3})\big]+\frac{2}{n(n\!+\!1)}\!\sum_{1 \leq i<j \leq n+1}\!\!x_{i}x_{j}+\] \[+\frac{1}{n+1}\bigg[\bigg((x_{2}^{2}-x_{1}^{2})\sum_{k=3}^{n+1}x_{k}+(x_{2}^{2}x_{1}-x_{2}x_{1}^{2})\!\bigg)+\bigg(\!(x_{2}-x_{1})\sum_{k=3}^{n+1}x_{k}^{2}-(x_{2}^{2}x_{1}-x_{2}x_{1}^{2})\bigg)\bigg]+\] \[-\frac{2}{n^{2}-1}\!\sum_{j=3}^{n+1}\!\bigg[\!\bigg(\!(x_{j}^{2}-x_{1}^{2})\!\sum_{\substack{k=2 \\ k \neq j}}^{n+1}\!x_{k}+(x_{j}^{2}x_{1}-x_{j}x_{1}^{2})\!\bigg)+\bigg(\!(x_{j}-x_{1})\!\sum_{\substack{k=2 \\ k \neq j}}^{n+1}\!x_{k}^{2}-(x_{j}^{2}x_{1}-x_{j}x_{1}^{2})\!\bigg)\!\bigg].\] \label{d3sup}
\end{ex}
Note that only the basis associated with one of the tableaux with $f=0$ appears in the last expansion in Example \ref{d3sup}, and in the last two rows there, there are easy cancelations (though we kept the canceling terms, in order to keep the form of the bases from Example \ref{smallex}). The last two expansions in Example \ref{d3sup} can be combined to obtain those of the monomials $x_{1}^{2}x_{2}$ and $x_{1}x_{2}^{2}$.

\begin{rmk}
We saw in Examples \ref{supgen} and \ref{d3sup} that when the expanded polynomial depends on a small number of variables, the full symmetric functions appear with appropriate normalizing denominators. These are very similar to the ones used in \cite{[Z1]} (for a different purpose). It may be interesting to check whether such normalizations of symmetric functions can be used, perhaps in expansions using basis elements for the representations from Definition \ref{repswithmon}, to obtain neater coefficients in some cases. \label{normsymfunc}
\end{rmk}

\medskip

Based on the latter four examples, we pose the following conjecture.
\begin{conj}
Let $d$ and $f$ be positive integers.
\begin{enumerate}[$(i)$]
\item Given $M\in\operatorname{SSYT}_{d}(\lambda)$ for some $\lambda \vdash n$ with $f_{M}=f$, consider an element $0 \neq F \in V_{M}$, namely $F\in\mathbb{Q}[\mathbf{x}_{n}]_{d}$ with $\operatorname{supp}_{n}F=\{M\}$. Then the set $\operatorname{supp}_{n+1}F$ consists of tableaux $N$ with $f_{N}\in\{f,f-1\}$.
\item For $F\in\bigoplus_{\lambda \vdash n}\bigoplus_{M\in\operatorname{SSYT}_{d}(\lambda),\ f_{M}=f}V_{M}$, we define the projection $p_{n}^{f}F$ of its image in $\mathbb{Q}[\mathbf{x}_{n+1}]_{d}$ to be $\sum_{N\in\operatorname{supp}_{n+1}F,\ f_{N}=f}F_{N}\in\mathbb{Q}[\mathbf{x}_{n+1}]_{d}^{f}$. The map $p_{n}^{f}:\bigoplus_{\lambda \vdash n}\bigoplus_{M\in\operatorname{SSYT}_{d}(\lambda),\ f_{M}=f}V_{M}\to\bigoplus_{\nu \vdash n}\bigoplus_{M\in\operatorname{SSYT}_{d}(\nu),\ f_{N}=f}V_{N}$ is then injective.
\item Given $F$ as in part $(ii)$, write $q_{n}^{f}F$ for the difference $F-p_{n}^{f}F$, which also equals $\sum_{N\in\operatorname{supp}_{n+1}F,\ f_{N}=f-1}F_{N}\in\mathbb{Q}[\mathbf{x}_{n+1}]_{d}^{f-1}$. Then the resulting map $q_{n}^{f}:\bigoplus_{\lambda \vdash n}\bigoplus_{M\in\operatorname{SSYT}_{d}(\lambda),\ f_{M}=f}V_{M}\to\bigoplus_{\nu \vdash n}\bigoplus_{M\in\operatorname{SSYT}_{d}(\nu),\ f_{N}=f-1}V_{N}$ combines with $p_{n}^{f-1}$ to produce an injective map from the larger space $\bigoplus_{\lambda \vdash n}\bigoplus_{M\in\operatorname{SSYT}_{d}(\lambda),\ f_{M}\in\{f,f-1\}}V_{M}$ to $\bigoplus_{\nu \vdash n}\bigoplus_{M\in\operatorname{SSYT}_{d}(\nu),\ f_{N}=f-1}V_{N}$.
\end{enumerate} \label{polsincn}
\end{conj}
We assume $d>0$ in Conjecture \ref{polsincn} since for $d=0$ there is a single tableau $M$ with $f_{M}=0$ (and $\mathbb{Q}[\mathbf{x}_{n}]_{0}$, $\mathbb{Q}[\mathbf{x}_{\infty}]_{0}$, and $\tilde{\Lambda}_{0}$ is just the space $\mathbb{Q}$ of constants). The last assertion of part $(iii)$ there implies, for $f=1$, means that $q_{n}^{f}F$ must depend on the last variable $x_{n+1}$ wherever $F\neq0$ (a situation that is expected to occur for every $f>0$). I could verify the conjecture also for the tableaux whose non-zero contents are 111, 1111, 211, and 221, thus covering (by appropriately modifying Examples \ref{suprep} and \ref{supd3}) all the cases with $d\leq4$.

\begin{rmk}
When $F$ lies in a single $V_{M}$ (and $f=f_{M}>0$), Remark \ref{fiota} yields $\hat{\iota}M$ as a candidate for an element of $\operatorname{supp}_{n+1}F$ that participates in $p_{n}^{f}F$. However, the latter image is not necessarily contained in $V_{\hat{\iota}M}$, and $\operatorname{supp}_{n+1}F$ can contain more than one tableau $N$ with $f_{N}=f$, as Example \ref{supd3} shows. \label{notiota}
\end{rmk}

\begin{lem}
Assume Conjecture \ref{polsincn}, and take $d>0$ and $f>0$.
\begin{enumerate}[$(i)$]
\item We have $\mathbb{Q}[\mathbf{x}_{n}]_{d}^{f}\subseteq\mathbb{Q}[\mathbf{x}_{n+1}]_{d}^{f}$.
\item Elements of $\mathbb{Q}[\mathbf{x}_{n}]_{d}^{f}\setminus\mathbb{Q}[\mathbf{x}_{n}]_{d}^{f-1}$ lie in $\mathbb{Q}[\mathbf{x}_{n+1}]_{d}^{f}\setminus\mathbb{Q}[\mathbf{x}_{n+1}]_{d}^{f-1}$.
\item If $n>2d$ then $\mathbb{Q}[\mathbf{x}_{n}]_{d}^{f}$ generates $\mathbb{Q}[\mathbf{x}_{n+1}]_{d}^{f}$ over $\mathbb{Q}[S_{n+1}]$.
\end{enumerate} \label{conjeq}
\end{lem}

\begin{proof}
Take $F\in\mathbb{Q}[\mathbf{x}_{n}]_{d}^{f}$, and decompose it as $\sum_{M\in\operatorname{supp}_{n}F}F_{M}$ as in Definition \ref{nsup}. Then $f_{M} \leq f$ for every $M\in\operatorname{supp}_{n}F$ through Remark \ref{contsup}, and part $(i)$ of Conjecture \ref{polsincn} implies that for every $N\in\operatorname{supp}_{n+1}F_{M}$ we have $f_{N} \leq f_{M} \leq f$. As $\operatorname{supp}_{n+1}F$ is clearly contained in $\bigcup_{M\in\operatorname{supp}_{n}F}\operatorname{supp}_{n+1}F_{M}$, part $(i)$ follows from Remark \ref{contsup}.

For parts $(ii)$ and $(iii)$ we work by induction on $f$. Lemma \ref{sameiota} implies that each representation $V_{M}$ with $f_{M}=0$ is contained in $V_{\hat{\iota}M}$ and generates it over $\mathbb{Q}[S_{n+1}]$, and we recall from part $(iii)$ of Theorem \ref{FMTdecom} that for $n>2d$ every tableau $N\in\operatorname{SSYT}_{d}(\nu)$ for $\nu \vdash n+1$ is a $\hat{\iota}$-image. As $\hat{\iota}$ does not affect the parameter $f$ (by Remark \ref{fiota}), both parts hold unconditionally for $f=0$. We thus assume that $f>0$, and work modulo $\mathbb{Q}[\mathbf{x}_{n+1}]_{d}^{f-1}$.

We thus take $F\in\mathbb{Q}[\mathbf{x}_{n}]_{d}^{f}\setminus\mathbb{Q}[\mathbf{x}_{n}]_{d}^{f-1}$, and by part $(i)$ the space by which we divided contains $F_{M}$ for every $M\in\operatorname{supp}_{n}F$ for which $f_{M}<f$. We thus consider the part $\sum_{M\in\operatorname{supp}_{n}F,\ f_{M}=f}F_{M}$, which is non-zero by our assumption on $F$, and it clearly lies in $\bigoplus_{\lambda \vdash n}\bigoplus_{M\in\operatorname{SSYT}_{d}(\lambda),\ f_{M}=f}V_{M}$. But as an element of $\mathbb{Q}[\mathbf{x}_{n+1}]_{d}^{f}$ (or just of $\mathbb{Q}[\mathbf{x}_{n+1}]_{d}$), it equals its $p_{n}^{f}$-image plus an element of $\mathbb{Q}[\mathbf{x}_{n+1}]_{d}^{f-1}$, and the former does not vanish by the injectivity from part $(ii)$ of Conjecture \ref{polsincn}. This establishes part $(ii)$.

Finally, we note that the map $p_{n}^{f}$ from that conjecture respects the action of $S_{n}$. Hence its injectivity implies that the number of representations of this group in the image is the same as the number of representations in the domain. They thus generate, over $\mathbb{Q}[S_{n+1}]$, at least the same number of representations of $S_{n+1}$. But part $(iii)$ of Theorem \ref{FMTdecom} and Remark \ref{fiota} show that the number of representations of $S_{n+1}$ in the range of $p_{n}^{f}$ is the same as the number of representations of $S_{n}$ in its domain, so that the image of $\mathbb{Q}[\mathbf{x}_{n}]_{d}^{f}$ in $\mathbb{Q}[\mathbf{x}_{n+1}]_{d}^{f}\big/\mathbb{Q}[\mathbf{x}_{n+1}]_{d}^{f-1}$ generates the entire space over $\mathbb{Q}[S_{n+1}]$. As the induction hypothesis implies that $\mathbb{Q}[\mathbf{x}_{n}]_{d}^{f-1}$ generates $\mathbb{Q}[\mathbf{x}_{n+1}]_{d}^{f-1}$, we deduce the desired assertion also for $f$, yielding part $(iii)$. This proves the lemma.
\end{proof}
Parts $(i)$ and $(ii)$ of Lemma \ref{conjeq} also implies Conjecture \ref{polsincn} back, by inverting the proof, except for the assertion that if $\operatorname{supp}_{n}F$ is a single tableau $M$ with $f_{M}=f$ then $f_{N}$ cannot be smaller than $f-1$ for any $N\in\operatorname{supp}_{n+1}F$. Note that only the number of representations is preserved by applying $p_{n}^{f}$ in the proof of part $(iii)$ of that lemma (for which, in fact, the inequality $n\geq2d$ suffices as usual), as Remark \ref{notiota} implies that this map mixes the irreducible components in Theorem \ref{FMTdecom}. By applying this map, and then $p_{n+1}^{f}$ and so forth, even to a polynomial whose $n$-support is a singleton, does produce an $(n+k)$-support that involve all values $0\leq\tilde{f} \leq f$ for large enough $k$ (most likely $k=f$), as one sees in Examples \ref{supgen} and \ref{d3sup}.

\begin{rmk}
By intersecting all the spaces with $\mathbb{Q}[\mathbf{x}_{n+1}]_{\mu}$ for some content $\mu$ of sum $d$, we get an injective map $p_{n}^{f}$ from $\bigoplus_{\lambda \vdash n}\bigoplus_{M\in\operatorname{SSYT}_{\mu}(\lambda),\ f_{M}=f}V_{M}$ into $\bigoplus_{\nu \vdash n}\bigoplus_{M\in\operatorname{SSYT}_{\mu}(\nu),\ f_{N}=f}V_{N}$, in part $(ii)$ and Conjecture \ref{polsincn}, as well as the inclusion of $\mathbb{Q}[\mathbf{x}_{n}]_{\mu}^{f}\setminus\mathbb{Q}[\mathbf{x}_{n}]_{\mu}^{f-1}$ into $\mathbb{Q}[\mathbf{x}_{n+1}]_{\mu}^{f}\setminus\mathbb{Q}[\mathbf{x}_{n+1}]_{\mu}^{f-1}$ in Lemma \ref{conjeq}. This is so, since the $n$-support from Definition \ref{nsup} is based on the representations $V_{M}$ (it will no longer hold for an equivalent definition that uses the representations $\tilde{V}_{M}$ from Definition \ref{repswithmon}). \label{conjcont}
\end{rmk}

\medskip

We may consider the spaces $\mathbb{Q}[\mathbf{x}_{n}]_{d}^{f}$ from Definition \ref{repswithmon} as a filtration on the completely reducible representation $\mathbb{Q}[\mathbf{x}_{n}]_{d}$ on $S_{n}$, starting with $\{0\}$ for $f=-1$, and concluding at a finite level determined by the maximal value of $f_{M}$ for $\lambda \vdash n$ and $M\in\operatorname{SSYT}_{d}(\lambda)$. In order to do so, we deduce from Lemma \ref{conjeq} the following consequence.
\begin{cor}
Assume that Conjecture \ref{polsincn} holds.
\begin{enumerate}[$(i)$]
\item Take $F\in\mathbb{Q}[\mathbf{x}_{\infty}]_{d}$, choose $n$ large enough such that $F\in\mathbb{Q}[\mathbf{x}_{n}]_{d}$, and write $f_{F}$ for the parameter $f$ such that $F\in\mathbb{Q}[\mathbf{x}_{n}]_{d}^{f}\setminus\mathbb{Q}[\mathbf{x}_{n}]_{d}^{f-1}$. Then $f_{F}$ is independent of $n$, and is thus an intrinsic parameter of $F\in\mathbb{Q}[\mathbf{x}_{\infty}]_{d}$.
\item We have $f_{F} \leq d$ for every such $F$, but equality can hold.
\item If $F$ lies in $\mathbb{Q}[\mathbf{x}_{\infty}]_{\mu}$ for some content $\mu$ of sum $d$, consisting of $\ell$ non-zero entries, then we have $f_{F}\leq\ell$, but again an equality can hold.
\end{enumerate} \label{Ffdef}
\end{cor}

\begin{proof}
If $f=f_{F}$ is defined using a value of $n$, so that $F\in\mathbb{Q}[\mathbf{x}_{n}]_{d}^{f}\setminus\mathbb{Q}[\mathbf{x}_{n}]_{d}^{f-1}$, then Lemma \ref{conjeq} shows that $F$ also lies in $\mathbb{Q}[\mathbf{x}_{n+1}]_{d}^{f}\setminus\mathbb{Q}[\mathbf{x}_{n+1}]_{d}^{f-1}$, so that the parameter defined using $n+1$ is the same value $f_{F}$, yielding part $(i)$.

Next, we recall from Definition \ref{repswithmon} that $F$ lies in $\mathbb{Q}[\mathbf{x}_{n}]_{d}^{f}$ but not in $\mathbb{Q}[\mathbf{x}_{n}]_{d}^{f-1}$ if it is contained in $\bigoplus_{\lambda \vdash n}\bigoplus_{M\in\operatorname{SSYT}_{d}(\lambda),\ f_{M} \leq f}V_{M}$ but at least one $M$ with $f_{M}=f$ yields a non-zero contribution. Similarly, for $F\in\mathbb{Q}[\mathbf{x}_{n}]_{d}$ we obtain the same but with tableaux $M\in\operatorname{SSYT}_{\mu}(\lambda)$.

By letting $n$ vary, we deduce that the tight upper bound on $f_{F}$ for $F\in\mathbb{Q}[\mathbf{x}_{n}]_{d}$ is the maximal value $f_{M}$ can attain when we let $n$ run over $\mathbb{N}$, $\lambda$ goes over all partitions of every such $n$, and $M$ is taken from $M\in\operatorname{SSYT}_{d}(\lambda)$, and for $F\in\mathbb{Q}[\mathbf{x}_{n}]_{\mu}$ we do the same and restrict to $M\in\operatorname{SSYT}_{\mu}(\lambda)$.

But the number $f_{M}$ for $M\in\operatorname{SSYT}_{\mu}(\lambda)$ is, by the same definition, the number of non-zero entries of $M$ that lie in the first row. It is thus the number of a subset of $\ell$ elements (the non-zero entries of the content $\mu$ of $M$), which is bounded by $\ell$ but can be $\ell$ in case $M$ has a single line (and content $\mu$ plus enough zeros, for $n$ large enough). This establishes part $(iii)$, and for part $(ii)$ we observe that the maximal value of $\ell$ that can be obtained for a content $\mu$ of sum $d$ is $\ell=d$, when $\mu$ consists of $d$ instances of 1 (and $n-d$ instances of 0 if needed). This proves the corollary.
\end{proof}
The tableau from the proof of parts $(iii)$ and $(ii)$ of Corollary \ref{Ffdef} is the only one with $f_{M}=\ell$ or with $f_{M}=d=\Sigma(M)$. We deduce that $\mathbb{Q}[\mathbf{x}_{n}]_{\mu}$ for $n\geq\ell$ is spanned by $\bigoplus_{\lambda \vdash n}\bigoplus_{M\in\operatorname{SSYT}_{\mu}(\lambda),\ f_{M}<\ell}V_{M}$ plus the symmetric function $m_{\mu}$ (in $n$ variables). Similarly, if $n \geq d$ then $\mathbb{Q}[\mathbf{x}_{n}]_{d}$ is the sum of $\bigoplus_{\lambda \vdash n}\bigoplus_{M\in\operatorname{SSYT}_{d}(\lambda),\ f_{M}<d}V_{M}$ and the multiples of the symmetric function $e_{d}^{(n)}$.

Since it is clear that $\mathbb{Q}[\mathbf{x}_{n}]_{d}^{f}$ consists of those elements $F\in\mathbb{Q}[\mathbf{x}_{n}]_{d}$ such that, viewed as elements of $\mathbb{Q}[\mathbf{x}_{\infty}]_{d}$, satisfy $f_{F} \leq f$ via Corollary \ref{Ffdef}, we make the following definition.
\begin{defn}
For any $d\geq0$ and $f\geq0$ we write $\mathbb{Q}[\mathbf{x}_{\infty}]_{d}^{f}$ for the set of polynomials $F\in\mathbb{Q}[\mathbf{x}_{\infty}]_{d}$ whose parameter $f_{F}$ from Corollary \ref{Ffdef} satisfies $f_{F} \leq f$. We also extend this notation to $f=-1$ by setting $\mathbb{Q}[\mathbf{x}_{\infty}]_{d}^{-1}=\{0\}$, and write $\mathbb{Q}[\mathbf{x}_{\infty}]_{\mu}^{f}$ for the intersection of $\mathbb{Q}[\mathbf{x}_{\infty}]_{d}^{f}$, for all $f\geq-1$, with $\mathbb{Q}[\mathbf{x}_{\infty}]_{\mu}$ for every content $\mu$ of sum $d$. \label{Qxinfdf}
\end{defn}
Note that Definition \ref{Qxinfdf} for $f=0$ reproduces the spaces $\mathbb{Q}[\mathbf{x}_{\infty}]_{d}^{0}$ and $\mathbb{Q}[\mathbf{x}_{\infty}]_{\mu}^{0}$ from parts $(i)$ and $(ii)$ of Theorem \ref{compred}. It does so unconditionally, but defining the spaces with $f>0$ depends on Conjecture \ref{polsincn}. We henceforth assume the validity of the latter, even without saying so explicitly. All the results that follow are hence conditional on that conjecture.

\medskip

The first properties of the spaces from Definition \ref{Qxinfdf} are as follows.
\begin{prop}
Take $d\geq0$, and a content $\mu$ consisting of $\ell$ positive integers of sum $d$.
\begin{enumerate}[$(i)$]
\item The spaces $\{\mathbb{Q}[\mathbf{x}_{\infty}]_{d}^{f}\}_{f=-1}^{d}$ form a strictly increasing filtration of $\mathbb{Q}[\mathbf{x}_{\infty}]_{d}$ by representations of $S_{\infty}$ and of $S_{\mathbb{N}}$, that begins with $\{0\}$ and ends at the full representation. The quotient $\mathbb{Q}[\mathbf{x}_{\infty}]_{d}^{f}\big/\mathbb{Q}[\mathbf{x}_{\infty}]_{d}^{f-1}$ is isomorphic to $\bigoplus_{\hat{\lambda}\vdash\infty}\bigoplus_{\hat{M}\in\operatorname{SSYT}_{d}(\hat{\lambda}),\ f_{\hat{M}}=f}V_{\hat{M}^{0}}$.
\item The same assertion holds for $\{\mathbb{Q}[\mathbf{x}_{\infty}]_{\mu}^{f}\}_{f=-1}^{\ell}$ inside $\mathbb{Q}[\mathbf{x}_{\infty}]_{\mu}$, where the $f$th quotient is isomorphic to $\bigoplus_{\hat{\lambda}\vdash\infty}\bigoplus_{\hat{M}\in\operatorname{SSYT}_{\mu}(\hat{\lambda}),\ f_{\hat{M}}=f}V_{\hat{M}^{0}}$.
\end{enumerate} \label{filtprop}
\end{prop}
By Remark \ref{mudetf}, the condition $f_{\hat{M}}=f$ in part $(ii)$ of Proposition \ref{filtprop} takes all of the tableaux from $\operatorname{SSYT}_{\mu}(\hat{\lambda})$ for specific infinite Ferrers diagrams $\hat{\lambda}\vdash\infty$, and none of the tableaux for other choices of $\hat{\lambda}\vdash\infty$. We will exemplify all these notions and results, and those that follow, after Corollary \ref{submulti} below.

\begin{proof}
Definition \ref{Qxinfdf} gives $\mathbb{Q}[\mathbf{x}_{\infty}]_{d}^{-1}=\{0\}$, and the bounds from parts $(ii)$ and $(iii)$ of Corollary \ref{Ffdef} show that $\mathbb{Q}[\mathbf{x}_{\infty}]_{d}^{d}=\mathbb{Q}[\mathbf{x}_{\infty}]_{d}$ and $\mathbb{Q}[\mathbf{x}_{\infty}]_{\mu}^{\ell}=\mathbb{Q}[\mathbf{x}_{\infty}]_{\mu}$. Moreover, for such $\mu$ we can take $n$ large and $\lambda$ with two rows in which the second one is of length $\ell-f$, fill it with some positive entries from $\mu$ in a non-decreasing order, and fill the first row with enough zeros plus the remaining $f$ entries from $\mu$. This produces $M\in\operatorname{SSYT}(\lambda)$ with $f_{M}=f$, and any $0 \neq F \in V_{M}$ will be, via Corollary \ref{Ffdef}, in $\mathbb{Q}[\mathbf{x}_{\infty}]_{\mu}^{f}\setminus\mathbb{Q}[\mathbf{x}_{\infty}]_{\mu}^{f-1}$. Hence the sequence in part $(ii)$ is strictly increasing, and the argument proving part $(ii)$ of that corollary shows the same for part $(i)$.

We now recall from Proposition \ref{samereps} and Remark \ref{polstimessym} that $\mathbb{Q}[\mathbf{x}_{\infty}]_{d}$ is preserved under the action of $S_{\mathbb{N}}$, and that the orbit of every polynomial under this group is the same as its orbit under $S_{\infty}$. In addition, Corollary \ref{Ffdef} implies that for an element of $\mathbb{Q}[\mathbf{x}_{n}]_{d}$, being in $\mathbb{Q}[\mathbf{x}_{n}]_{d}^{f}$ is equivalent to being in $\mathbb{Q}[\mathbf{x}_{\infty}]_{d}^{f}$.

We thus take a polynomial $F\in\mathbb{Q}[\mathbf{x}_{\infty}]_{d}^{f}$, and assume that it is related to some polynomial $G\in\mathbb{Q}[\mathbf{x}_{\infty}]_{d}$ by the action of $S_{\mathbb{N}}$. As there is an element of $S_{\infty}$ taking $F$ to $G$, we can take $n$ large enough so that $F$ is in $\mathbb{Q}[\mathbf{x}_{\infty}]_{d}$ (hence in $\mathbb{Q}[\mathbf{x}_{\infty}]_{d}^{f}$), $G\in\mathbb{Q}[\mathbf{x}_{\infty}]_{d}$, and there is an element of $S_{n}$ taking $F$ to $G$.

But as $\mathbb{Q}[\mathbf{x}_{n}]_{d}^{f}$ is a representation of $S_{n}$ (which is clear from Definition \ref{repswithmon}), the polynomial $G$ must be in $\mathbb{Q}[\mathbf{x}_{n}]_{d}^{f}$ hence also in $\mathbb{Q}[\mathbf{x}_{\infty}]_{d}^{f}$. This shows that $\mathbb{Q}[\mathbf{x}_{\infty}]_{d}^{f}$ is a representation of $S_{\mathbb{N}}$ (and of $S_{\infty}$), and as so is $\mathbb{Q}[\mathbf{x}_{\infty}]_{\mu}$, we obtain the same property for $\mathbb{Q}[\mathbf{x}_{\infty}]_{\mu}^{f}$, as their intersection.

Next, take some $n>2d$, and then the quotient $\mathbb{Q}[\mathbf{x}_{n}]_{d}^{f}\big/\mathbb{Q}[\mathbf{x}_{n}]_{d}^{f-1}$ is, via Theorem \ref{FMTdecom}, the isomorphic image of $\bigoplus_{\lambda \vdash n}\bigoplus_{M\in\operatorname{SSYT}_{d}(\lambda),\ f_{M}=f}V_{M}$, which is isomorphic to $\bigoplus_{\lambda \vdash n}\bigoplus_{M\in\operatorname{SSYT}_{d}(\lambda),\ f_{M}=f}V_{M^{0}}$. Part $(iii)$ of Lemma \ref{conjeq} implies that the former direct sum generates $\mathbb{Q}[\mathbf{x}_{n+1}]_{d}^{f}\big/\mathbb{Q}[\mathbf{x}_{n+1}]_{d}^{f-1}$ over $\mathbb{Q}[S_{n+1}]$. The isomorphism type of the resulting representation of $S_{n+1}$, on a subspace of $\mathbb{Q}[\mathbf{x}_{n+1}]_{d}^{f}$ generating this quotient, is the same as the one generated over $\mathbb{Q}[S_{n+1}]$ by $\bigoplus_{\lambda \vdash n}\bigoplus_{M\in\operatorname{SSYT}_{d}(\lambda),\ f_{M}=f}V_{M^{0}}$.

But Lemma \ref{sameiota} implies that the representation thus obtained is isomorphic to $\bigoplus_{\lambda \vdash n}\bigoplus_{M\in\operatorname{SSYT}_{d}(\lambda),\ f_{M}=f}V_{\hat{\iota}(M^{0})}$. If we write this index as $(\hat{\iota}M)^{0}$ via Remark \ref{fiota}, then this direct sum becomes $\bigoplus_{\nu \vdash n+1}\bigoplus_{N\in\operatorname{SSYT}_{d}(\nu),\ f_{N}=f}V_{N^{0}}$ by this remark and part $(iii)$ of Theorem \ref{FMTdecom}. Doing so repeatedly and taking $n$ to $\infty$ yields the space generated over $\mathbb{Q}[S_{\infty}]$ by the original direct sum, which means that the space thus obtained is contained in $\mathbb{Q}[\mathbf{x}_{\infty}]_{d}^{f}$ (via Corollary \ref{Ffdef} and Definition \ref{Qxinfdf}) and maps bijectively onto $\mathbb{Q}[\mathbf{x}_{\infty}]_{d}^{f}\big/\mathbb{Q}[\mathbf{x}_{\infty}]_{d}^{f-1}$.

As the representation type (over both $S_{\mathbb{N}}$ and $S_{\infty}$, as usual) of the resulting space is the asserted one, part $(i)$ follows. Intersecting all these quotients $\mathbb{Q}[\mathbf{x}_{\infty}]_{\mu}$ yields the desired isomorphism type of the quotient in part $(ii)$ as well. This completes the proof of the proposition.
\end{proof}
Note that when passing from $n$ to $n+1$ in the proof of Proposition \ref{filtprop}, we did not use the map $p_{n}^{f}$, but rather the direct embedding of $\mathbb{Q}[\mathbf{x}_{n}]_{d}^{f}$ into $\mathbb{Q}[\mathbf{x}_{n+1}]_{d}^{f}$. Hence the subspace yielding this quotient is not canonical (and Theorem \ref{compred} shows that it indeed cannot be such), but rather depends on the choice of $n$ there. In fact, we could have also used Proposition \ref{decomVM0mM} and the presentation of the quotient as the isomorphic image of $\bigoplus_{\lambda \vdash n}\bigoplus_{M\in\operatorname{SSYT}_{d}(\lambda),\ f_{M}=f}\tilde{V}_{M}$ from Definition \ref{repswithmon} (where $V_{M^{0}}$ also equals $\tilde{V}_{M^{0}}$), but this presentations will no longer respect the decompositions into content for part $(ii)$ of Proposition \ref{filtprop}.

We get the following immediate consequence.
\begin{cor}
The representation $\mathbb{Q}[\mathbf{x}_{\infty}]_{d}$, as well as its sub-representation $\mathbb{Q}[\mathbf{x}_{\infty}]_{\mu}$ for every content $\mu$ of sum $d$, admits a finite composition series. Their semi-simplification of these representations are isomorphic to the direct sums $\bigoplus_{\hat{\lambda}\vdash\infty}\bigoplus_{\hat{M}\in\operatorname{SSYT}_{d}(\hat{\lambda})}V_{\hat{M}}^{0}$ and $\bigoplus_{\hat{\lambda}\vdash\infty}\bigoplus_{\hat{M}\in\operatorname{SSYT}_{\mu}(\hat{\lambda})}V_{\hat{M}}^{0}$ respectively. \label{semisimp}
\end{cor}

\begin{proof}
The existence of a finite composition series for a representation is equivalent to a filtration on that representation in which the quotients are completely reducible, and the semi-simplification is isomorphic to the direct sum of the completely reducible sub-quotients of the filtration. The result thus follows directly from Proposition \ref{filtprop}. This proves the corollary.
\end{proof}
While $V_{\hat{M}}^{0}$ is isomorphic to the original representation $V_{\hat{M}}$ for every $M$ in $\operatorname{SSYT}(\hat{\lambda})$, and the latter keeps track on the content of $\hat{M}$ and its degree is $\Sigma(\hat{M})$, we used $V_{\hat{M}}^{0}$ in Corollary \ref{semisimp} as these sub-representations of $\tilde{\Lambda}$ are the ones that are contained in $\mathbb{Q}[\mathbf{x}_{\infty}]$.

\medskip

We now obtain the natural property of the filtrations from Proposition \ref{filtprop}.
\begin{thm}
The quotients $\mathbb{Q}[\mathbf{x}_{\infty}]_{d}^{f}\big/\mathbb{Q}[\mathbf{x}_{\infty}]_{d}^{f-1}$ and $\mathbb{Q}[\mathbf{x}_{\infty}]_{\mu}^{f}\big/\mathbb{Q}[\mathbf{x}_{\infty}]_{\mu}^{f-1}$ from Proposition \ref{filtprop} are the maximal completely reducible representations inside $\mathbb{Q}[\mathbf{x}_{\infty}]_{d}\big/\mathbb{Q}[\mathbf{x}_{\infty}]_{d}^{f-1}$ and $\mathbb{Q}[\mathbf{x}_{\infty}]_{\mu}\big/\mathbb{Q}[\mathbf{x}_{\infty}]_{\mu}^{f-1}$ respectively. \label{filtQxinf}
\end{thm}

\begin{proof}
As the case $f=0$ holds unconditionally by parts $(i)$ and $(ii)$ of Theorem \ref{compred}, we assume that $f>0$. The expressions for the quotients in Proposition \ref{filtprop} implies that they are completely reducible. Noting that the intersection of $\mathbb{Q}[\mathbf{x}_{\infty}]_{d}^{f}\big/\mathbb{Q}[\mathbf{x}_{\infty}]_{d}^{f-1}$ with $\mathbb{Q}[\mathbf{x}_{\infty}]_{\mu}\big/\mathbb{Q}[\mathbf{x}_{\infty}]_{\mu}^{f-1}$ is $\mathbb{Q}[\mathbf{x}_{\infty}]_{\mu}^{f}\big/\mathbb{Q}[\mathbf{x}_{\infty}]_{\mu}^{f-1}$, and the former is the direct sum of the latter over contents $\mu$ of sum $d$, we deduce that any sub-representation of $\mathbb{Q}[\mathbf{x}_{\infty}]_{\mu}^{f}\big/\mathbb{Q}[\mathbf{x}_{\infty}]_{\mu}^{f-1}$ is contained in $\mathbb{Q}[\mathbf{x}_{\infty}]_{d}^{f}\big/\mathbb{Q}[\mathbf{x}_{\infty}]_{d}^{f-1}$, and the intersection of an irreducible sub-representation $U$ of $\mathbb{Q}[\mathbf{x}_{\infty}]_{d}^{f}\big/\mathbb{Q}[\mathbf{x}_{\infty}]_{d}^{f-1}$ with $\mathbb{Q}[\mathbf{x}_{\infty}]_{\mu}^{f}\big/\mathbb{Q}[\mathbf{x}_{\infty}]_{\mu}^{f-1}$ is either $\{0\}$ or irreducible there. Hence, as in the proof of Theorem \ref{compred}, the two assertions imply one another.

We thus let $U$, as in that theorem, be an irreducible representation inside $\mathbb{Q}[\mathbf{x}_{\infty}]_{\mu}\big/\mathbb{Q}[\mathbf{x}_{\infty}]_{\mu}^{f-1}$ for some content $\mu$ of sum $d$, and we need to show that it is contained in the quotient in question. So take a non-zero element of $U$, which is thus the image of some polynomial $F\in\mathbb{Q}[\mathbf{x}_{\infty}]_{d}$ modulo $\mathbb{Q}[\mathbf{x}_{\infty}]_{\mu}^{f-1}$. The irreducibility of $U$ implies that the image of the space $\mathbb{Q}[S_{\infty}]F$ (or equivalently $\mathbb{Q}[S_{\mathbb{N}}]F$) in $\mathbb{Q}[\mathbf{x}_{\infty}]_{\mu}\big/\mathbb{Q}[\mathbf{x}_{\infty}]_{\mu}^{f-1}$ is $U$.

Write $\hat{f}=f_{F}$, which satisfies $\hat{f} \geq f$ since $F$ does not vanish modulo $\mathbb{Q}[\mathbf{x}_{\infty}]_{\mu}^{f-1}$, and we can take $n$ such that $F$ lies in $\mathbb{Q}[\mathbf{x}_{n}]_{\mu}^{\hat{f}}\setminus\mathbb{Q}[\mathbf{x}_{n}]_{\mu}^{\hat{f}-1}$, via Corollary \ref{Ffdef}. We decompose $F$ via Definition \ref{nsup} as $\sum_{M\in\operatorname{supp}_{n}F}F_{M}$, and we write it as $\sum_{\tilde{f}=0}^{\hat{f}}F_{\tilde{f}}$ where $F_{\tilde{f}}:=\sum_{M\in\operatorname{supp}_{n}F,\ f_{M}=\tilde{f}}F_{M}$.

We then view $F$ as an element of $\mathbb{Q}[\mathbf{x}_{n+1}]_{\mu}$, and examine its decomposition. Using the maps from Conjecture \ref{polsincn}, we get three expressions. One is $p_{n}^{\hat{f}}F_{\hat{f}}\in\bigoplus_{\nu \vdash n}\bigoplus_{M\in\operatorname{SSYT}_{d}(\nu),\ f_{N}=\hat{f}}V_{N}$. The second one is the sum of $q_{n}^{\hat{f}}F_{\hat{f}}$ and $p_{n}^{\hat{f}-1}F_{\hat{f}-1}$, both lying in $\bigoplus_{\nu \vdash n}\bigoplus_{M\in\operatorname{SSYT}_{d}(\nu),\ f_{N}=\hat{f}}V_{N}$. The third expression arises similarly from the remaining parts.

But from part $(i)$ of Conjecture \ref{polsincn} we deduce that the third expression lies in $\mathbb{Q}[\mathbf{x}_{n+1}]_{\mu}^{\hat{f}-2}$ hence in $\mathbb{Q}[\mathbf{x}_{\infty}]_{\mu}^{\hat{f}-2}$. Moreover, part $(ii)$ of that conjecture implies that the first expression is non-zero, and the second one does not vanish as well via part $(iii)$ there. Moreover, the space generated by $F$ over $\mathbb{Q}[S_{\infty}]$ or $\mathbb{Q}[S_{\mathbb{N}}]$ contained in the direct sum of the one generated by $p_{n}^{\hat{f}}F_{\hat{f}}$, the one generated by $q_{n}^{\hat{f}}F_{\hat{f}}+p_{n}^{\hat{f}-1}F_{\hat{f}-1}$, and the subspace arising from the third expression, which is contained in $\mathbb{Q}[\mathbf{x}_{\infty}]_{\mu}^{\hat{f}-2}$. There are also injective maps from $\mathbb{Q}[S_{\infty}]F=\mathbb{Q}[S_{\mathbb{N}}]F$ into $\mathbb{Q}[S_{\infty}]p_{n}^{\hat{f}}F_{\hat{f}}=\mathbb{Q}[S_{\mathbb{N}}]p_{n}^{\hat{f}}F_{\hat{f}}$ and the space associated with $q_{n}^{\hat{f}}F_{\hat{f}}+p_{n}^{\hat{f}-1}F_{\hat{f}-1}$.

We now observe that the space generated by $p_{n}^{\hat{f}}F_{\hat{f}}$ has a non-zero image in $\mathbb{Q}[\mathbf{x}_{\infty}]_{\mu}^{\hat{f}}\big/\mathbb{Q}[\mathbf{x}_{\infty}]_{\mu}^{\hat{f}-1}$, and the one arising from $q_{n}^{\hat{f}}F_{\hat{f}}+p_{n}^{\hat{f}-1}F_{\hat{f}-1}$ has a non-zero image in $\mathbb{Q}[\mathbf{x}_{\infty}]_{\mu}^{\hat{f}-1}\big/\mathbb{Q}[\mathbf{x}_{\infty}]_{\mu}^{\hat{f}-2}$. Combining this with the injective maps from $\mathbb{Q}[S_{\infty}]F=\mathbb{Q}[S_{\mathbb{N}}]F$, and using the fact that these quotients are completely reducible (by Proposition \ref{filtprop}), we can find one irreducible component in each of these quotients onto which the space generated by $F$ projects non-trivially.

But Proposition \ref{filtprop} shows that the first quotient is the direct sum of representations $V_{\hat{M}^{0}}$ for infinite Ferrers diagrams $\hat{\lambda}\vdash\infty$ and elements $\hat{M}\in\operatorname{SSYT}_{\mu}(\hat{\lambda})$ with $f_{\hat{M}}=\hat{f}$, while the second one is a similar direct sum, but now with the condition $f_{\hat{M}}=\hat{f}-1$. From the previous paragraph we obtain tableaux $\hat{M}_{\hat{f}}$ and $\hat{M}_{\hat{f}-1}$, whose $f$-values are $\hat{f}$ and $\hat{f}-1$ respectively, such that the space generated by $F$ projects non-trivially onto both irreducible sub-representations $V_{\hat{M}_{\hat{f}}^{0}}\subseteq\mathbb{Q}[\mathbf{x}_{\infty}]_{\mu}^{\hat{f}}\big/\mathbb{Q}[\mathbf{x}_{\infty}]_{\mu}^{\hat{f}-1}$ and $V_{\hat{M}_{\hat{f}-1}^{0}}\subseteq\mathbb{Q}[\mathbf{x}_{\infty}]_{\mu}^{\hat{f}-1}\big/\mathbb{Q}[\mathbf{x}_{\infty}]_{\mu}^{\hat{f}-2}$ (using part $(iv)$ of Theorem \ref{repsinf}).

We now recall from Remark \ref{mudetf} that $\mu$ and the shape $\operatorname{sh}(\hat{M})$ for any infinite semi-standard young tableau $\hat{M}$ determines the value of $f_{\hat{M}}$. Since this value is different for $\hat{M}_{\hat{f}}$ and $\hat{M}_{\hat{f}-1}$, their shapes must be different. As part $(iii)$ of Theorem \ref{repsinf} shows that the isomorphism type of $V_{\hat{M}}$, hence also of its isomorph $V_{\hat{M}^{0}}$, is $\mathcal{S}^{\hat{\lambda}}$ where $\hat{\lambda}=\operatorname{sh}(\hat{M})$, we deduce, via part $(v)$ of Theorem \ref{repsinf}, that the irreducible representations $V_{\hat{M}_{\hat{f}}^{0}}$ and $V_{\hat{M}_{\hat{f}-1}^{0}}$ are not isomorphic.

We saw that $\hat{f} \geq f$, and we now assume that this inequality is strict. Then $\mathbb{Q}[\mathbf{x}_{\infty}]_{\mu}^{f-2}$ is contained in $\mathbb{Q}[\mathbf{x}_{\infty}]_{\mu}^{\hat{f}-2}$, and we can replace, in our projections, the space generated by $F$ by its image in $\mathbb{Q}[\mathbf{x}_{\infty}]_{\mu}\big/\mathbb{Q}[\mathbf{x}_{\infty}]_{\mu}^{f-1}$. But our assumption was that this image is the irreducible representation $U$, which thus cannot project non-trivially onto two non-isomorphic irreducible representations. Hence the case $\hat{f}>f$ cannot occur.

But this proves that $\hat{f}=f$, meaning that $F$ lies in $\mathbb{Q}[\mathbf{x}_{\infty}]_{\mu}^{f}$, and with it we get $U\subseteq\mathbb{Q}[\mathbf{x}_{\infty}]_{\mu}^{f}\big/\mathbb{Q}[\mathbf{x}_{\infty}]_{\mu}^{f-1}$ (note that this case does not lead to the same contradiction, because now the representation $V_{\hat{M}_{\hat{f}-1}^{0}}$ vanishes). Therefore our arbitrary irreducible sub-representation $U$ of $\mathbb{Q}[\mathbf{x}_{\infty}]_{\mu}\big/\mathbb{Q}[\mathbf{x}_{\infty}]_{\mu}^{f-1}$ lies in $\mathbb{Q}[\mathbf{x}_{\infty}]_{\mu}^{f}\big/\mathbb{Q}[\mathbf{x}_{\infty}]_{\mu}^{f-1}$, which is the desired maximality as a completely reducible sub-representation of the former quotient. This proves the theorem.
\end{proof}
The proof of Theorem \ref{filtQxinf} exemplifies how for every polynomial $F\in\mathbb{Q}[\mathbf{x}_{\infty}]_{d}$ for which $f_{F}>0$, the representation that it generates cannot be completely reducible, as Theorem \ref{compred} states. The proof in this case is harder, and requires a conjecture, since we no longer work with explicit polynomials, but rather inside quotients, where there are many choices for representatives.

\medskip

We now wish to obtain a similar filtration for $\tilde{\Lambda}_{d}$ and $\tilde{\Lambda}_{\mu}$. For this we adopt the notation $|\alpha|$ for the degree of the monomial element $m_{\alpha}\in\Lambda$ (which is the sum of the entries of $\alpha$). We recall from the proof of Theorem \ref{compred} the decomposition of $\tilde{\Lambda}$ from Remark \ref{polstimessym} with respect to the monomial basis of $\Lambda$, and combining it with the grading yields $\tilde{\Lambda}_{d}=\bigoplus_{e=0}^{d}\bigoplus_{|\alpha|=d-e}\mathbb{Q}[\mathbf{x}_{\infty}]_{e}m_{\alpha}$. The same proof also constructed, for every $\alpha$, the $\mathbb{Q}[\mathbf{x}_{\infty}]$-module homomorphisms $\pi_{\alpha}:\tilde{\Lambda}\to\mathbb{Q}[\mathbf{x}_{\infty}]$, and showed that $F=\sum_{\alpha}\pi_{\alpha}(F)m_{\alpha}$ for every $F\in\Lambda$.

We now make the following definition.
\begin{defn}
Fix $d\geq0$, $f\geq-1$, $\hat{\lambda}\vdash\infty$, an element $\hat{M}\in\operatorname{SSYT}_{d}(\hat{\lambda})$, a content $\mu$ of sum $d$, and some $k\geq1$.
\begin{enumerate}[$(i)$]
\item We set $\tilde{\Lambda}_{d}^{f}$ to be the space of those elements $F=\sum_{\alpha}\pi_{\alpha}(F)m_{\alpha}$ of $\tilde{\Lambda}_{d}$ for which the multiplier $\pi_{\alpha}(F)\in\mathbb{Q}[\mathbf{x}_{\infty}]_{e}$ is in $\mathbb{Q}[\mathbf{x}_{\infty}]_{e}^{f}$ for every such $e$ and $\alpha$.
\item We write $\tilde{\Lambda}_{\mu}^{f}$ for the intersection of $\tilde{\Lambda}_{d}^{f}$ with $\tilde{\Lambda}_{\mu}$.
\item We denote by $c_{\hat{M},f}$ the number of sub-multi-sets of the index of $m_{\hat{M}}$ that have size $f$.
\item We also set $c_{\hat{M}}:=\sum_{f=0}^{d}c_{\hat{M},f}$.
\item We write $R_{\infty,k,d}^{f}$ for the image of $\tilde{\Lambda}_{d}^{f}$ under the projection from $\tilde{\Lambda}_{d}$ onto $R_{\infty,k}$ from Definition \ref{defRinfk}.
\end{enumerate} \label{Lambdafilt}
\end{defn}
Also in Definition \ref{Lambdafilt}, the case $f=0$ reproduces the unconditionally defined $\tilde{\Lambda}_{d}^{0}$, $\tilde{\Lambda}_{\mu}^{0}$, and $R_{\infty,k,d}^{0}$ from parts $(iii)$, $(iv)$, and $(v)$ of Theorem \ref{compred} (with the relation between the first and the third holding, by said part $(v)$, in this case), but the spaces with $f>0$ require Conjecture \ref{polsincn} to be well-defined.

\begin{rmk}
The number $c_{\hat{M},f}$ from Definition \ref{Lambdafilt} is, equivalently, the number of ways to attach to each $h>0$ a multiplicity between 0 and the one that it has in $\hat{M}$ via Definition \ref{tabinf} such that the sum of these multiplicities is $f$. Note that it is positive if and only if $f$ lies between 0 and the size $f_{\hat{M}}$ of the index of $m_{\hat{M}}$ as a multi-set, and if the content $\mu$ of $\hat{M}$ consists of $\ell$ positive entries of sum $d$ then we have $f_{\hat{M}}\leq\ell \leq d$. Therefore the sum $c_{\hat{M}}$ counts all the possible sub-multi-sets of the index of $m_{\hat{M}}$, or equivalently the number of ways to attach a non-negative multiplicity for every $h>0$ in a manner that is bounded its multiplicity in $\hat{M}$, without the sum condition. \label{chatMf}
\end{rmk}

\medskip

Here is the analogue of Proposition \ref{filtprop} and Theorem \ref{filtQxinf} for $\tilde{\Lambda}_{d}$, $\tilde{\Lambda}_{\mu}$, and $R_{\infty,k,d}$, where for the latter we recall the set $\operatorname{SSYT}_{d}^{k}(\hat{\lambda})$ from Definition \ref{Adinf}.
\begin{thm}
Let $d\geq0$ be an integer, and consider a content $\mu$ involving $\ell$ positive integers whose sum is $d$, and an integer $k\geq2$.
\begin{enumerate}[$(i)$]
\item The space $\tilde{\Lambda}_{d}^{f}$ is a sub-representation of $\tilde{\Lambda}_{d}$ over $S_{\infty}$ and $S_{\mathbb{N}}$ for every $f\geq-1$, with $\tilde{\Lambda}_{d}^{-1}=\{0\}$, $\tilde{\Lambda}_{d}^{d}=\tilde{\Lambda}_{d}$, and $\tilde{\Lambda}_{f-1}\subsetneq\tilde{\Lambda}_{f}$ for every $0 \leq f \leq d$. The quotient $\tilde{\Lambda}_{d}^{f}\big/\tilde{\Lambda}_{d}^{f-1}$ is the maximal completely reducible representation inside $\tilde{\Lambda}_{d}\big/\tilde{\Lambda}_{d}^{f-1}$, and it is isomorphic to $\bigoplus_{\hat{\lambda}\vdash\infty}\bigoplus_{\hat{M}\in\operatorname{SSYT}_{d}(\hat{\lambda})}\tilde{V}_{\hat{M}}^{c_{\hat{M},f}}$.
\item The representations $\{\tilde{\Lambda}_{\mu}^{f}\}_{f=-1}^{\ell}$ have similar properties, in which the quotient $\tilde{\Lambda}_{\mu}^{f}\big/\tilde{\Lambda}_{\mu}^{f-1}$ is isomorphic to $\bigoplus_{\hat{\lambda}\vdash\infty}\bigoplus_{\hat{M}\in\operatorname{SSYT}_{\mu}(\hat{\lambda})}V_{\hat{M}}^{c_{\hat{M},f}}$
\item Inside the homogenous part $R_{\infty,k,d}$ of the quotient $R_{\infty,k}$, the representations $\{R_{\infty,k,d}^{f}\}_{f=-1}^{d}$ form a filtration with the same properties, with the quotient $R_{\infty,k,d}^{f}/R_{\infty,k,d}^{f-1}$ being isomorphic to $\bigoplus_{\hat{\lambda}\vdash\infty}\bigoplus_{\hat{M}\in\operatorname{SSYT}_{d}^{k}(\hat{\lambda})}V_{\hat{M}}^{c_{\hat{M},f}}$.
\end{enumerate} \label{filtrations}
\end{thm}
We took $k\geq2$ in Theorem \ref{filtrations} is because for $R_{\infty,1}=\mathbb{Q}$ the filtration is trivial, and $\operatorname{SSYT}_{d}^{1}(\hat{\lambda})$ is empty for any $d>0$. The unique tableau $\hat{M}$ with $f_{\hat{M}}=d=\Sigma(\hat{M})$ is in $\operatorname{SSYT}_{d}^{k}(\hat{\lambda})$ for the appropriate $\hat{\lambda}$ wherever $k\geq2$, keeping the bound $d$ the correct one in part $(iii)$ there.

\begin{proof}
Since we have $\mathbb{Q}[\mathbf{x}_{\infty}]_{e}^{-1}=\{0\}$ and $\mathbb{Q}[\mathbf{x}_{\infty}]_{e}^{e}=\mathbb{Q}[\mathbf{x}_{\infty}]_{e}$ for all $0 \leq e \leq d$, and hence $\mathbb{Q}[\mathbf{x}_{\infty}]_{e}^{d}=\mathbb{Q}[\mathbf{x}_{\infty}]_{e}$ because $e \leq d$, we get that $\tilde{\Lambda}_{d}^{-1}$ and $\tilde{\Lambda}_{d}^{d}$ are the asserted spaces via Definition \ref{Lambdafilt}. The containment $\tilde{\Lambda}_{f-1}\subseteq\tilde{\Lambda}_{f}$ is clear, and since intersecting both with $\mathbb{Q}[\mathbf{x}_{\infty}]_{d}$ yields $\mathbb{Q}[\mathbf{x}_{\infty}]_{d}^{f-1}$ and $\mathbb{Q}[\mathbf{x}_{\infty}]_{d}^{f}$ respectively, the strict containment between the latter implies the desired one.

Now, Definition \ref{Lambdafilt} expresses $\tilde{\Lambda}_{d}^{f}$ as $\bigoplus_{e=0}^{d}\bigoplus_{|\alpha|=d-e}(\tilde{\Lambda}_{d}^{f}\cap\mathbb{Q}[\mathbf{x}_{\infty}]_{e}m_{\alpha})$ for every $f$, with $(\tilde{\Lambda}_{d}^{f}\cap\mathbb{Q}[\mathbf{x}_{\infty}]_{e}m_{\alpha})=\mathbb{Q}[\mathbf{x}_{\infty}]_{e}^{f}m_{\alpha}$. Hence the quotients $\tilde{\Lambda}_{d}\big/\tilde{\Lambda}_{d}^{f-1}$ and $\tilde{\Lambda}_{d}^{f}\big/\tilde{\Lambda}_{d}^{f-1}$ can be written as $\bigoplus_{e=0}^{d}\bigoplus_{|\alpha|=d-e}(\mathbb{Q}[\mathbf{x}_{\infty}]_{e}\big/\mathbb{Q}[\mathbf{x}_{\infty}]_{e}^{f-1})m_{\alpha}$ and $\bigoplus_{e=0}^{d}\bigoplus_{|\alpha|=d-e}(\mathbb{Q}[\mathbf{x}_{\infty}]_{e}^{f}\big/\mathbb{Q}[\mathbf{x}_{\infty}]_{e}^{f-1})m_{\alpha}$ respectively. In particular, since each direct summand is completely reducible by Proposition \ref{filtprop}, so is $\tilde{\Lambda}_{d}^{f}\big/\tilde{\Lambda}_{d}^{f-1}$.

For the maximality, we note that the expressions from the previous paragraph produce a well-defined map $\pi_{d,\alpha}^{f}:\tilde{\Lambda}_{d}\big/\tilde{\Lambda}_{d}^{f-1}\to\mathbb{Q}[\mathbf{x}_{\infty}]_{e}\big/\mathbb{Q}[\mathbf{x}_{\infty}]_{e}^{f-1}$ for every $\alpha$ with $|\alpha|=d-e$. The equality $F=\sum_{e=0}^{d}\sum_{|\alpha|=d-e}\pi_{d,\alpha}^{f}(F)m_{\alpha}$ continues to hold on $\tilde{\Lambda}_{d}\big/\tilde{\Lambda}_{d}^{f-1}$. Hence we have $U\subseteq\bigoplus_{e=0}^{d}\bigoplus_{|\alpha|=d-e}\pi_{d,\alpha}^{f}(U)m_{\alpha}$ for every subspace $U$ of $\tilde{\Lambda}_{d}\big/\tilde{\Lambda}_{d}^{f-1}$, which is a containment of representations of $S_{\infty}$ and of $S_{\mathbb{N}}$ if $U$ is such a representation.

In particular, if $U$ is irreducible then $\pi_{d,\alpha}^{f}(U)$ is either $\{0\}$ or irreducible as well, meaning that $\pi_{d,\alpha}^{f}(U)\subseteq\mathbb{Q}[\mathbf{x}_{\infty}]_{e}\big/\mathbb{Q}[\mathbf{x}_{\infty}]_{e}^{f-1}$ is actually contained in $\mathbb{Q}[\mathbf{x}_{\infty}]_{e}^{f}\big/\mathbb{Q}[\mathbf{x}_{\infty}]_{e}^{f-1}$ by Theorem \ref{filtQxinf}. After multiplying by $m_{\alpha}$ we obtain a sub-representation of $\tilde{\Lambda}_{d}^{f}\big/\tilde{\Lambda}_{d}^{f-1}$, and as the direct sum over $e$ and $\alpha$ is also contained in that representation, the maximality follows.

It remains, for part $(i)$, to determine the isomorphism type of the completely reducible quotient $\tilde{\Lambda}_{d}^{f}\big/\tilde{\Lambda}_{d}^{f-1}=\bigoplus_{e=0}^{d}\bigoplus_{|\alpha|=d-e}(\mathbb{Q}[\mathbf{x}_{\infty}]_{e}^{f}\big/\mathbb{Q}[\mathbf{x}_{\infty}]_{e}^{f-1})m_{\alpha}$. Proposition \ref{filtprop} expresses the multiplier of $m_{\alpha}$ in that direct sum as isomorphic to $\bigoplus_{\hat{\lambda}\vdash\infty}\bigoplus_{\hat{N}\in\operatorname{SSYT}_{e}(\hat{\lambda}),\ f_{\hat{N}}=f}V_{\hat{N}^{0}}$, and the coefficient $m_{\alpha}$ itself distinguishes this part from those associated with other monomial functions.

Adding the multiplicities from $\alpha$ to $\hat{N}$ yields an element $\hat{M}\in\operatorname{SSYT}_{d}(\hat{\lambda})$, with $\hat{N}^{0}=\hat{M}^{0}$, and since $V_{\hat{M}^{0}}=V_{\hat{M}^{0}}$, $V_{\hat{M}^{0}}m_{\alpha}$, and $\tilde{V}_{\hat{M}}$ are all isomorphic, this is how we express the total quotient. We thus need to determine, for each $\hat{M}\in\operatorname{SSYT}_{d}(\hat{\lambda})$, how many times combinations of $e$, $\alpha$ with $|\alpha|=d-e$, and $\hat{N}\in\operatorname{SSYT}_{e}(\hat{\lambda})$ produce $\hat{M}$ in this construction.

But as $\hat{M}^{0}=\hat{N}^{0}$, the finite part of $\hat{N}$ is determined, and we also know that the multiplicity of every $h>0$ in $\hat{N}$ as in Definition \ref{tabinf} is bounded by the one in $\hat{M}$ (since the latter is obtained by adding those from $\alpha$ to those from $\hat{N}$). The condition $f_{\hat{N}}=f$ implies that we need to take a total of $f$ multiplicities among those from $\hat{M}$ in order to obtain a candidate for $\hat{N}$, and every such choice of $f$ multiplicities produces $\hat{N}$, in $\operatorname{SSYT}_{e}(\hat{\lambda})$ for some $0 \leq e \leq d$, that yields such a contribution (for the appropriate $\alpha$). As the number of choices to take such an $\hat{N}$ is $c_{\hat{M},f}$ by Definition \ref{Lambdafilt}, part $(i)$ is established.

We now note that Remark \ref{mudetf} and part $(ii)$ of Lemma \ref{difmodf-1} and imply that every basis element of $\tilde{V}_{\hat{M}}$ for $\hat{\lambda}\vdash\infty$ and $M\in\operatorname{SSYT}_{d}(\hat{\lambda})$ is congruent to the corresponding basis element of $V_{\hat{M}}$. Therefore we may view the quotient in part $(i)$ as isomorphic to $\bigoplus_{\hat{\lambda}\vdash\infty}\bigoplus_{\hat{M}\in\operatorname{SSYT}_{d}(\hat{\lambda})}V_{\hat{M}}^{c_{\hat{M},f}}$. Since the summand associated with $M$ in this direct sum is now contained in $\tilde{\Lambda}_{\mu}$ where $\mu$ is the content of $\hat{M}$ (of sum $d$), we can intersect everything with $\tilde{\Lambda}_{\mu}$, and obtain part $(ii)$, including the maximal completely reducible assertion, as above.

Finally, projecting these presentations onto the $R_{\infty,k}$, preserving the homogeneity degree, and recalling that this projection keeps precisely the isomorphic images of $\tilde{\Lambda}_{\mu}$ for which $\operatorname{SSYT}_{\mu}(\hat{\lambda})\subseteq\operatorname{SSYT}_{d}^{k}(\hat{\lambda})$ for every $\hat{\lambda}\vdash\infty$ (as in the proof of part $(v)$ of Theorem \ref{compred}) yields part $(iii)$ as well. This completes the proof of the theorem.
\end{proof}
As in Theorem \ref{filtQxinf}, the case $f=0$ of Theorem \ref{filtrations} holds unconditionally, as parts $(iii)$, $(iv)$, and $(v)$ of Theorem \ref{compred}. Indeed, Definition \ref{Lambdafilt} gives $c_{\hat{M},0}=0$, because of the single choice of the empty sub-multi-set of size 0 required for getting $\hat{N}$ with $f_{\hat{N}}=0$ in the proof of Theorem \ref{filtrations} (in fact, we have $\hat{N}=\hat{M}^{0}$), relating the descriptions of $\tilde{\Lambda}_{d}^{0}$, $\tilde{\Lambda}_{\mu}^{0}$, and $R_{\infty,k,d}^{0}$ in these two theorems. Remark \ref{chatMf} implies that only elements $\hat{M}$ with $f_{\hat{M}} \geq f$ participate in the expression for the quotients in Theorem \ref{filtrations}, so in particular these expressions are indeed non-zero if and only if $f$ is bounded by $d$ in part $(i)$ and $\ell$ in part $(ii)$.

\begin{rmk}
The proof of Theorem \ref{filtQxinf} involved taking a total of $f$ multiplicities from those of $\hat{M}$ in order to see which $\hat{N}$ constructed its contribution in the quotient $\tilde{\Lambda}_{d}^{f}\big/\tilde{\Lambda}_{d}^{f-1}$ or $\tilde{\Lambda}_{\mu}^{f}\big/\tilde{\Lambda}_{\mu}^{f-1}$. Then the monomial symmetric function $m_{\alpha}$ such that the corresponding contribution shows up in $(\mathbb{Q}[\mathbf{x}_{\infty}]_{e}^{f}\big/\mathbb{Q}[\mathbf{x}_{\infty}]_{e}^{f-1})m_{\alpha}$ (for the appropriate $0 \leq e \leq d$) is based on the complement, which is of size $\ell-f$ in case $\hat{M}$ involves $\ell$ positive multiplicities in total. Since it is more natural to follow $m_{\alpha}$ than $\hat{N}$, we replace each sub-multi-set by its complement, which in particular shows that the numbers from Definition \ref{Lambdafilt} satisfy $c_{\hat{M},f}=c_{\hat{M},\ell-f}$ for such $\ell$. This is the convention where $\tilde{\Lambda}_{d}^{0}$ has the form from Proposition \ref{decomVM0mM} (or the one from Proposition \ref{ExtVMlim}, which is also suitable for $\tilde{\Lambda}_{\mu}^{0}$, if instead of multiplying by a symmetric function from $\Lambda$ we just leave the multiplicities there for the representation in Definition \ref{infSpecht}), while that of the semi-simplification of $\mathbb{Q}[\mathbf{x}_{\infty}]_{d}$ uses only elements $\hat{M}$ with $f_{\hat{M}}=0$, as in Corollary \ref{semisimp}. \label{fcomp}
\end{rmk}

\medskip

Our last result is an analogue of Corollary \ref{semisimp}.
\begin{cor}
There are composition series for $\tilde{\Lambda}_{d}$, $\tilde{\Lambda}_{\mu}$, and $R_{\infty,k,d}$, and their semi-simplifications are isomorphic to $\bigoplus_{\hat{\lambda}\vdash\infty}\bigoplus_{\hat{M}\in\operatorname{SSYT}_{d}(\hat{\lambda})}\tilde{V}_{\hat{M}}^{c_{\hat{M}}}$ (or equivalently $\bigoplus_{\hat{\lambda}\vdash\infty}\bigoplus_{\hat{M}\in\operatorname{SSYT}_{d}(\hat{\lambda})}V_{\hat{M}}^{c_{\hat{M}}}$), $\bigoplus_{\hat{\lambda}\vdash\infty}\bigoplus_{\hat{M}\in\operatorname{SSYT}_{\mu}(\hat{\lambda})}V_{\hat{M}}^{c_{\hat{M}}}$, and the direct sum $\bigoplus_{\hat{\lambda}\vdash\infty}\bigoplus_{\hat{M}\in\operatorname{SSYT}_{d}^{k}(\hat{\lambda})}V_{\hat{M}}^{c_{\hat{M}}}$ respectively. \label{submulti}
\end{cor}

\begin{proof}
The assertions follow as in the proof of Corollary \ref{semisimp}, combined with the definition of the exponents $c_{\hat{M}}$ in Definition \ref{Lambdafilt}. This proves the corollary.
\end{proof}

\begin{rmk}
As $\mathbb{Q}[\mathbf{x}_{\infty}]_{d}$ is a sub-representation of $\tilde{\Lambda}_{d}$, and it is clear from Definitions \ref{Qxinfdf} and \ref{Lambdafilt} that $\mathbb{Q}[\mathbf{x}_{\infty}]_{d}^{f}=\tilde{\Lambda}_{d}^{f}\cap\mathbb{Q}[\mathbf{x}_{\infty}]_{d}$, we may view the sub-quotients from Proposition \ref{filtprop} and Theorem \ref{filtQxinf} as sub-representations of those from Theorem \ref{filtrations}. Similarly, the semi-simplifications from Corollary \ref{semisimp} can be seen as sub-representations of those from Corollary \ref{submulti}. All these sub-representations can be considered, via the convention from Remark \ref{fcomp}, as taking, for each $\hat{M}$, the unique summand corresponding with the empty sub-multi-set, associated with the value 1 of $c_{\hat{M},0}$, in which we take no part of the index of $m_{\hat{M}}$ to into an index of a monomial symmetric function. \label{subrep}
\end{rmk}

\begin{ex}
When $d=0$ we have $\tilde{\Lambda}_{0}=\mathbb{Q}[\mathbf{x}_{\infty}]_{0}=\mathbb{Q}$, which is $V_{000\cdots}$ with index of $f$-value 0, hence also equals $\tilde{\Lambda}_{0}^{0}$ and $\mathbb{Q}[\mathbf{x}_{\infty}]_{0}^{0}$. In the case $d=1$, the infinite irreducible standard representation $V_{\substack{000\cdots \\ 1\hphantom{11\cdots}}}$ from Example \ref{exd1} is $\mathbb{Q}[\mathbf{x}_{\infty}]_{1}^{0}$, and $\tilde{\Lambda}_{1}^{0}$ was seen there to be obtained by adding the space $V_{000\cdots}^{1}$ spanned by $e_{1}$, which is also the monomial symmetric function $m_{1}$. We have $\mathbb{Q}[\mathbf{x}_{\infty}]_{1}^{1}=\mathbb{Q}[\mathbf{x}_{\infty}]_{1}$ and $\tilde{\Lambda}_{1}^{1}=\tilde{\Lambda}_{1}$ (since this is the case $f=1=d$), and the quotient modulo $\mathbb{Q}[\mathbf{x}_{\infty}]_{1}^{0}$ and $\tilde{\Lambda}_{1}^{0}$ respectively is a trivial representation in degree 1 without the symmetric function, which we write as $V_{000\cdots}^{\bar{1}}$ to indicate that it appears at the sub-quotient of $f=1$, thus related to a tableau $\hat{M}$ with $f_{\hat{M}}=1$ but in which the sub-multi-set counted by $c_{\hat{M},0}=1$ does not contain the single element 1 from the superscript (as in Remark \ref{fcomp}). The semi-simplification of $\mathbb{Q}[\mathbf{x}_{\infty}]_{1}$ is $V_{\substack{000\cdots \\ 1\hphantom{11\cdots}}} \oplus V_{000\cdots}^{\bar{1}}$, and in order to get that of $\tilde{\Lambda}_{1}$ we add $V_{000\cdots}^{1}$. The two copies of the (trivial) representation with $f_{\hat{M}}=1$ correspond to the fact that $c_{\hat{M},0}=c_{\hat{M},1}=1$ hence $c_{\hat{M}}=2$, with the copy contributing to $c_{\hat{M},0}$ showing up in the semi-simplification of $\mathbb{Q}[\mathbf{x}_{\infty}]_{1}$ (as Remark \ref{subrep} predicts), and the other one, in which we took the 1 to the index, represents explicit multiplication by $m_{1}=e_{1}$. By replacing each variable $x_{i}$ by $x_{i}^{h}$ we obtain the same description for $\mathbb{Q}[\mathbf{x}_{\infty}]_{\mu}\subseteq\mathbb{Q}[\mathbf{x}_{\infty}]_{h}$ and $\tilde{\Lambda}_{\mu}\subseteq\tilde{\Lambda}_{h}$ where $\mu$ is the content consisting of a single element $h$ (with $\ell=1$), and the symmetric function involved is now $m_{h}=p_{h}$. \label{d01filts}
\end{ex}

\begin{ex}
For $d=2$ we have two contents, one of which consists of a single integer 2, which was described in Example \ref{d01filts}. For the other content $\mu=11$, we get $\mathbb{Q}[\mathbf{x}_{\infty}]_{\mu}^{0}=V_{\substack{000\cdots \\ 11\hphantom{1\cdots}}}$, and for $\tilde{\Lambda}_{\mu}^{0}$ we also add $V_{\substack{000\cdots \\ 1\hphantom{11\cdots}}}^{1}$ and $V_{000\cdots}^{11}$ (together with the other content it gives $\tilde{\Lambda}_{2}^{0}$ from Example \ref{exd2}), the latter being the trivial representation on multiples of $m_{\mu}=e_{2}$. The quotient $\mathbb{Q}[\mathbf{x}_{\infty}]_{\mu}^{1}\big/\mathbb{Q}[\mathbf{x}_{\infty}]_{\mu}^{0}$ is written as $V_{\substack{000\cdots \\ 1\hphantom{11\cdots}}}^{\bar{1}}$ to indicate the isomorph $V_{\substack{000\cdots \\ 1\hphantom{11\cdots}}}$ of $V_{\substack{000\cdots \\ 1\hphantom{11\cdots}}}^{1}$ inside $\mathbb{Q}[\mathbf{x}_{\infty}]$, and for $\tilde{\Lambda}_{\mu}^{1}\big/\tilde{\Lambda}_{\mu}^{0}$ we also get the representation $V_{000\cdots}^{1\bar{1}}$. The final quotient $\mathbb{Q}[\mathbf{x}_{\infty}]_{\mu}^{2}\big/\mathbb{Q}[\mathbf{x}_{\infty}]_{\mu}^{1}=\tilde{\Lambda}_{\mu}^{2}\big/\tilde{\Lambda}_{\mu}^{1}$ inside $\mathbb{Q}[\mathbf{x}_{\infty}]_{\mu}\subseteq\tilde{\Lambda}_{\mu}$ is $V_{000\cdots}^{\bar{1}\bar{1}}$, which is a trivial representation. \label{d2filts}
\end{ex}

The semi-simplification of $\mathbb{Q}[\mathbf{x}_{\infty}]_{2}$ from Corollary \ref{semisimp} equals, via Example \ref{d2filts}, to $V_{000\cdots}^{\bar{2}} \oplus V_{\substack{000\cdots \\ 2\hphantom{22\cdots}}} \oplus V_{000\cdots}^{\bar{1}\bar{1}} \oplus V_{\substack{000\cdots \\ 1\hphantom{11\cdots}}}^{\bar{1}} \oplus V_{\substack{000\cdots \\ 11\hphantom{1\cdots}}}$, which is similar to the expression for $\tilde{\Lambda}_{2}^{0}$ from Example \ref{exd2}, as they are indeed known to be isomorphic (as in Remarks \ref{fcomp} and \ref{subrep}). The semi-simplification from Corollary \ref{submulti} of $\tilde{\Lambda}_{\mu}$ from that example is $V_{000\cdots}^{11} \oplus V_{000\cdots}^{1\bar{1}} \oplus V_{000\cdots}^{\bar{1}\bar{1}} \oplus V_{\substack{000\cdots \\ 1\hphantom{11\cdots}}}^{1} \oplus V_{\substack{000\cdots \\ 1\hphantom{11\cdots}}}^{\bar{1}} \oplus V_{\substack{000\cdots \\ 11\hphantom{1\cdots}}}$, and in any case we note that the missing representation $V_{000\cdots}^{1\bar{1}}$ is the same as $V_{000\cdots}^{\bar{1}1}$, because in both we take one of the two instances of 1 for the multiplicity.

\begin{ex}
We now consider the degree $d=3$, with three contents. One, with $\ell=1$, is again dealt with in Example \ref{d01filts}. The longest one with $\ell=3=d$ exhibits the same behaviour as the one from Example \ref{d2filts}, as it only involves instances of 1. For the remaining content $\mu=21$, which satisfies $\ell=2$, the representation $\mathbb{Q}[\mathbf{x}_{\infty}]_{\mu}^{0}$ equals $V_{\substack{000\cdots \\ 12\hphantom{2\cdots}}} \oplus V_{\substack{000\cdots \\ 1\hphantom{11\cdots} \\ 2\hphantom{22\cdots}}}$, the quotient $\mathbb{Q}[\mathbf{x}_{\infty}]_{\mu}^{1}\big/\mathbb{Q}[\mathbf{x}_{\infty}]_{\mu}^{0}$ is $V_{\substack{000\cdots \\ 2\hphantom{22\cdots}}}^{\bar{1}} \oplus V_{\substack{000\cdots \\ 1\hphantom{11\cdots}}}^{\bar{2}}$, and the final quotient $\mathbb{Q}[\mathbf{x}_{\infty}]_{\mu}^{2}\big/\mathbb{Q}[\mathbf{x}_{\infty}]_{\mu}^{1}$ is just $V_{000\cdots}^{\bar{1}\bar{2}}$, which is also the expression for $\tilde{\Lambda}_{\mu}^{2}\big/\tilde{\Lambda}_{\mu}^{1}$. To get $\tilde{\Lambda}_{\mu}^{0}$ from $\mathbb{Q}[\mathbf{x}_{\infty}]_{\mu}^{0}$ we add the direct sum $V_{\substack{000\cdots \\ 2\hphantom{22\cdots}}}^{1} \oplus V_{\substack{000\cdots \\ 1\hphantom{11\cdots}}}^{2} \oplus V_{000\cdots}^{12}$, and $\tilde{\Lambda}_{\mu}^{1}\big/\tilde{\Lambda}_{\mu}^{0}$ is the middle quotient plus the two trivial representations $V_{000\cdots}^{1\bar{2}} \oplus V_{000\cdots}^{\bar{1}2}$. \label{d3filts}
\end{ex}

Note that the contents in both Examples \ref{d2filts} and \ref{d3filts} were with $\ell=2$, but as the entries in the former case were the same and those in the latter one were different, the expressions do not look the same. Indeed, if $\hat{M}$ stands for the tableau with $f=2$ in these two cases, then we have $c_{\hat{M},1}=1$ for $d=2$ and $c_{\hat{M},1}=2$ when $d=3$. We also saw in these examples the general fact that if $\mu$ has $\ell$ positive integers then for $f=\ell$ both $\mathbb{Q}[\mathbf{x}_{\infty}]_{\mu}^{\ell}\big/\mathbb{Q}[\mathbf{x}_{\infty}]_{\mu}^{\ell-1}$ and $\tilde{\Lambda}_{\mu}^{\ell}\big/\tilde{\Lambda}_{\mu}^{\ell-1}$ are the same single trivial representation. The images in the quotients $R_{\infty,k}$ amounts, as we saw above, to taking the parts with the appropriate contents.

Finally, we consider the full representations, with all the graded pieces, for which we recall the union $\operatorname{SSYT}^{k}(\hat{\lambda})$ from Definition \ref{Adinf}.
\begin{cor}
The following assertions hold:
\begin{enumerate}[$(i)$]
\item We have $\mathbb{Q}[\mathbf{x}_{n}]=\bigoplus_{\lambda \vdash n}\bigoplus_{M\in\operatorname{SSYT}(\lambda)}V_{M}$ as a representation of $S_{n}$.
\item Setting $\mathbb{Q}[\mathbf{x}_{\infty}]^{f}=\bigoplus_{d=0}^{\infty}\mathbb{Q}[\mathbf{x}_{\infty}]_{d}^{f}$ for every $f\geq-1$ yields an infinite filtration on $\mathbb{Q}[\mathbf{x}_{\infty}]$ as a representation of $S_{\mathbb{N}}$ and $S_{\infty}$, where $\mathbb{Q}[\mathbf{x}_{\infty}]^{-1}=\{0\}$ and $\mathbb{Q}[\mathbf{x}_{\infty}]^{f}/\mathbb{Q}[\mathbf{x}_{\infty}]^{f-1}$ is, for any $f\geq0$, the maximal completely reducible sub-representation of $\mathbb{Q}[\mathbf{x}_{\infty}]/\mathbb{Q}[\mathbf{x}_{\infty}]^{f-1}$ and is isomorphic to the direct sum $\bigoplus_{\hat{\lambda}\vdash\infty}\bigoplus_{\hat{M}\in\operatorname{SSYT}(\hat{\lambda}),\ f_{\hat{M}}=f}V_{\hat{M}^{0}}$.
\item With $\tilde{\Lambda}^{f}:=\bigoplus_{d=0}^{\infty}\tilde{\Lambda}_{d}^{f}$ we get again $\tilde{\Lambda}^{-1}=\{0\}$, $\tilde{\Lambda}^{f}\big/\tilde{\Lambda}^{f-1}$ is the maximal completely reducible sub-representation of $\tilde{\Lambda}\big/\tilde{\Lambda}^{f-1}$, and its isomorphism class is that of $\bigoplus_{\hat{\lambda}\vdash\infty}\bigoplus_{\hat{M}\in\operatorname{SSYT}(\hat{\lambda})}V_{\hat{M}}^{c_{\hat{M}}}$.
\item The image $R_{\infty,k}^{f}$ of $\tilde{\Lambda}^{f}$ inside $R_{\infty,k}$ is $\bigoplus_{d=0}^{\infty}R_{\infty,k,d}^{f}$. It satisfies the equality $R_{\infty,k}^{-1}=\{0\}$, and for $f\geq0$ the quotient $R_{\infty,k}^{f}/R_{\infty,k}^{f-1}$ is isomorphic to $\bigoplus_{\hat{\lambda}\vdash\infty}\bigoplus_{\hat{M}\in\operatorname{SSYT}^{k}(\hat{\lambda})}V_{\hat{M}}^{c_{\hat{M}}}$, and it is the maximal completely reducible sub-representation inside $R_{\infty,k}/R_{\infty,k}^{f-1}$.
\item The semi-simplifications of the representations $\mathbb{Q}[\mathbf{x}_{\infty}]$, $\tilde{\Lambda}$, and $R_{\infty,k}$ are isomorphic, respectively, to the three direct sums $\bigoplus_{\hat{\lambda}\vdash\infty}\bigoplus_{\hat{M}\in\operatorname{SSYT}(\hat{\lambda})}V_{\hat{M}}$, $\bigoplus_{\hat{\lambda}\vdash\infty}\bigoplus_{\hat{M}\in\operatorname{SSYT}(\hat{\lambda})}V_{\hat{M}}^{c_{\hat{M}}}$, and $\bigoplus_{\hat{\lambda}\vdash\infty}\bigoplus_{\hat{M}\in\operatorname{SSYT}^{k}(\hat{\lambda})}V_{\hat{M}}^{c_{\hat{M}}}$.
\end{enumerate} \label{totdecom}
\end{cor}

\begin{proof}
Part $(i)$ is the direct sum of Theorem \ref{FMTdecom}. Part $(ii)$ arises in this way from Proposition \ref{filtprop} and Theorem \ref{filtQxinf}. Parts $(iii)$ are, in the same manner, consequences of Theorem \ref{filtrations}, and part $(v)$ follows from Corollaries \ref{semisimp} and \ref{submulti}. This proves the corollary.
\end{proof}
As always, parts $(i)$ and $(iii)$ of Corollary \ref{totdecom} admit analogues with direct sums of the representations $\tilde{V}_{\hat{M}}$ from Definition \ref{repswithmon}. Comparing it with Corollaries \ref{semisimp} and \ref{submulti}, we no longer have composition series, as the direct sums in part $(iv)$ there are infinite.

\begin{rmk}
We saw in Theorem \ref{decomRinf} that $R_{\infty,k}^{0}$ decomposes as the direct sum over multi-sets $I$ of size $k-1$ of the representations $R_{\infty,I}^{0}$ or $R_{\infty,I}^{\mathrm{hom},0}$ from Definition \ref{RinfIdef}, with the latter preserving the homogeneous parts. One may ask whether such decompositions exist for $R_{\infty,k}^{f}$ in case $f>0$, namely is there a natural construction of representations $R_{\infty,I}^{f}$ or $R_{\infty,I}^{\mathrm{hom},f}$ for such multi-sets $I$ such that $R_{\infty,k}^{f}$ will be their direct sum, or at least the quotient $R_{\infty,k}^{f}/R_{\infty,k}^{f-1}$ will be the direct sum of the quotients $R_{\infty,I}^{f}/R_{\infty,I}^{f-1}$ or $R_{\infty,I}^{\mathrm{hom},f}/R_{\infty,I}^{\mathrm{hom},f-1}$. However, the multipliers showing up in that definitions are of elementary symmetric functions, while those arising from $V_{\hat{C}}$ with $f_{\hat{C}}>0$ are still closer, via Proposition \ref{decomVM0mM}, to those from Definition \ref{repswithmon}, in which the multipliers from $\Lambda$ are monomial symmetric functions. The mixing of these two types of multipliers, which diverts from our constructions (including those mentioned in Remark \ref{otherdecoms}), makes the determination of such a construction significantly more difficult. \label{RinfIgen}
\end{rmk}

We conclude by commenting that the complete reducibility in the finite group case seems to be related to the fact that symmetric polynomials mix with the representations with $f=0$ from Lemma \ref{sameiota}, via products and direct sums, to give the full representation. This is so, since in the infinite case the generators of $\Lambda$ and those of $\mathbb{Q}[\mathbf{x}_{\infty}]$ are algebraically independent (by part $(i)$ Proposition \ref{samereps}), and indeed these two parts become separated at the limit, as we saw in Examples \ref{exd1}, \ref{exd2}, \ref{d01filts}, \ref{d2filts}, \ref{d3filts}, and others. This separation is perhaps related to taking limits in normalizations like the one from Remark \ref{normsymfunc} (or another variation of the ones from \cite{[Z1]}), since they are based on more and more variables but with coefficients that tend to 0. We leave this question, as well as a proof of Conjecture \ref{polsincn} or other unconditional proofs of Proposition \ref{filtprop} and Theorems \ref{filtQxinf} and \ref{filtrations}, or finding a decomposition as in Remark \ref{RinfIgen}, for further research.

\noindent\textsc{Einstein Institute of Mathematics, the Hebrew University of Jerusalem, Edmund Safra Campus, Jerusalem 91904, Israel}

\noindent E-mail address: zemels@math.huji.ac.il

\end{document}